\documentclass[10pt]{amsart}
\usepackage{amsmath}
\usepackage{amssymb}
\usepackage{amsthm}
\usepackage{amscd}
\usepackage{amsfonts}
\usepackage{graphicx}
\usepackage{fancyhdr}
\usepackage[utf8]{inputenc}
\usepackage{mathrsfs}
\usepackage{tikz-cd}
\usepackage{hyperref}
\numberwithin{equation}{section}

\def \a{{\bf a}}
\def\nmin{(\frakn^{+}\text{-min})}
\def\nimpr{(\frakn^{+}\text{-DT})}

\def\ur{\text{ur}}
\def\PO{\text{PO}}

\def\rig{\rm{rig}}
\def\cl{{\rm cl}}
\def\d{{\bf d}}

\def\fin{\mathrm{fin}}
\def\fin{{\rm fin}}
\def\ord{{\rm ord}}
\def\sing{{\rm sing}}

\def\rk{\text{rank }}
\def\ord{\text{ord }}
\def\char{\text{char}}
\def\res{{\text{res}}}

\usepackage{amsmath}
\usepackage{amscd,amsthm,amssymb,amsfonts}
\usepackage{mathrsfs}
\usepackage{dsfont}
\usepackage{stmaryrd}
\usepackage{euscript}
\usepackage{expdlist}
\usepackage{enumerate}
\usepackage{enumitem}
\input xy
\xyoption{all}
\usepackage[OT2,T1]{fontenc}

\theoremstyle{definition}
\newtheorem{thm}{Theorem}[section]

\newtheorem{assumption}[thm]{Assumption}

\newtheorem*{thm*}{Theorem}

\newtheorem{lm}[thm]{Lemma}
\newtheorem{cor}[thm]{Corollary}
\newtheorem*{cor*}{Corollary}
\newtheorem{prop}[thm]{Proposition}
\newtheorem*{conj*}{Conjecture}
\newtheorem{conj}{Conjecture}

\newtheorem{I_thm}{Theorem} % Used in Introduction
\newtheorem{I_cor}{Corollary} % Used in Introduction

\theoremstyle{remark}
\newtheorem*{remark}{Remark}
\theoremstyle{definition}
\newtheorem*{defn*}{Definition}

\newtheorem{I_Remark*}{Remark}
\newtheorem{defn}[thm]{Definition}

\newtheorem*{hyp}{Hypothesis }
\newcommand{\nc}{\newcommand}

\newcommand{\beq}{\begin{equation}}
\newcommand{\eeq}{\end{equation}}
\newcommand{\bpmx}{\begin{pmatrix}}
\newcommand{\epmx}{\end{pmatrix}}
\newcommand{\bbmx}{\begin{bmatrix}}
\newcommand{\ebmx}{\end{bmatrix}}

\newcommand{\beqcd}[1]{\begin{equation*}\label{#1}\tag{#1}}
\newcommand{\eeqcd}{\end{equation*}}

\numberwithin{equation}{section}
\newenvironment{mylist}{
  \begin{enumerate}[label=(\arabic*)]{}{%
      \setlength{\itemsep}{5pt} \setlength{\parsep}{0in}
      \setlength{\parskip}{0in} \setlength{\topsep}{0in}
      \setlength{\partopsep}{0in}
      \setlength{\leftmargin}{0.17in}}}{\end{enumerate}}

\def\parref#1{\ref{#1}}
\def\thmref#1{Theorem~\parref{#1}}

\def\propref#1{Proposition~\parref{#1}}
\def\corref#1{Corollary~\parref{#1}}

\def\lmref#1{Lemma~\parref{#1}}

\def\defref#1{Definition~\parref{#1}}
\def\assref#1{Assumption~\parref{#1}}

\def\makeop#1{\expandafter\def\csname#1\endcsname
  {\mathop{\rm #1}\nolimits}\ignorespaces}
\makeop{Hom}   \makeop{End}   \makeop{Aut}   %\makeop{Isom}
\makeop{Pic} \makeop{Gal}       \makeop{Div} \makeop{Lie}
\makeop{PGL}   \makeop{Corr} \makeop{PSL} \makeop{sgn} \makeop{Spf}
 \makeop{Tr} \makeop{Nr} \makeop{Fr} \makeop{disc}
\makeop{Proj} \makeop{supp} \makeop{ker}   \makeop{Im} \makeop{dom}
\makeop{coker} \makeop{Stab} \makeop{SO} \makeop{SL} \makeop{SL}
\makeop{Cl}    \makeop{cond} \makeop{Br} \makeop{inv} \makeop{rank}
\makeop{id}    \makeop{Fil} \makeop{Frac}  \makeop{GL} \makeop{SU}
\makeop{Trd}   \makeop{Sp} \makeop{Tr}    \makeop{Trd} \makeop{Res}
\makeop{ind} \makeop{depth} \makeop{Tr} \makeop{st} \makeop{Ad}
\makeop{Int} \makeop{tr}    \makeop{Sym} \makeop{can} \makeop{SO}
\makeop{torsion} \makeop{GSp} \makeop{Tor}\makeop{Ker} \makeop{rec}
\makeop{Ind} \makeop{Coker}
 \makeop{vol} \makeop{Ext} \makeop{gr} \makeop{ad}
 \makeop{Gr}\makeop{corank} \makeop{Ann}
\makeop{Hol} %Holomorphic
\makeop{Fitt} \makeop{Mp} \makeop{CAP}
\makeop{Char}

\def\Sel{\text{Sel}}
\def\hatSel{\widehat{\text{Sel}}}
\def\ord{\text{ord}}

\DeclareMathOperator{\length}{length}
\DeclareMathAlphabet{\mathpzc}{OT1}{pzc}{m}{it}
\DeclareSymbolFont{cyrletters}{OT2}{wncyr}{m}{n}
\DeclareMathSymbol{\SHA}{\mathalpha}{cyrletters}{"58}

\def\makebb#1{\expandafter\def
  \csname bb#1\endcsname{{\mathbb{#1}}}\ignorespaces}
\def\makebf#1{\expandafter\def\csname bf#1\endcsname{{\bf
      #1}}\ignorespaces}
\def\makegr#1{\expandafter\def
  \csname gr#1\endcsname{{\mathfrak{#1}}}\ignorespaces}
\def\makescr#1{\expandafter\def
  \csname scr#1\endcsname{{\EuScript{#1}}}\ignorespaces}
\def\makecal#1{\expandafter\def\csname cal#1\endcsname{{\mathcal
      #1}}\ignorespaces}
\def\doLetters#1{#1A #1B #1C #1D #1E #1F #1G #1H #1I #1J #1K #1L #1M
                 #1N #1O #1P #1Q #1R #1S #1T #1U #1V #1W #1X #1Y #1Z}
\def\doletters#1{#1a #1b #1c #1d #1e #1f #1g #1h #1i #1j #1k #1l #1m
                 #1n #1o #1p #1q #1r #1s #1t #1u #1v #1w #1x #1y #1z}
\doLetters\makebb   \doLetters\makecal  \doLetters\makebf
\doLetters\makescr
\doletters\makebf   \doLetters\makegr   \doletters\makegr

\normalsize

\makeop{Ram} \makeop{Rep} \makeop{mass}

\makeop{Bl}

\def\Qbarp{\C_p}

\def\Zp{\Z_p}

\def\cA{{\mathcal A}}  %automorphic forms

\def\cD{\mathcal D}
\def\cE{{\mathcal E}}
\def\cF{{\mathcal F}}  %Hida family
\def\cG{{\mathcal G}}
\def\cL{{\mathcal L}}
\def\cH{{\mathcal H}}
\def\cI{\mathcal I}
\def\cJ{\mathcal J}
\def\cK{{\mathcal K}}  %imaginary quadratic field
\def\cM{\mathcal M}
\def\cR{{\mathcal R}}
\def\cO{\mathcal O}

\def\cf{{\mathcal f}}

\def\cX{\mathcal X}
\def\cV{{\mathcal V}}

\def\cJ{\mathcal J}

\def\cT{\mathcal T}
\def\cY{\mathcal Y}

\def\cU{\mathcal U}

\def\EucH{{\EuScript H}}

\def\EucU{\EuScript U}

\def\ap{\bfa_{\frakp}(f)}

\def\bda{\mathbf a}

\def\bftheta{\boldsymbol{\theta}}

\newcommand{\Z}{\mathbf Z}
\newcommand{\Q}{\mathbf Q}
\newcommand{\R}{\mathbf R}
\newcommand{\C}{\mathbf C}
\newcommand{\A}{\mathbf A}    % for adele
\newcommand{\F}{\mathbb F}

\def\frakc{{\mathfrak c}}

\def\frakp{{\mathfrak p}}
\def\frakP{\mathfak P}
\def\frakq{\mathfrak q}

\def\frakm{\mathfrak m}
\def\frakn{\mathfrak n}
\def\frakl{\mathfrak l}
\def\frakP{\mathfrak P}

\def\frakL{{\mathfrak L}}

\def\frakN{\mathfrak N}

\def\Hhat{\widehat{H}}

\def\Neron{N\'{e}ron }
\def\Frob{\mathrm{Frob}}
\newcommand{\<}{\langle}   %\< is not defined yet.
\renewcommand{\>}{\rangle} %\> is already defined.

\def\imply{\Rightarrow}

  \nc{\opp}{\mathrm{opp}} \nc{\ul}{\underline}

\def\CR{\text{CR}^{+}}
\def\nminus{\frakn^{-}}
\def\nplus{\frakn^{+}}
\def\Bhat{\widehat{B}}
\def\cross{\times}
\def\Bcross{B^{\cross}}
\def\Bhatcross{\Bhat^{\cross}}
\def\openU{\EucU}
\def\openUprime{\openU^{\prime}}
\def\isomor{\simeq}
\def\Fcross{F^{\cross}}
\def\Fhat{\widehat{F}}
\def\Fhatcross{\widehat{F}^{\cross}}

\def\Rhat{\widehat{R}}
\def\Khat{\widehat{K}}
\def\Khatcross{\Khat^{\cross}}
\def\mplus{\frakm^{+}}
\def\Ohat{\widehat{\cO}_{F}}
\def\Hecke{\bbT}
\newcommand{\away}[1]{(#1)}
\newcommand{\SHecke}[2]{\Hecke_{B}^{\away{#1}}(#2)}
\newcommand{\CHecke}[2]{\Hecke_{B}(#1,#2)}
\def\Bprime{B^{\prime}}
\def\Bprimecross{B^{\prime\cross}}
\def\Bprimehatcross{\widehat{B}^{\prime\cross}}
\def\ShiUprime{M_{\openU^{\prime}}}
\def\ShiUprimeY{M_{\openU^{\prime}Y}}
\def\Shiln{M^{[\frakl]}_{n}}
\def\Bruhat{\cT_{\frakl}}
\def\vBruhat{\cV(\Bruhat)}
\def\oeBruhat{\vec{\cE}(\Bruhat)}
\def\Pic{\text{Pic}^{0}}
\newcommand{\Nm}[1]{{\bf{N}}(#1)}
\def\LnplusY{\frakl\nplus Y}
\def\Heckepara{\bf{a}}
\def\Satakepara{\alpha}
\def\Jnl{J^{[\frakl]}_{n}}
\def\m{\frakm}
\def\ml{\frakm^{[\frakl]}}

\def\ILg{\cI^{[\frakl]}_{g}}

\def\chagroup{\cX^{[\frakl]}}
\def\Fll{F_{\frakl^{2}}}
\def\Kum{\text{Kum}}
\def\compo{\Phi^{[\frakl]}}
\def\Nplus{\frakN^{+}}
\def\Oxi{\cO_{\xi}}
\def\uxi{\uf_{\xi}}

\def\Dnton{\cD_{n}}

\def\cSel{\widehat{\Sel}}
\def\Sxi{\mathfrak{S}^{\xi}}
\def\uxi{\uf_{\xi}}
\def\tka{\tilde{\kappa}}

\def\pkaDxil{\kappa^{\prime}_{\cD_{n+t},\xi}(\frakl)}
\def\kaDntonxil{\kappa_{\Dnton,\xi}(\frakl)}
\def\pkaDxi{\kappa^{\prime}_{\cD_{n+t},\xi}}
\def\kaDxi{\kappa_{\cD_{n+t},\xi}}
\def\oxi{\ord_{\uxi}}

\def\Dpp{\cD^{\prime\prime}}
\def\f{{\bf{f}}}
\def\Dfn{\cD^{f}_{n}}

\def\XYmatrix{\xymatrix@M=8pt} % make \xymatrix not too cluttered
\def\ncmd{\newcommand}
\ncmd{\xysubset}[1][r]{\ar@<-2.5pt>@{^(-}[#1]\ar@<2.5pt>@{_(-}[#1]}
\ncmd{\XYmatrixc}[1]{\vcenter{\XYmatrix{#1}}}
\ncmd{\xyto}[1][r]{\ar@{->}[#1]}
\ncmd{\xyinj}[1][r]{\ar@{^(->}[#1]}
\ncmd{\xysurj}[1][r]{\ar@{->>}[#1]}
\ncmd{\xyline}[1][r]{\ar@{-}[#1]}
\ncmd{\xydotsto}[1][r]{\ar@{.>}[#1]}
\ncmd{\xydots}[1][r]{\ar@{.}[#1]}
\ncmd{\xyleadsto}[1][r]{\ar@{~>}[#1]}
\ncmd{\xyeq}[1][r]{\ar@{=}[#1]} \ncmd{\xyequal}[1][r]{\ar@{=}[#1]}
\ncmd{\xyequals}[1][r]{\ar@{=}[#1]}
\ncmd{\xymapsto}[1][r]{l\ar@{|->}[#1]}\ncmd{\xyimplies}[1][r]{\ar@{=>}[#1]}
\ncmd{\xyiso}{\ar[r]_-{\sim}}
\def\injxy{\ar@{^(->}}

\newcommand{\pMX}[4]{\begin{pmatrix}
{#1}& {#2}\\
{#3}&{#4}\end{pmatrix} }

\newcommand{\seesaw}[4]{{#1}\ar@{-}[rd]\ar@{-}[d]&{#2}\ar@{-}[d]\\
{#3}\ar@{-}[ru]&{#4}}

\newcommand\dual[1]{{#1}^{\!\vee}} %Dual
\def\ie{i.e. }

\def\cf{\mbox{{\it cf.} }}

\def\resp{resp.$\,$}

\def\mapR{\smash{\mathop{\longrightarrow}}}

\def\MapRR#1{\xrightarrow{\hspace*{0.2cm}#1\hspace*{0.2cm}}}
\def\MapR#1#2{\smash{\mathop{\longrightarrow}\limits^{#1}_{#2}}}
\newcommand{\exact}[3]{
0\mapR{#1}\mapR{#2}\mapR{#3}\mapR 0 }
\newcommand{\lexact}[3]{
0\mapR{#1}\mapR{#2}\mapR{#3} }

\def\uf{\varpi} %uniformizer
\def\kap{\kappa}

\def\dirlim{\varinjlim}
\def\prolim{\varprojlim}

\renewcommand\pmod[1]{\,(\mbox{mod }{#1})}
\renewcommand\mod[1]{\,\mbox{mod }{#1}}

\usepackage[top=1.3in, bottom=1.3in, left=1.3in, right=1.3in]{geometry}
\author{Haining Wang}
\address{\parbox{\linewidth} { Department of Mathematics,\\ McGill University,\\ 805 Sherbrooke St W,\\ Montreal, QC H3A 0B9, Canada.~ }}
\email{wanghaining1121@outlook.com}
\date{\today}

\begin{document}
\title[Anticyclotomic Iwasawa theory for Hilbert modular forms]{On the anticyclotomic Iwasawa main conjecture for Hilbert modular forms of parallel weights}

\begin{abstract}
In this article, we study the Iwasawa theory for Hilbert modular forms over the anticyclotomic extension of a CM field. We prove a one sided divisibility result toward the Iwasawa main conjecture. The proof relies on the first and second reciprocity law relating theta elements to Heegner points Euler system. As a by-product we also prove certain Bloch-Kato type result in the rank $0$ case and a parity conjecture.
\end{abstract}

\maketitle

\tableofcontents

\section{Introduction}
\subsection{Main results} Let $F$ be a totally real field over $\Q$ of degree $d$ with ring of integers $\cO_{f}$ and $\frakn$ be a non-zero integral ideal of $\cO_{F}$. We write $F_{\infty}=\prod_{v\mid \infty} F_{v}$. Let $$\Gamma_{0}(\frakn)=\{\pMX{a}{b}{c}{d}\in \GL_{2}(\hat{\cO}_{F}): c\equiv 0 \mod \frakn\}.$$ We take $f\in S_{k}(\Gamma_{0}(\frakn),1)$ to be a Hilbert cusp form of level $\frakn$, parallel  even weight $k$ for $k\geq 2$ and trivial central character. We recall the definition of Hilbert modular forms following \cite{Hida:p-adic-Hecke-algebra}.
For a complex function $f: \GL_{2}(\A_{F})\rightarrow \C$ and an element in $u\in \GL_{2}(\A_{F})$, we define
$$f\mid_{k}u(x)= j(u_{\infty}, z_{0})^{-k}\det(u_{\infty})^{(k-1)}f(xu^{-1})$$
for $x\in \GL_{2}(\A_{F})$, $z_{0}=(i, \cdots, i)\in \cH^{d}$ and $$j: \GL_{2}(F_{\infty})\times \cH^{d}\rightarrow \C, (\pMX{a_{v}}{b_{v}}{c_{v}}{d_{v}})_{v\mid \infty}\times (z_{v})_{v\mid \infty}\rightarrow \prod_{v\mid \infty} (c_{v}z_{v}+d_{v})_{v\mid \infty}.$$
Then $f$ belongs to the space $S_{k}(\Gamma_{0}(\frakn),1)$ of Hilbert cusp forms of parallel weight $k$ and trivial central character if
\begin{enumerate}
\item $f\mid_{k} u= f$ for all $u\in \Gamma_{0}(\frakn)C_{\infty}$ with $C_{\infty}=F_{\infty}\SO_{2}(F_{\infty})$.
\item For $\gamma\in \GL_{2}(F)$ and $z\in Z$ the center of $\GL_{2}(\A_{F})$, $f(\gamma z g)=f(g)$.
\item For $x\in \GL_{2}(\A^{f}_{F})$, the function $f_{x}: \cH^{d}\rightarrow \C, uz_{0}\rightarrow j(u, z_{0})^{k}\det(u)^{1-k}f(xu)$ is holomorphic where $u$ is any element of $\GL_{2}(F_{\infty})$.
\item For all $x\in \GL_{2}(\A_{F})$, $\int_{F\backslash \A_{F}}f(\pMX{1}{a}{0}{1}x) d(a)=0$ with a Harr measure $d(a)$ on $F\backslash \A_{F}$.
\end{enumerate}

We denote the space of Hilbert modular cusp forms of weight $k$, level $\Gamma_{0}(\frakn)$, with trivial central character by $S_{k}(\Gamma_{0}(\frakn), 1)$. The space $S_{k}(\Gamma_{0}(\frakn), 1)$ is acted on by the algebra of Hecke operators $\Hecke_{0}(\frakn)$. A typical element in this 
algebra acts on a Hilbert modular form by the double coset $$[\Gamma_{0}(\frakn)\pMX{\uf_{v}}{0}{0}{1}\Gamma_{0}(\frakn)]$$ for a place $v$ in $F$.

We assume that $f$ is a new(primitive) form in the sense of  Shimura \cite{Shimura:Hilbert1978}. 
%add reference to what is a new Hilbert modular form
Let $K$ be a totally imaginary quadratic extension over  $F$ with relative discriminant $D_{K/F}$. Then $K$ determines a decomposition of $\frakn$ into two integral ideals: $\frakn=\frakn^{+}\frakn^{-}$ where $\frakn^{+}$ is only divisible by primes split in $K$ and $\frakn^{-}$ is only divisible by primes that are non-split. In this article we assume that
\beqcd{ST}
\begin{aligned}
\text{ the number of prime divisors of $\nminus$ has the same parity as } d.
\end{aligned}
\eeqcd
The assumption \eqref{ST} puts us in the situation that 
\begin{enumerate}
\item the Hilbert modular form comes from a modular form on a totally definite quaternion algebra via the Jacquet-Langlands correspondence; 
\item  the sign of the functional equation of the Rankin Selberg $L$-function $L(f/K,\chi, s)=L(f\otimes \theta_{\chi}, s)$ of $f$ and $\theta_{\chi}$, the theta series attached to an anticyclotomic character $\chi$, is $+1$ \cf \cite[Remark 1.2]{Vatsal_Cornut:Non_Van}. 
\end{enumerate}
Let $p$ be an odd prime. We fix an embedding $\iota_{p}:\overline{\Q}\rightarrow\Qbarp$. Let $E_{f}$ be the $p$-adic Hecke field for the Hilbert modular form $f$ i.e. the field over $\Q_{p}$ generated by the Hecke eigenvalues of $f$ induced by $\iota_{p}$. We let $\cO_{f}$ be the ring of integers of $E_{f}$ and $\uf$ be a uniformizer of $\cO_{f}$.  Denote by $\bfa_{v}(f)$ the Hecke eigenvalue of $f$ at the prime $v$. We assume 
the following ordinary assumption at $p$ for $f$:
\beqcd{ORD}
\begin{aligned}
\bfa_{v}(f) \text{ is a unit in } \cO_{f} \text{ for all } v\mid p.
\end{aligned}
\eeqcd

%Galois representation intro
Let $\rho_{f}:\Gal(\bar{F}/F)\rightarrow GL_{2}(E_{f}) $ be the $p$-adic  Galois representation associated to $f$ \cf\cite{Wiles:ordinary_hilbert}, \cite{Taylor:Galois_rep_HMF}. We know that $\det \rho_{f}=\epsilon^{k-1}$ where $\epsilon$ is the $p$-adic cyclotomic character of $G_{F}=\Gal(\bar{F}/F)$, and that if $\frakl$ is a prime not dividing $p \frakn$, then the characteristic polynomial of $\rho_{f}(\Frob_{\frakl})$ is given by $$x^{2}-{\bf{a}}_{\frakl}(f) x+ {\Nm\frakl}^{k-1},$$ here $\Frob_{\frakl}$ is the Frobenius element at $\frakl$. In this article we will consider the selfdual twist of $\rho_{f}$, namely $\rho_{f}^{*}:=\rho_{f}\otimes \epsilon^{\frac{2-k}{2}}$. Let $V_{f}$ be the underlying representation space for $\rho_{f}^{*}$ and $T_{f}$ be a $G_{F}$ stable $\cO_{f}$ lattice in $V_{f}$ and $A_{f}={V_{f}}/{T_{f}}$.  By (ORD) and a result of Wiles \cite[Theorem 2]{Wiles:ordinary_hilbert}, for $v$ a place dividing $p$ in $F$,  the restriction of $\rho_{f}^{*}$ at $G_{F_{v}}$ is of the form
$$\rho^{*}_{f}\mid_{ G_{F_{v}}}=\left(
                                   \begin{array}{cc}
                                     \chi^{-1}_{v}\epsilon^{\frac{k}{2}} & * \\
                                                      0                &  \chi_{v}\epsilon^{\frac{2-k}{2}}\\
                                   \end{array}
                                 \right)$$ 
where $\chi_{v}$ is an unramified character of $G_{F_{v}}$ such that $\chi_{v}(\Frob_{v})=\alpha_{v}$ with $\alpha_{v}$ the unit root of the Hecke polynomial $x^{2}-\bfa_{v}(f)x+\Nm{v}^{k-1}$.This induces an exact sequence of representations of $G_{F_{v}}$
$$\exact{F^{+}_{v}V_{f}}{V_{f}}{F^{-}_{v}V_{f}}$$ 
where $G_{F_{v}}$ acts on $F^{+}_{v}V_{f}$ (\resp $F^{-}_{v}V_{f}$) via $\epsilon^{\frac{k}{2}}\chi^{-1}_{v}$ (\resp $\chi_{v}\epsilon^{\frac{2-k}{2}}$). Define $F^{+}_{v}A_{f}$ to be the image of $F^{+}_{v}V_{f}$ in $A_{f}$. 

Let $\frakp$ be a fixed prime in $F$ over a rational prime $p$ such that $\frakp\nmid \frakn D_{K/F}$. Let $H_{m}$ be the ring class field of conductor $\frakp^{m}$ and let $G_{m}=\Gal(H_{m}/k)$ be the ring class group. The group $G_{m}$ corresponds to $K^{\cross}\backslash\hat{K}^{\cross}/ \hat{\cO}^{\cross}_{\frakp^{m}}\hat{F}^{\cross}$ via class field theory and here $\cO_{\frakp^{m}}$ is the order in $K$ of conductor $\frakp^{m}$. Our convention of ring class field follows that of \cite[Section 1.1]{Nekovar_CV}. Let $d_{\frakp}=[F_{\frakp}:\Q_{p}]$ be the inertia degree of  $\frakp$ and we put $K_{\infty}$ to be the maximal $\Z_{p}^{d_{\frakp}}$-extension of $K$ in the tower of ring class fields of $K$ with $\frakp$-power conductor. Let $K_{m}$ be the $m$-th layer of $K_{\infty}$. Let $\Gamma=\Gal(K_{\infty}/K)$ and let $\Lambda=\cO_{f}[[\Gamma]]$ be the $d_{\frakp}$ variable Iwasawa algebra over $\cO_{f}$. We also put $\Gamma_{m}=\Gal(K_{m}/K)$. In this article we study the structure of the minimal Selmer group
$$\Sel(K_{\infty},A_{f}):=\ker\{H^{1}(K_{\infty}, A_{f})\rightarrow \prod_{v\nmid p }H^{1}(K_{\infty,{v}}, A_{f})\times \prod_{v\mid p}H^{1}(K_{\infty, v}, A_{f}/F^{+}_{\frakp}A_{f})\} .$$

On the analytic side, based on the work of Chida-Hsieh \cite{Chida_Hsieh}, Hung  \cite{Hung:thesis} has constructed an element $\theta_{\infty}(f)\in \Lambda$ which we will refer to as the theta element such that for each finite character $\chi: \Gamma\mapR \mu_{p^{\infty}}$ of conductor $\frakp^{s}$, it satisfies
$$\chi({\theta_{\infty}(f)^{2}})=\Gamma(\frac{k}{2})^{2}u^{2}_{K}\frac{\sqrt{D_{K}}D^{k-1}_{K}}{D_{F}^{\frac{3}{2}}}\Nm{\frakp}^{s(k-1)}\chi(\frakN^{+})\epsilon_{\frakp}(f)\ap^{-s}e_{\frakp}(f,\chi)^{2}\frac{L(f/K,\chi, k/2)}{\Omega_{f,\frakn^{-}}}.$$
Here 
\begin{enumerate}
\item $D_{F}$ (\resp $D_{K}$) is the absolute discriminant of  $F$(\resp $K$). 
\item We fix a decomposition $\frakn^{+}=\frakN^{+}{\overline{\frakN^{+}}}$ where ${\overline{\frakN^{+}}}$ is the complex conjugation of $\frakN^{+}$.
\item $\epsilon_{\frakp}(f)$ is the local root number of $f$ at $\frakp$. 
\item $u_{K}=[\cO^{\times}_{K}:\cO^{\times}_{F}]$.
\item $e_{\frakp}(f,\chi)$ is the $p$-adic multiplier defined as follows
 \[e_{\mathfrak{p}}(f,\chi)=\begin{cases}
1 & \hbox{if $\chi$ is ramified;}\\
(1-\alpha_{\mathfrak{p}}^{-1}\chi(\mathfrak{P}))(1-\alpha_{\mathfrak{p}}^{-1}\chi(\overline{\mathfrak{P}})) & \hbox{if $s=0$ and $\mathfrak{p}=\mathfrak{P\overline{P}}$ is split;}\\
1-\alpha_{\mathfrak{p}}^{-2} & \hbox{if $s=0$ and $\mathfrak{p}$ is inert;}\\
1-\alpha_{\mathfrak{p}}^{-1}\chi(\mathfrak{P}) & \hbox{if $s=0$\hbox{ and }$\mathfrak{p}=\mathfrak{P}^{2}$ is ramified.}
\end{cases}\]
\item $\Omega_{f,\frakn^{-}}$ is a complex period associated to $f$. It is known as the Gross period and is given by \cite[(5.2)]{Hung:thesis}.  
\end{enumerate}

In this article we prove half of the anticyclotomic Iwasawa main conjecture for Hilbert modular forms under similar hypothesis at $p$ as in \cite{Chida_Hsieh:main}. Let $\frakn_{\bar{\rho}}$ be the Artin conductor of the residual Galois representation $\bar{\rho}_{f}$.

\begin{hyp}[$\CR$]\label{CR}\hfill
\begin{enumerate}
\item$p>k+1$ and $\#(\F^{\times}_{\frakp})^{k-1}>5$;
\item The restriction of $\bar{\rho}_{f}$ to $G_{F(\sqrt{p^{*}})}$ is irreducible where $p^{*}=(-1)^{\frac{p-1}{2}}p$;
\item$\bar{\rho}_{f}$ is ramified at $\frakl$ if $\frakl\mid\frakn^{-}$ and ${\Nm{\frakl}^{2}} \equiv 1 \pmod{p}$;
\item$\frac{\frakn}{\frakn_{\bar{\rho}}}$ is coprime to $\frakn_{\bar{\rho}}$.
\end{enumerate}
\end{hyp}
For primes dividing $\frakn^{+}$, we assume that
\begin{hyp}[$\frakn^{+}\text{-DT}$]
$\text{if $\frakl\mid\mid \nplus$ and $\Nm{\frakl}\equiv1\pmod{p}$, then $\bar{\rho}_{f}$ is ramified.}$
\end{hyp}

At places above $p$, we impose the condition
\begin{hyp}[$\PO$]
$
\a_{v}^{2}(f)\not\equiv 1(\mod p) \text{ for all $v\mid p$ if $k=2$  }.
$
\end{hyp}

In order to prove the second reciprocity law, we assume that the so-called "Ihara's Lemma" is true in our setting. We refer the reader to \assref{Ihara} for the precise statement. 

We put $L_{p}(K_{\infty}, f)=\theta_{\infty}(f)^{2}$ as an element in $\Lambda$.  It is well known that the Selmer group $\Sel(K_{\infty}, A_{f})$ is a cofinitely generated $\Lambda$-module and let $\text{char}_{\Lambda}  \Sel(K_{\infty}, A_{f})^{\lor}$ be the characteristic ideal of its Pontryagain dual.  The Iwasawa main conjecture for Hilbert modular forms over the anticyclotomic extension $K_{\infty}$ predicts the following
\begin{conj}
$$\text{char}_{\Lambda}  \Sel(K_{\infty}, A_{f})^{\lor} =(L_{p}(K_{\infty},f)).$$
\end{conj}
It can be seen as a family version of the Bloch-Kato conjecture which relates the size of the Selmer groups with the $L$-function. 
We prove in this article the following result (one sided divisibility) toward the Iwasawa main conjecture for Hilbert modular forms over the anticyclotomic extension.
\begin{I_thm}\label{main1}
Assume $(\text{CR}^{+})$, $(\text{PO})$, $\nimpr$ and the "Ihara's lemma" are true. Then
$$\text{char}_{\Lambda}  \Sel(K_{\infty}, A_{f})^{\lor} \mid (L_{p}(K_{\infty},f)).$$
\end{I_thm}

Recall for each finite generated module $X$ over $\Lambda$, $X$ is pseudo-isomorphic to $\Lambda^{r} \oplus\prod_{i}\Lambda/(\uf^{\mu_{i}})\oplus \prod_{j}\Lambda/(f^{m_{j}}_{j})$ for some distinguished polynomials $f_{j}$ \cite[Theorem 3.1]{Lang:cyclotomic}. The algebraic Iwasawa $\mu$-invariant $\mu$ of $X$ is defined to be $\mu=\sum_{j}\mu_{j}$. An immediate corollary of this theorem and the result of Hung \cite[Theorem B]{Hung:thesis} is the following corollary.
\begin{I_cor}
Assume the same hypothesis in \thmref{main1}. Then the $\Lambda$-module $\Sel(K_{\infty}, A_{f})$ is cotorsion with trivial algebraic $\mu$-invariant.
\end{I_cor}

The ingredients of the proof of this theorem can also be used to prove the following result in the direction of the Bloch-Kato conjecture. We refer the reader to \cite{Bloch_Kato} for more about the conjecture and for the definition of their Selmer group.
\begin{I_thm}\label{intro_BK}
Assume  $(\text{CR}^{+})$, $(\text{PO})$, and $\nimpr$ hold. Let $\chi$ be a finite order character of $\Gamma_{m}=\Gal(K_{m}/K)$ and assume $$L(f/K,\chi, k/2)\neq 0,$$ then the Bloch-Kato Selmer group $\Sel_{f}(K, V_{f,\chi})=0$.
\end{I_thm}
 
Using this result, we obtain as a corollary the proof of the parity conjecture below.
\begin{I_thm}\label{intro_par}
Assume  $(\text{CR}^{+})$, $(\text{PO})$, and $\nimpr$ hold. Let $\chi$ be as in \thmref{intro_BK}, then
$$\ord_{s=k/2}L(f/K,\chi,s)\equiv  \rank \Sel_{f}(K, V_{f,\chi}) \pmod{2}.$$
\end{I_thm}

\subsection{Strategy of the proof}
To prove this one-sided divisibility result, we follow the scheme of \cite{Bertolini_Darmon:IMC_anti} and its modification in \cite{Pollack_Weston:AMU}, \cite{Chida_Hsieh:main}. We give an account of the strategy to prove \thmref{main1} in this section and summarize the structure of this article.
\begin{itemize}
\item In section $2$ we collect some basic facts about Galois cohomology over the anticyclotomic tower and introduce various Selmer groups associated to the Galois representation $\rho^{*}_{f}$. Notably we introduce the $\Delta$-ordinary, $S$-relaxed Selmer group $\Sel_{\Delta}^{S}(L,T_{f,n})$ for some $L$ contained in $K_{\infty}$. In section $3$, we study Hilbert modular forms on a definite quaternion algebra $B$ over $F$. In section $4$, we study the Shimura curve obtained from $B$ by interchanging the invariant at an infinite place of $F$ and a place inert in $K$. In particular, we give a modular description of the component group of its special fiber at a prime of bad reduction. We will construct in section 5 an Euler system $\kappa_{\cD}(\frakl) \in \widehat{H}^{1}(K_{\infty},T_{f,n})$ indexed by $(\frakl, \cD)$. Here $\frakl$ is an $n$-admissible prime in the sense of \defref{n_admissible} and $\cD=(\Delta,g)$ consists of $\Delta=\frakn^{-}S$ where $S$ is a square free product of even number of n-admissible primes and $g$ is a weight two modular form on the definite quaternion algebra of discriminant $\Delta$ whose Hecke eignevalues are congruent to those of $f$ modulo $\uf^{n}$. The construction is based on a level raising principle initiated by Ribet \cite{Ribet:level_lowering}, see also \cite[Theorem 5.15]{Bertolini_Darmon:IMC_anti}. Roughly speaking, for each admissible form $ \cD=(\Delta, g)$ and an $n$-admissible prime $\frakl$, one can prove an isomorphism $T_{f,n}\isomor T_{p} J^{[\frakl]}_{n}/\ILg$  where $T_{p} J^{[\frakl]}_{n}$ is the $p$-adic Tate module of the Jacobian $J^{[\frakl]}_{n}$ of a Shimura curve and $\ILg$ is an ideal corresponding to $g$ in a suitable Hecke alegbra. Then the Euler system $\kappa_{\cD}(\frakl)$ is produced using the Euler system of Heegner points on $J^{[\frakl]}_{n}$.

\item Section 6 is devoted to the proof of the so called first and second reciprocity laws. The first reciprocity law relates the residue of $\kappa_{\cD}(\frakl) \in \widehat{\Sel}_{\Delta\frakl}(K_{\infty}, T_{f,n})$ at $\frakl$ to an element $\theta_{\infty}(\cD)\in \cO_{f,n}[[\Gamma]]$. This element $\theta_{\infty}(\cD)\in \cO_{f,n}[[\Gamma]]$ is closely related to the theta element $\theta_{\infty}(f)$ via congruences. For two admissible primes $\frakl_{1}$ and $\frakl_{2}$, the second reciprocity law is the assertion that there exists an $n$-admissible form $\cD^{\prime\prime}=(\Delta_{B}\frakl_{1}\frakl_{2}, g^{\prime\prime})$ such that $v_{\frakl_{1}}(\kappa_{\cD}(\frakl_{2}))=v_{\frakl_{2}}(\kappa_{\cD}(\frakl_{1}))=\theta_{\infty}(\cD^{\prime\prime})$. Here $v_{\frakl_{1}}(\kappa_{\cD}(\frakl_{2}))$ (\resp  $v_{\frakl_{2}}(\kappa_{\cD}(\frakl_{1}))$) is the image of $\kappa_{\cD}(\frakl_{2})$ (\resp   $\kappa_{\cD}(\frakl_{1})$) in the finite part of the local Galois cohomology group which is of rank $1$ over $\Lambda$. The proof of the first reciprocity law is quite straightforward while the proof of the second reciprocity law relies on the "Ihara's lemma" which we treat as \assref{Ihara}. 

\item The first reciprocity law and second reciprocity law allow one to use an induction argument to prove the important \thmref{induction}. We sketch the structure of the proof here.  Let $\xi:\Lambda\rightarrow \cO_{\xi}$ be a character of the Iwasawa algebra of $\Gamma$ over $\cO_{f}$. For any $\cO_{\xi}$ module $M$ and each $x\in M$, we denote by $\ord_{\uf_{\xi}}(x)=\sup\{m\in \Z\cup\{\infty\}\mid x\in \uxi^{m}M\}$ the divisibility index of $x$ . For each positive integer $n$ and $n$-admissible form $\cD=(\Delta, f)$, we define two integers:
\begin{equation}
\begin{aligned}
& s_{\cD}=\length_{\cO_{\xi}} \dual{\Sel_{\Delta}(K_{\infty}, A_{f,n})}\otimes_{\xi}\cO_{\xi};\\
& t_{\cD}=\ord_{\uxi} \xi(\theta_{\infty}(\cD))\text{ where $\xi(\theta_{\infty}(\cD))\in \cO_{f,n}[[\Gamma]]\otimes_{\xi}\cO_{\xi}=\cO_{\xi}/\uxi^{n}$}.\\
\end{aligned}
\end{equation}
The first one measures the length of the specialization of the Selmer group and the second one measures the divisibility of the specialization of the theta element. We would like to show $s_{\cD}\leq 2 t_{\cD}$. We do induction on the number $t_{\cD}$. It follows from the first reciprocity law that if $t_{\cD}=0$, then $s_{\cD}=0$ and we have the base case. The underlying principle in the base case is the rank $0$ case of the Bloch-Kato conjecture which we will treat in chapter 8. For the induction step, we choose a suitable $n$-admissible prime $\frakl_{1}$ and define
\begin{equation}
e_{\cD}(\frakl_{1})= \ord_{\uf_{\xi}}(\kappa_{\cD,\xi}(\frakl_{1}))
\end{equation}
which is the divisibility index of the specialization class $\kappa_{\cD,\xi}(\frakl_{1})$ in the Selmer group. By the choice of $\frakl_{1}$, we have $e_{\cD}(\frakl_{1})< t_{\cD}$.  Now we choose a second $n$-admissible prime $\frakl_{2}$ and define an $n$-admissible form $\cD^{\prime\prime}=(\Delta\frakl_{1}\frakl_{2}, g)$. By carefully choosing the admissible prime $\frakl_{2}$, we can show $t_{\cD^{\prime\prime}}=e_{\cD}(\frakl_{1})< t_{\cD} $ using the second reciprocity law. Thus by the induction hypothesis, we have $s_{\cD^{\prime\prime}}\leq 2t_{\cD^{\prime\prime}}$. To finish the proof, we show $s_{\cD}-s_{\cD^{\prime\prime}}\leq 2(t_{\cD}-t_{\cD^{\prime\prime}})$ by comparing the Selmer group ordinary at $\Delta$ with the Selmer group ordinary at $\Delta\frakl_{1}\frakl_{2}$.

The induction argument is based on the following principle which plays a prominent role in the recent proof of the Kolyvagin's conjecture by Wei Zhang \cite{Zhang_wei:Selmer1}: level raising will make the theta element more primitive and at the same time will bring down the rank of the Selmer group. The one divisibility of Iwasawa main conjecture follows directly from \thmref{induction}.

\item Beyond the applications on the proof of the one divisibility result, the first reciprocity law can also be used in proving the Bloch-Kato type result \thmref{intro_BK} stated above.
We note that the proof of the first reciprocity law does not rely on the "Ihara's lemma" \cf \assref{Ihara}. With this Bloch-Kato type result at hand, we prove the parity conjecture in \thmref{intro_par} in our setting using the general method of Nekovar in \cite{Nekovar_parity3}. We remark that this result is not new and is treated by Nekovar in \cite{Nekovar_Grow}. However our method is different with Nekovar's method, while Nekovar uses Hida family of Hilbert modular forms, we deform the character $\chi$ that we twist our Hilbert modular form by. A similar argument is used in \cite[Theorem 6.4]{Castella_Hsieh} to prove the Parity conjecture in their setting. Chapter 8 is devoted to prove these results. In chapter 9, we prove a multiplicity one result for the space of quaternionic Hilbert modular forms as a module over the Hecke algebra upon localizing at certain maximal ideal. This multiplicity one result is crucial in proving the first and second reciprocity law mentioned earlier.

\item  The proof of the one divisibility result \thmref{main1} relies on the freeness result of the Selmer group theorem \cf \propref{free}. This freeness result in turn is based on the control theorem for Selmer groups \cf \propref{control}. However the control theorem only holds when we replace the hypothesis $\nimpr$ by the stronger one:
\begin{hyp}[$\frakn^{+}\text{-min}$]
if $\frakl\mid \frakn^{+}$, then $\bar{\rho}_{f}$ is ramified at $\frakl$.  
\end{hyp}
Thus we will first prove \thmref{main1} under the hypothesis $\nmin$. However we can relax the assumption $\nmin$ to $\nimpr$ using the one divisibility result proved under $\nmin$. This is carried out in chapter 10 based on the method invented by Pollack-Weston \cf \cite{PW_free}. Using their method we bypass the control theorem to establish the freeness result \lmref{new_free}. Note their result fixes a gap in \cite{Chida_Hsieh:main} who mistakenly proved  the control theorem under the hypothesis $\nimpr$ \cf\cite[Proposition 1.9(2)]{Chida_Hsieh:main}. This gap is also implicit in \cite{Pollack_Weston:AMU}.
\end{itemize}

\subsection{Historical remark}
The method to construct Euler system and to prove the one divisibility result employed in this paper originates from the work of Bertolini and Darmon \cite{Bertolini_Darmon:IMC_anti}, \cite{BD:Heegner_Mumford} where the one divisibility result for elliptic curve over $\Q$ is proved. We also note that Bertolini-Darmon assume the corresponding modular forms are $p$-isolated, which is assumption that excludes the possibility of congruences between the modular form in question with the others. The study of the $\mu$-invariant of the corresponding $p$-adic $L$-function is carried out by Vastal \cite{mu_vastal}, \cite{uniform_vastal} where the celebrated Mazur's conjecture for the definite case (Gross points) is handled. The method in \cite{Bertolini_Darmon:IMC_anti} is subsequently strengthened to handled the more general weight $2$ modular form case by Pollack-Weston in \cite{Pollack_Weston:AMU} and the $p$-isolated condition is relaxed. However the proof there contains a little gap concerning the freeness of Selmer groups, fortunately the authors themselves are able to resolve this issue \cite{PW_free} see also the last chapter in the present article. For modular forms of small weight (comparing to $p$), Chida-Hsieh has constructed the higher weight theta element in \cite{Chida_Hsieh} using congruence between modular forms of different weights when evaluated at CM points. The one divisibility of the main conjecture is proved in \cite{Chida_Hsieh:main} adapting a similar method as in \cite{Pollack_Weston:AMU} and the study of the $\mu$-invariant of the $p$-adic $L$-function is carried out in \cite{Chida_Hsieh}. In another direction, Longo first proved the one divisbility result in \cite{Longo_IMC} for Hilbert modular forms of parallel weight $2$ under similar assumption as in \cite{Bertolini_Darmon:IMC_anti}, some of the assumptions are removed by Van Order in \cite{Van_order:main}. As for the $\mu$-invariant, Cornut-Vastal successfully extended the results in\cite{uniform_vastal} to the totally real field \cite{Vatsal_Cornut:Non_Van}. In \cite{Hung:thesis}, the higher weight theta element for Hilbert modular forms is studied, in particular a similar non-vanishing result to \cite{Vatsal_Cornut:Non_Van} is proved. The present article treat the one divisibility of the Iwasawa main conjecture for small parallel weight Hilbert modular forms based on the work of \cite{Chida_Hsieh:main} and \cite{Hung:thesis}.

\subsection{Notations and conventions} Let $L$ be a number field. We denote by $\cO_{L}$ its ring of integers. If $v$ is a place of $L$, we identify it with a non-zero prime ideal in $\cO_{L}$ and we denote its absolute norm by $\Nm v$. We write the completion of $L$ at $v$ as $L_{v}$ and its valuation ring as $\cO_{L,v}$. We will write $\uf_{v}$ as the uniformizer of $\cO_{L,v}$. If $M/L$ is an algebraic extension and $v$ is a place in $L$, we will set $M_{v}=M\otimes \cO_{L,v}$ and similarly $\cO_{M,v}=\cO_{M}\otimes \cO_{L,v}$. We use the standard notations for Galois groups $G_{L}=\Gal(\bar{L}/L)$. We denote by $\Frob_{v}$ a Frobinius element at $v$ in $\Gal(\bar{L}/L)$. For a place $v$ of $L$, we use the notation $I_{L_{v}}$ for the inertia group of $L$ at $v$. When without the risk of confusion, we abbreviate $I_{L_{v}}$ by $I_{v}$.  We denote by $L^{\ur}_{v}$ the maximal unramified extension of $L_{v}$.

Recall we defined in the introduction for $f$, the $p$-adic Hecke field $E_{f}$, its ring of integers $\cO_{f}$, a uniformizer $\uf$ and $\cO_{f,n}=\cO_{f}/\uf^{n}$.
\subsection{Acknowledgments}
This article is the author's thesis at Penn state university. There seems to be an increasing interest towards this thesis and the author finally decided to revise it and make it available on arXiv. I would like to thank my advisor Professor Winnie Li for constant help and encouragement during my graduate studies. As is clear to the reader, this thesis is an extension of a result of Masataka Chida and Ming-Lun Hsieh. I would like to express my gratitude to my thesis advisor Professor Ming-Lun Hsieh for generously sharing his knowledge with me and answering numerous questions not restricted to the topic of this thesis.

\section{Selmer groups}
\subsection{Galois cohomology} In this section we define several versions of Selmer groups that are related to the Galois representations attached $f$ and prove some standard results about the behavior of the local cohomology over the tower of anticyclotomic extensions. 
Let $L$ be a finite algebraic extension of $F$ and $M$ be a discrete $G_{F}$ module. For a place $\frakl$ in $F$, we define the local cohomology groups
$$H^{1}(L_{\frakl},M)=\oplus_{\lambda\mid \frakl} H^{1}(L_{\lambda},M)$$
where the sum is taken over the places in $L$ above $\frakl$. Similarly we define
$$H^{1}(I_{\frakl},M)=\oplus_{\lambda\mid\frakl} H^{1}(I_{{\lambda}},M)$$
for the inertia group $I_{\lambda}=I_{L_{\lambda}}$.
Let $\text{res}_{\frakl}: H^{1}(L,M)\rightarrow H^{1}(L_{\frakl},M)$ be the restriction map. We define the finite part of the local cohomology group as 
$$H^{1}_{\fin}(L_{\frakl},M)=\ker\{H^{1}(L_{\frakl},M)\rightarrow H^{1}(I_{\frakl},M)\}$$
and the singular quotient as
$$H^{1}_{\sing}(L_{\frakl},M)=H^{1}(L_{\frakl},M)/H^{1}_{\fin}(L_{\frakl},M).$$
We denote by $\partial_{\frakl}$  the residual map obtained from the composition of the restriction map with the natural quotient map
\begin{equation}\label{residue}
\partial_{\frakl}:H^{1}(L,M)\rightarrow H^{1}(L_{\frakl}, M)\rightarrow H_{\sing}^{1}(L_{\frakl},M).
\end{equation}

\subsection{Selmer group}
In this subsection, we define certain Slemer groups for the Hilbert modular form $f$. First we need to recall some properties of the Galois representation associated to $f$, since we are working in the ordinary setup, this representation is constructed in \cite{Wiles:ordinary_hilbert}. Let $\rho_{f}: G_{F}\rightarrow \Aut_{E_{f}}(V_{f})$ be the Galois representation attached to $f$ and we fix a Galois stable lattice $T_{f}$ in $V_{f}$ and put $A_{f}=V_{f}/T_{f}$. Denote  by $\rho^{*}_{f}=\rho_{f}\otimes\epsilon^{\frac{2-k}{2}}$  the self dual twist of $\rho_{f}$. We collect the following properties of $\rho^{*}_{f}$ \cf \cite[Chapter 12]{Nekovar:SelmerComplexes}:
\begin{mylist}
\item {$\rho^{*}_{f}$ is unramified outside of $\frakn p$.}
\item{$\rho^{*}_{f}\mid_{ G_{F_{v}}}=\left(
                                   \begin{array}{cc}
                                     \chi^{-1}_{v}\epsilon^{\frac{k}{2}} & * \\
                                                      \O                &  \chi_{v}\epsilon^{\frac{2-k}{2}} \\
                                   \end{array}
                                 \right)$ for any $v$ above $p$. Here $\chi_{v}$ is an unramfied character such that $\chi_{v}(\Frob_{v})=\alpha_{v}(f)$ with $\alpha_{v}(f)$ the unit root of the Hecke polynomial of $f$ at $v$.}
\item{$\rho^{*}_{f}\mid_{ G_{F_{\frakl}}}=\left(
                                   \begin{array}{cc}
                                     \pm\epsilon           & * \\
                                                      \O                &  \pm1 \\
                                   \end{array}
                                 \right)$ for any $\frakl$ dividing $\frakn$ exactly once.}                            
\end{mylist} 

Let $\uf$ be a uniformizer of $\cO_{f}$ we set:
\begin{mylist}
\item{$\cO_{f,n}=\cO_{f}/\uf^{n}$};
 \item{$T_{f,n}=T_{f}/\uf^{n}$};
 \item{$A_{f,n}=\ker\{A_{f}\MapR{ \uf^{n}}{}A_{f}\}$.}
 \end{mylist}

\begin{defn}\label{n_admissible}
A prime $\frakl$ in $F$ is said to be $n$-admissible for $f$ if
\begin{mylist}
\item{$\frakl\nmid p\frakn$};
\item{$\frakl$ is inert in $K$};
\item{${\Nm \frakl}^{2}-1$ is not divisible by $p$};
\item{$\uf^{n}$ divides $\Nm\frakl^{\frac{k}{2}}+\Nm \frakl^{\frac{k-2}{2}}-\epsilon_{\frakl}\bda_{\frakl}(f)$ where $\epsilon_{\frakl}\in\{\pm1\}$}.
\end{mylist} 
\end{defn}

Let $L/K$ be a finite extension, we now define $F^{+}_{\frakl}A_{f,n}$ and $H^{1}_{\ord}(L_{\frakl}, A_{f,n})$ for a prime $\frakl$ of $F$ which  divides $p\frakn^{-}$ or is $n$-admissible.

\subsubsection{$\frakl=v\mid p$}\label{p-filteration} We have $\exact{F^{+}_{v}V_{f}}{V_{f}}{F^{-}_{v}V_{f}}$ where $F^{+}_{v}V_{f}$ is the subspace that $G_{F_{v}}$ acts by  $\chi^{-1}_{v}\epsilon^{\frac{k}{2}} $ as in the introduction. We take $F^{+}_{v}A_{f}$ to be the image of $F^{+}_{v}V_{f}$ in $A_{f}$ and $F^{+}_{v}A_{f,n}$ be $A_{f,n}\cap F^{+}_{v}A_{f}$.

\subsection{$\frakl\mid\frakn^{-}$} In this case we let $F^{+}_{\frakl}V_{f}$ to be the subspace of $V_{f}$ such that $G_{F_{\frakl}}$ acts by $\pm\epsilon$. Then $F^{+}_{\frakl}A_{f}$ and $F^{+}_{\frakl}A_{f,n}$ are defined similarly as above

\subsubsection{$\frakl$ is $n$-admissible} Since $\frakl$ is $n$-admissible, we know that the characteristic polynomial  for the action of $\Frob_{\frakl}$ is given by $x^{2}-\epsilon_{\frakl}(1+\Nm \frakl)x+\Nm \frakl=(x-\epsilon_{\frakl} \Nm \frakl)(x-\epsilon_{\frakl})$ modulo $\uf^{n}$. We thus let $F^{+}_{\frakl}A_{f,n}$ to be the line that $\Frob_{\frakl}$ acts by $\epsilon_{\frakl}\Nm \frakl$.

\begin{defn}\label{ordpart}
In all three cases we will define the ordinary cohomology group by
 $$H^{1}_{\rm ord}(L_{\frakl}, A_{f,n})=\ker\{H^{1}(L_{\frakl}, A_{f,n})\rightarrow H^{1}(L_{\frakl}, A_{f,n}/F^{+}_{\frakl}A_{f,n})\}.$$
  \end{defn}
 
 Let $\Delta$ be a square free product of primes in $F$ such that $\Delta/\frakn^{-}$ is a product of even number of $n$-admissible primes and let $S$ be an integral ideal of $F$ such that $S$ is prime to $p\Delta \frakn$.

 \begin{defn}\label{Delta_selmer}
 For $M=A_{f,n}\text{ or }T_{f,n}$, we define the $\Delta$-ordinary, $S$-relaxed Selmer group $\Sel_{\Delta}^{S}(L,M)$ with respect to the data $(f, n, \Delta, S)$ to be the group of elements $s \in H^{1}(L,M)$ such that
 \begin{mylist}
 \item{$\partial_{\frakl}(s)=0 \text{ for all } \frakl \nmid p\Delta S$;}
 \item{${\rm res}_{\frakl}(s)\in H^{1}_{\rm ord}(L_{\frakl},M) \text{ for all }  \frakl\mid p\Delta$;}
 \item{${\rm res}_{\frakl}(s)$ is arbitrary at $\frakl\mid S$.}
 \end{mylist}
 \end{defn}
 
 When $S$ is trivial we will abbreviate $\Sel_{\Delta}^{1}(L,M)$ as $\Sel_{\Delta}(L,M)$.

\subsection{Anticyclotomic tower and Selmer groups}
Let $m$ be a non-negative integer. Denote by $H_{ m}$ the ring class field of conductor $\frakp^{m}$ over $K$. Let $G_{m}=\Gal(H_{m}/K)$ and $H_{\infty}=\cup_{m=1}^{\infty}H_{m}$. Let $K_{\infty}$ be the unique subfield of $H_{\infty}$ such that $\Gal(K_{\infty}/K)\simeq \Zp^{d_{\frakp}}$ where $d_{\frakp}$ is the inertia degree of $\frakp$ and we call $K_{\infty}$ the $\frakp$-anticyclotomic $\Zp^{d_{\frakp}}$-extension over $K$. We put $\Gamma=\Gal(K_{\infty}/K)$. Let $K_{m}=K_{\infty}\cap H_{m}$ be the $m$-th layer of this tower of extensions and $\Gamma_{m}=\Gal(K_{m}/K)$ so that $\Gamma=\prolim_{m}\Gamma_{m}$. Let $\Lambda=\cO_{f}[[\Gamma]]$ be the Iwasawa algebra over $\cO_{f}$ and $\frakm_{\Lambda}$ be its maximal ideal.

Let $$H^{1}(K_{\infty}, A_{f,n})=\dirlim_{m}H^{1}(K_{m}, A_{f,n})$$
 where the limit is with respect to the restriction maps and 
 $$\Hhat^{1}(K_{\infty}, T_{f,n})=\prolim_{m}H^{1}(K_{m}, T_{f,n})$$
  where the limit is with respect to the corestriction(norm) maps.  %Try to use propagation of local conditions%
  
 We can define various local cohomology groups $H^{1}_{*}(K_{\infty,\frakl}, A_{f,n})$ and $\Hhat^{1}_{*}(K_{\infty, \frakl}, T_{f,n})$ in the similar fashion. Here $*$ can be any of the finite, singular or ordinary structure defined above. Now we can define 
 $$\Sel^{S}_{\Delta}(K_{\infty},A_{f,n})=\dirlim_{m} \Sel^{S}_{\Delta}(K_{m}, A_{f,n}),$$ 
 $$\hatSel^{S}_{\Delta}(K_{\infty}, T_{f,n})=\prolim_{m}\Sel_{\Delta}^{S}(K_{m}, T_{f,n}).$$
 Recall we have defined the minimal Selmer group $\Sel(K_{\infty}, M)$ in the introduction.

\begin{prop}\label{pollack-weston}
Assume $(\CR)$ and $\nimpr$ hold. Then $\Sel(K_{\infty}, A_{f})$ is of finite index in $$\dirlim_{n}\Sel_{\nminus}(K_{\infty}, A_{f,n}).$$ If furthermore $\Sel_{\nminus}(K_{\infty}, A_{f})$ is $\Lambda$-cotorsion, then they are equal to each other.
\end{prop}

\begin{proof}
This follows from the proof of \cite[Proposition  3.6]{Pollack_Weston:AMU}. The exact same computation shows that $\Sel_{\nminus}(K_{\infty}, A_{f,n})$ is of finite index in $\Sel(K_{\infty}, A_{f})[\uf^{n}]$ with index independent of $n$, then the result follows from the fact that $\Sel(K_{\infty}, A_{f})$ has no finite index submodule if it is $\Lambda$-cotorsion. 
\end{proof}

\subsection{Properties of local Iwasawa cohomologies} 
In this subsection we prove some standard results on the local cohomology groups defined before.
 Since we know that $A_{f,n}$ and $T_{f,n}$ are Cartier dual to each other, we have the following local Tate pairing for each prime $\frakl$ in $K_{m}$
 $$\<,\>_{\frakl}: H^{1}(K_{m,\frakl}, T_{f,n})\times H^{1}(K_{m,\frakl},A_{f,n})\rightarrow E_{f}/\cO_{f}.$$
 Taking limit with respect to $m$, we have
$$\<,\>_{\frakl}: \widehat{H}^{1}(K_{\infty,\frakl}, T_{f,n})\times H^{1}(K_{\infty,\frakl},A_{f,n})\rightarrow E_{f}/\cO_{f}.$$ The following lemma is well-known.

\begin{lm}\label{fin_orth}
Suppose $\frakl\nmid p$,  $H^{1}_{\fin}(K_{\infty,\frakl}, A_{f,n})$ and $\Hhat^{1}_{\fin}(K_{\infty,\frakl}, T_{f,n})$ are orthogonal complement to each other under the pairing $\<,\>_{\frakl}$.
\end{lm}
\begin{proof}
This is a standard result. See \cite[Proposition 1.4.2, 1.4.3]{Rubin:book}.
\end{proof}

It follows from the above lemma that $H^{1}_{\fin}(K_{\infty,\frakl}, A_{f,n})$ and $\Hhat^{1}_{\sing}(K_{\infty,\frakl}, T_{f,n})$ are dual to each other under the above pairing.

One of the differences between the cyclotomic Iwasawa theory and the anticyclotomic Iwasawa theory is the existence of primes that split completely in the $\Zp$-extension and as a consequence we have the following lemma.

\begin{lm}\label{split_inert}
Suppose $\frakl\nmid p$ is a prime in $F$. Then
\begin{mylist}
\item{if $\frakl$ splits in $K$, then  $H^{1}_{\rm fin}(K_{\infty,\frakl}, A_{f,n})=\{0\}$ and hence $\Hhat^{1}_{\rm sing}(K_{\infty,\frakl},T_{f,n})=\{0\}$ .} 
\item{If $\frakl$ is inert in $K$, then $\Hhat^{1}_{\rm sing}(K_{\infty,\frakl},T_{f,n})\simeq H^{1}(K_{\frakl}, T_{f,n})\otimes \Lambda$ and it follows that ${H}^{1}_{\rm fin}(K_{\infty,\frakl}, A_{f,n})\simeq \Hom(\widehat{H}^{1}_{\rm sing}(K_{\frakl},T_{f,n})\otimes \Lambda, {E_{f}}/{\cO_{f}})$.}
\end{mylist}
\end{lm}

\begin{proof}
Suppose that $\frakl=\frakl_{1}\frakl_{2}$ in $K$, then the Frobenius element in $\Gal(K_{\infty}/K)$ above $\frakl_{i}$ is non-trivial. Hence it generates a subgroup of finite index in $G_{\infty}$ and $K_{\infty,\frakl_{i}}$ is a finite product of unramified $\Zp$ extension of $K_{\frakl_{i}}$. Then $(1)$ follows because the pro-$p$ part of $\Gal(K^{ur}_{\frakl_{i}}/K_{\infty,\frakl_{i}})$ is trivial by definition.
If $\frakl$ is inert in $K$, then by class field theory, $\frakl$ splits completely in the  tower $K_{\infty}/K$. The conclusion of $(2)$ follows easily from Shapiro's lemma.
\end{proof}

For $n$-admissible primes, we have a simple decomposition of the local cohomologies.

\begin{lm}\label{admi_coho}
If $\frakl$ is $n$-admissible, then 
\begin{mylist}
\item{both $\Hhat^{1}_{\rm sing}(K_{\infty, \frakl},T_{f,n})$ and $\Hhat^{1}_{\rm fin}(K_{\infty,\frakl},T_{f,n}$) are free of rank one over $\Lambda/\uf^{n}\Lambda$.}
\item{We have the following decompositions: $$\Hhat^{1}(K_{\infty,\frakl}, T_{f,n})=\Hhat^{1}_{\rm fin}(K_{\infty,\frakl},T_{f,n})\oplus\Hhat^{1}_{\rm ord}(K_{\infty,\frakl},T_{f,n})$$ and $$H^{1}(K_{\infty,\frakl}, A_{f,n})=H^{1}_{\rm fin}(K_{\infty,\frakl},A_{f,n})\oplus H^{1}_{\rm ord}(K_{\infty,\frakl},A_{f,n}).$$}
\end{mylist}
\end{lm}

\begin{proof}
By the definition of $n$-admissible primes, we have a decomposition $A_{f,n}=F^{+}_{\frakl}A_{f,n}\oplus F^{-}_{\frakl}A_{f,n}$. Hence it induces a decomposition $H^{1}(K_{\frakl}, A_{f,n})=H^{1}(K_{\frakl}, F^{+}_{\frakl}A_{f,n})\oplus H^{1}(K_{\frakl}, F^{-}_{\frakl}A_{f,n})$.  By definition, we have $H^{1}_{\ord}(K_{\frakl}, A_{f,n})= H^{1}(K_{\frakl}, F^{+}_{\frakl}A_{f,n})$.  On the other hand, we have
$$H^{1}_{\rm fin}(K_{\frakl}, A_{f,n})= F^{-}_{\frakl}A_{f,n}/(\Frob_{\frakl}-1)F^{-}_{\frakl}A_{f,n}=F^{-}_{\frakl}A_{f,n}$$
since $\frakl$ is $n$-admissible and $\Frob_{\frakl}$ acts trivially on $F^{-}_{\frakl}A_{f,n}$. The claim then follows from the fact that $H^{1}(K_{\frakl},F^{-}_{\frakl}A_{f,n})= F^{-}_{\frakl}A_{f,n}$ by local Tate duality and Kummer theory. The rest of the lemma then follows immediately.
\end{proof}

\begin{prop}\label{orthogonal}
Assume $(\CR)$, $\nimpr$ and $(\PO$) hold. Then for $\frakl$ an $n$-admissible prime or a prime dividing $p\frakn^{-}$, $H^{1}_{\ord}(K_{\infty,\frakl}, A_{f,n})$ and $\Hhat^{1}_{\ord}(K_{\infty,\frakl}, T_{f,n})$ are orthogonal complement to each other.
\end{prop}

\begin{proof} We will treat three different cases.

($\frakl$ is admissible) In this case the claim follows directly from the previous lemmas.

($\frakl \mid \nminus$ ) By definition we have an exact sequence 
\begin{equation}\label{Fexact}
\begin{aligned}
0&\mapR H^{1}_{\ord}(K_{\frakl}, A_{f,n})\mapR H^{1}(K_{\frakl}, A_{f,n})\mapR H^{1}(K_{\frakl},A_{f,n}/ F^{+}_{\frakl}A_{f,n})\mapR\\
& H^{2}(K_{\frakl}, F^{+}_{\frakl}A_{f,n})\mapR H^{2}(K_{\frakl}, A_{f,n})\mapR H^{2}(K_{\frakl}, A_{f,n}/F^{+}_{\frakl}A_{f,n})\mapR 0.\\
\end{aligned}
\end{equation}
It follows from this exact sequence and the local Euler characteristic formula, $|H^{1}(K_{\frakl}, A_{f,n})|= |A_{f,n}^{G_{K_{\frakl}}}|^{2}$ and $|H^{1}_{\ord}(K_{\frakl}, A_{f,n})|=|A_{f,n}^{G_{K_{\frakl}}}|.$ This means we only need to prove the ordinary cohomology groups are orthogonal to each other.

By definition, we have an exact sequence

$$0\mapR H^{1}_{\fin}(K_{\frakl}, F^{+}_{\frakl}A_{f,n}) \mapR H^{1}(K_{\frakl}, F^{+}_{\frakl}A_{f,n})\mapR H^{1}_{\sing}(K_{\frakl}, F^{+}_{\frakl}A_{f,n})\mapR 0.$$

When $\Nm\frakl^{2}\not\equiv 1\pmod {\uf^{n}}$, $H^{1}(K_{\frakl}, F^{+}_{\frakl}A_{f,n})$ annihilates $H^{1}(K_{\frakl}, F^{+}_{\frakl} T_{f,n})$ since they pair to $H^{2}(K_{\frakl}, \cO_{f,n}(\epsilon^{2}))=0$. As $H^{1}_{\ord}(K_{\frakl}, A_{f,n})$ (\resp $H^{1}_{\ord}(K_{\frakl}, T_{f,n})$) is contained in the image of $H^{1}(K_{\frakl}, F^{+}_{\frakl}A_{f,n})$ (\resp $H^{1}_{\ord}(K_{\frakl}, T_{f,n})$), the claim in this case follows.

When $\Nm\frakl^{2}\equiv 1\pmod {\uf^{n}}$, a direct calculation shows that $H^{1}_{\ord}(K_{\frakl}, A_{f,n})=H^{1}_{\fin}(K_{\frakl}, A_{f,n})$ (\resp $H^{1}_{\ord}(K_{\frakl}, T_{f,n})=H^{1}_{\fin}(K_{\frakl}, T_{f,n})$). In fact, we can show that both $H^{1}_{\ord}(K_{\frakl}, A_{f,n})$ and $H^{1}_{\fin}(K_{\frakl}, A_{f,n})$ are of rank $1$ over $\cO_{f,n}$ and there is an injection $H^{1}_{\ord}(K_{\frakl}, A_{f,n})\hookrightarrow H^{1}_{\fin}(K_{\frakl}, A_{f,n})$. The orthogonality then follows from \lmref{fin_orth}.

($\frakl=v\mid p$) 
Let $L/K$ be a finite extension in $K_{\infty}$. 
We need the following lemma.
\begin{lm}\label{no-p-invariant}
For $v$ a place of $L$ above $p$, we have $H^{0}(L_{v}, F^{-}_{v}A_{f,n})=0$.
\end{lm}

\begin{proof}
Notice the cyclotomic character $\epsilon$ is surjective on $I_{L_{v}}$ and hence $F^{-}_{v}A_{f,n}=A_{f,n}/F^{+}_{v}A_{f,n}=\cO_{n}(\chi_{v}\epsilon^{1-\frac{k}{2}})$ has no fixed part by $I_{L_{v}}$ unless $k=2$. In the case $k=2$, $(F^{-}_{v}A_{f,n})^{G_{L_{v}}}=(\cO_{n}(\chi_{v}))^{\<\Frob_{v}\>}$ and $\chi_{v}(\Frob_{v})= \alpha_{v}^{p^{s}r_{v}}$. Here $p^{s}$ is the inertia degree of $L$ over $K$ at $v$ and $r_{v}\in\{1,2\}$ depending on whether $v$ splits in $K$. It follows from our  assumption ($\PO$) that $\alpha_{v}^{r_{v}p^{s}}$ is not $1$ on $F^{-}_{v}A_{f,n}$. Then we can conclude that $H^{0}(L_{v}, F^{-}_{v}A_{f,n})=0$. 
\end{proof}

By the definition of the ordinary cohomology and \lmref{no-p-invariant}, we have  $$H^{1}_{\rm ord}(L_{v}, A_{f,n})=H^{1}(L_{v}, F^{+}_{v}A_{f,n}).$$ The same statement holds if we replace $A_{f,n}$ above by $T_{f,n}$. The claim in the case $\frakl=v\mid p$ then follows from the fact that $H^{1}(L_{v},F^{+}_{v}A_{f,n})$ is the orthogonal complement of $H^{1}(L_{v},F^{+}_{v}T_{f,n})$.
\end{proof}

\subsection{Control theorems of Selmer groups}
Recall $S$ and $\Delta$ as in \defref{Delta_selmer} and let $L/K$ be a finite extension contained in $K_{\infty}$.

\begin{prop}\label{control}
Assume $(\CR)$, $\nmin$ and $(\PO)$
hold, then
\begin{mylist}
\item{the restriction maps $$H^{1}(K,A_{f,n})\rightarrow H^{1}(L, A_{f,n})^{\Gal(L/K)}$$ and $$\Sel_{\Delta}^{ S}(K, A_{f,n})\rightarrow \Sel_{\Delta}^{ S}(L,A_{f,n})^{\Gal(L/K)}$$ are isomorphisms;}
\item{$\Sel^{ S}_{\Delta}(L,A_{f,n})=\Sel^{ S}_{\Delta}(L, A_{f})[\uf^{n}]$.}
\end{mylist}
\end{prop}

\begin{proof}
The first claim in part (1) follows from $(\CR)(1)$ and the inflation-restriction exact sequence since by $(\CR)(1)$, $H^{0}(L, A_{f,n})=0$.

By the definition of the Selmer group and an application of the snake lemma, it suffices to show the following claims in order to prove the rest of the assertions in (1):
\begin{enumerate}
\item $H^{1}(K^{\ur}_{\frakl}, A_{f,n})\mapR H^{1}(L^{\ur}_{\frakl}, A_{f,n})$ is injective for $\frakl\nmid p\Delta$;
\item $H^{1}(K_{\frakl}, F^{-}_{\frakl}A_{f,n})\mapR H^{1}(L_{\frakl}, F^{-}_{\frakl}A_{f,n})$ is injective for $\frakl\mid p\Delta$.
\end{enumerate}
When $\frakl\nmid p\Delta$ or $\frakl\mid \Delta$, the first claim is clear as $K_{\infty}$ is unramified outside $\frakp$ and hence $K^{ur}_{\frakl}=L^{ ur}_{\frakl}$. 
If $\frakl=v\mid p$, then by the inflation-restriction exact sequence, we only need to show $F^{-}_{v}A_{f,n}^{G_{L_{\frakP}}}=0$ for $\frakP$ a place above $v$. This is \lmref{no-p-invariant} and this concludes the proof of (1).

As for (2), we consider the following exact sequence
$$0\mapR A_{f,n}\mapR A_{f}\mapR A_{f}\mapR 0.$$
By $(\CR)(1)$ and the fact that $K_{\infty}/K$ is Galois, we have $A^{G_{L}}_{f,m}=0$ for every $m$. Then $A_{f}^{G_{L}}=0$. Hence by the inflation-restriction exact sequence again, the map $$\Sel^{S}_{\Delta}(L, A_{f,n})\mapR \Sel^{S}_{\Delta}(L,A_{f})[\uf^{n}]$$ is injective. To show surjectivity, by a similar consideration as in (1), we are led to prove the following claims:
\begin{itemize}
\item $H^{1}(L^{\ur}_{\frakl}, A_{f,n})\mapR H^{1}(L^{\ur}_{\frakl}, A_{f})$ is injective for $\frakl\nmid p\Delta$.
\item $H^{1}(L_{\frakl}, F^{-}_{\frakl}A_{f,n})\mapR H^{1}(L_{\frakl}, F^{-}_{\frakl}A_{f})$ is injective for $\frakl\mid p\Delta$.
\end{itemize}

For the first claim, if $\frakl\nmid \frakn^{+}$,then the Galois action is trivial and the claim follows immediately. In the case $\frakl\mid \nplus$, the claim follows from our assumption in $(2)$ since then $H^{0}(L^{\ur}_{\frakl}, A_{f})$ is divisible.

We prove the second. For $\frakl\mid \Delta$, the action of $G_{L_{\frakl}}$ is trivial and hence the claim follows immediately. For $\frakl=v\mid p$, we have proved in \lmref{no-p-invariant} that $H^{0}(L_{v}, F^{-}_{v}A_{f,m})=0$ for any $m$ and hence $H^{0}(L_{v}, F^{-}_{v}A_{f})=0$. Therefore we have the injectivity in the case $\frakl=v\mid p$.

\end{proof}

\section{Automorphic forms on quaternion algebra}
In this section we define the notion of automorphic forms on quaternion algebra over a totally real field. In order to use the congruences between Hecke eigenvalues of the Hilbert modular forms, by the Jaquet-Langlands correspondence it suffices to study the Hecke eigenvalues of these quaternionic automorphic forms.
\subsection{Hecke algebra on quaternion algebras}
Let $B$ be a quaternion algebra over the totally real field $F$ and  $\widehat{B}=B\otimes \hat{\Z}$. Set $\Delta_{B}$ to be the discriminant of $B$ over $F$. Let $\Sigma$ be a finite set of places of $F$. We let $$\Bhat^{(\Sigma),\cross}=\{x\in\Bhat^{\cross}: x_{v}=1, \forall v\in \Sigma\}$$ and $$\Bhat^{\cross}_{(\Sigma)}=\{x\in\Bhat^{\cross}: x_{v}=1, \forall v\not\in \Sigma\}.$$
A similar notation is used for any open compact subset $\openU$ of $\Bhatcross$.

Let $\openU$ be an open compact subgroup of $\Bhat^{\cross}$ and $\openU_{v}$ be its component at $v$. Let $\EucH(\Bcross,\openU)$ be the Hecke algebra over $\Z$ on $\Bhatcross$ which is the space of $\Z$-valued bi-$\openU$ invariant functions on $\Bhatcross$  with compact support. Then $\EucH(\Bcross,\openU)$ is equipped with the natural involution product: $(f*g)(z)=\int_{\Bhatcross}f(x)g({x}^{-1}z)dx$ where $dx$ is the normalized Haar measure such that $\vol(\openU,dx)=1$. 

For $x\in \Bhatcross$, we  denote by $[\openU x\openU]$ the characteristic function of the double coset $\openU x\openU$. For each finite place $v$ of $F$, we set $$\pi_{v}=\pMX{\uf_{v}}{0}{0}{1}, z_{v}=\pMX{\uf_{v}}{0}{0}{\uf_{v}}$$
as elements in $\GL_{2}(F_{v})$. We fix an identification $i:\Bhat^{(\Delta_{B}),\cross}\isomor M_{2}(\Fhat^{(\Delta_{B})})^{\cross}$. 

Let $$\frakm^{+}\subset \cO_{F}$$ be an integral ideal which is prime to $\Delta_{B}$ and let $R_{\mplus}$ be an Eichler order(intersection of two maximal orders) of level $\mplus$. Note locally at $v\mid \frakm^{+}$ with $v^{t}\mid \mid \frakm^{+}$, $(R_{\mplus})_{v}=\{\pMX{a}{b}{c}{d}\in M_{2}(\cO_{F,v}): c\equiv 0\pmod{\uf^{t}_{v}}\}$. Suppose $(p,\Delta_{B})=1$. For a non-negative integer $n$  and our fixed $\frakp$, we define a compact open subgroup $\openU_{\frakm^{+}, \frakp^{n}}$ by:

\begin{equation}\label{open_cpt}
\openU_{\mplus, \frakp^{n}}=\{x\in\Rhat^{\cross}_{\frakm^{+}}: x_{\frakp}\equiv \pMX{a}{b}{0}{a}\pmod {\frakp^{n}}, a,b\in \cO_{F,\frakp}\}.
\end{equation}

We set $\openU=\openU_{\mplus,\frakp^{n}}$ and put $\Sigma=\frakp\mplus\Delta_{B}$. Notice by definition $i$ induces an identification $i:(\Bhat^{(\Sigma),\cross}, \openU^{(\Sigma)}_{\mplus, \frakp^{n}})\isomor(\GL_{2}(\Fhat^{(\Sigma)}), \GL_{2}(\Ohat^{\away{\Sigma}}))$. 

\subsubsection{Spherical Hecke algebra} 
We denote by $\Hecke^{\away{\Sigma}}_{B}(\mplus)$ the subalgebra of $\EucH(\Bcross, \openU)$ generated by $[\openU x\openU]$ for $x\in\Bhat^{\away{\Sigma}, \cross}$. Then we have 
$$\Hecke^{\away{\Sigma}}_{B}(\mplus)=\Z[\{T_{v}, S_{v}, S_{v}^{-1}: v\not\in \Sigma\}],$$ where 
\begin{equation}\label{Hecke_standard}
T_{v}=[\openU i^{-1}(\pi_{v})\openU],
S_{v}=[\openU i^{-1}(z_{v})\openU].
\end{equation}

\subsubsection{Complete Hecke algebra} For $v\mid \Delta_{B}$, we choose an element $\pi^{\prime}_{v}\in \Bcross_{v}$ whose norm is a uniformizer of $F_{v}$ and put $$U_{v}=[\openU\pi^{\prime}_{v}\openU].$$

For $v\mid \mplus$, we define $$U_{v}=[\openU i^{-1}(\pi_{v})\openU].$$

Finally for the prime $\frakp$, we define $$U_{\frakp}=[\openU i^{-1}(\pi_{\frakp})\openU],$$ $$ \<a\>=[\openU i^{-1}(\d(a))\openU]$$ where $\d(a)=\pMX{a}{0}{0}{1}$ for $a\in\cO^{\cross}_{F,\frakp}$.

If $n=0$, we have $\openU=\Rhat^{\cross}_{\mplus}$, then we let $$\Hecke_{B}(\mplus)=\Hecke^{\away{\Sigma}}_{B}(\mplus)[\{U_{v}: v\mid \Delta_{B}\mplus\}].$$

If $n>0$,  we have $\openU=\openU_{\mplus,\frakp^{n}}$, then we let $$\Hecke_{B}(\mplus, \frakp^{n})=\Hecke^{\away{\Sigma}}_{B}(\mplus)[\{U_{v}: v\mid \frakp^{n}\Delta_{B}\mplus\},\{\<a\>: a\in\cO^{\cross}_{F,\frakp} \}].$$

We refer to these as the complete Hecke algebras. And for $\frakl$ a prime in $\cO_{F}$, we let $\Hecke^{\away{\frakl}}_{B}(\mplus, \frakp^{n})$ be the subalgebra of the complete Hecke algebra generated by the Hecke operators away from $\frakl$.

\subsection{Optimal embeddings}

Let $K$ be a totally imaginary quadratic extension over $F$ and let $c: z\rightarrow \bar{z}$ be the nontrivial automorphism of $K$ over $F$. We fix a CM type $\Sigma$ of $K$ \ie a subset of $\Hom(K, \C)$ such that $\Sigma\cup\overline{\Sigma}=\Hom(K, \C)$ and $\Sigma\cap\overline{\Sigma}=\emptyset$. We choose $\bftheta\in K$ that satisfies
\begin{mylist}\label{theta} 
\item{$\Im(\sigma(\bftheta))>0, \forall\sigma\in \Sigma;$}
\item{$\{1,\bftheta_{v}\}$ is a basis of $\cO_{K, v}$ over $\cO_{F,v}$ for all $v\mid D_{K/F}\frakp\frakn$;}
\item{$\bftheta$ is local uniformizer at primes $v$ ramified in $K$.}
\end{mylist}
Then $\bftheta$ is a generator of $K$ over $F$. We let $\delta=\bftheta-\bar{\bftheta}$.

Let $B$ be a definite quaternion algebra over $F$ of discriminant $\Delta_{B}$. Assume that $\nminus\mid \Delta_{B}$ and $\Delta_{B}/\nminus$ is a product of inert primes in $K$, then we can regard $K$ as a subalgebra of $B$. Let $T_{B}$ and $N_{B}$ be the reduced trace and reduced norm of $B$. We choose a basis of $B=K\oplus K.J$ over $K$ such that 
\begin{mylist}\label{B_over_K}
\item{$J^{2}=\beta\in F^{\cross}$ with $\sigma(\beta)<0$ for all $\sigma\in \Sigma$ and $Jt=\bar{t}J$ for all $t\in K$;}\label{beta}
\item{$\beta\in (\cO_{F,v}^{\cross})^{2}$ for all $v\mid \frakp\nplus$ and $\beta\in \cO_{F,v}^{\cross}$ for all $v\mid D_{K/F}$. }
\end{mylist} 

We further require the fixed isomorphism $i$ satisfies $\prod_{v\nmid \Delta_{B}} i_{v}: \Bhat^{\away{\Delta_{B}}}\isomor M_{2}(\widehat{\Q}^{\away{\Delta_{B}}})$ such that for each $v\mid \frakp\nplus$, $i_{v}: B_{v}\isomor M_{2}(F_{v})$ is given by
$$i_{v}(\bftheta)=\pMX{T(\bftheta)}{-N(\bftheta)}{1}{0}$$
where $T(\bftheta)=\bftheta+\bar{\bftheta}$, $N(\bftheta)=\bftheta\bar{\bftheta}$; 
$$i_{v}(J)=\sqrt{\beta}\pMX{-1}{T(\bftheta)}{0}{1}$$
and for each place $v\nmid \frakp \nplus \Delta_{B}$, $i_{v}(\cO_{K,v})\subset M_{2}(\cO_{F,v})$. 
 
 \subsection{$P$-adic Hilbert modular forms on quaternion algebra}Let $\Phi$ be a finite extension of $\Q_{p}$ that contains the image of all embeddings $F\hookrightarrow \bar{\Q}_{p}$. Let $A$ be an $\cO_{\Phi}$ algebra. Let $k\geq 2$ be an even integer and for each $\sigma\in \Sigma$ we let $L_{k,\sigma}(A)=A[X_{\sigma},Y_{\sigma}]_{k-2}$ be the space of homogeneous polynomials of degree $k-2$ with two variables over $A$. Denote $\rho_{k,\sigma}: \GL_{2}(A)\rightarrow L_{k,\sigma}(A)$  the representation defined by 
 $$\rho_{k,\sigma}(g)P(X_{\sigma},Y_{\sigma})=\det(g)^{\frac{2-k}{2}}P((X_{\sigma},Y_{\sigma})g).$$
 Our fixed embedding $F\hookrightarrow \bar{\Q}_{p}$ gives a decomposition $\Sigma=\cup_{v\mid p}\Sigma_{v}$ where $\Sigma_{v}$ is the set of those embeddings inducing $v$. We define $$L_{k}(A)=\otimes_{v\mid p}\otimes_{\sigma\in \Sigma_{v}}L_{k,\sigma}(A).$$ We can regard $L_{k}(A)$ as an $\GL_{2}(\cO_{F,p})$ module by the action $\rho_{k}(u_{p})=\otimes_{v\mid p}\otimes_{\sigma\in\Sigma_{v}}\rho_{k,\sigma}(\sigma(u_{v})).$

 Let $\openU$ be an open compact set in $\Bhatcross$. The space $S_{k}^{B}(\openU, A)$ of $p$-adic Hilbert modular forms of level $\openU$ and weight $k$ is defined by
 $$S_{k}^{B}(\openU, A)=\{f:B^{\cross}\backslash \Bhatcross \rightarrow L_{k}(A)\mid f(bu)=\rho_{k}(u_{p})^{-1}f(b), u\in \openU\}.$$
 It is equipped with an action of the Hecke algebra $\Hecke_{B}(\nplus,\frakp^{n})$ given by 
 \begin{equation}
 \begin{aligned}
&[\openU x\openU] f(b) =\sum_{u\in \openU/ \openU\cap x\openU x^{-1}  }f(bux)\text{ for } x\in \Bhat^{\away{\frakp},\cross},\\
&\<a\>f(b)=\rho_{k}(b\d(a))f(b\d(a)), \\
&U_{\frakp}f(b)=\sum_{u\in \openU_{\frakp}/ \openU_{\frakp}\cap\pi_{\frakp}\openU_{\frakp}\pi^{-1}_{\frakp} } \rho_{k}(u)\hat{\rho}_{k}(\pi_{\frakp})f(bu\pi_{\frakp})
\end{aligned}
 \end{equation}
where $$\hat{\rho}_{k}(\pi_{\frakp})(\otimes_{v\mid p}\otimes_{\sigma\in \Sigma_{v}}P_{\sigma}(X_{\sigma},Y_{\sigma}))=\otimes_{\sigma\in\Sigma_{\frakp}}P_{\sigma}(\sigma(\uf_{\frakp})X_{\sigma}, Y_{\sigma})\otimes(\otimes_{\sigma\not\in \Sigma_{\frakp}}P_{\sigma}(X_{\sigma}, Y_{\sigma})).$$ 

We let $X_{B}(\openU)=\Bcross\backslash\Bhatcross/ \openU$ be the Shimura set of level $\openU$. The space of weight 2 Hilbert modular forms $S_{2}^{B}(\openU,A)=\{f: X_{B}(\openU)\rightarrow  A\}$ can be naturally identified with $A[X_{B}(\openU)]$ compatible with the Hecke action if we define Hecke action on the divisor group of the Shimura set  via Picard functoriality.

Let $\tau_{n}\in \Bhatcross_{(\frakp\nplus)}$ be given by 
\begin{equation}\label{atkin-inv}
\tau_{n,v}=\pMX{0}{1}{\uf_{v}^{\ord_{v}(\frakp^{n}\nplus)}}{0}
\end{equation} 
for $v\mid \frakp^{n}\nplus$. Then  $\tau_{n}$ normalizes $\openU$ and it induces an involution on $X_{B}(\openU)$. We refer to this involution as the Atkin-Lehner involution. We define a perfect pairing  $\<,\>_{\openU}: S_{2}^{B}(\openU,A)\times S_{2}^{B}(\openU,A)\rightarrow A$ by
\begin{equation}\label{pairing}
\<f,g\>_{\openU}=\sum_{[b]_{\openU}\in X_{B}(\openU) } f(b)g(b\tau_{n})\#(\Bcross\cap b\openU b^{-1}/F^{\cross})^{-1}. 
\end{equation}
Notice that the action of the Hecke algebra $\Hecke_{B}(\nplus, \frakp^{n})$ is self-adjoint with respect to this paring.

Set $Y= \Fhat^{\cross}$ . Then we have an action of $Y$ on the space $S_{k}^{B}(\openU, A)$ and we define $S_{k}^{B}(\openU Y, A)$ to be the subspace fixed by the action of $Y$ \ie 
$$S_{k}^{B}(\openU Y, A)=S_{k}^{B}(\openU, A)^{Y}.$$
 We also define $\Hecke_{B}^{\away{\Sigma}}(\mplus Y)$ to be the quotient of $\SHecke{\Sigma}{\mplus}$ by the ideal generated by $[\openU x\openU]-1$ for $x\in Y$. We can define $\CHecke{\mplus Y}{ \frakp^{n}}$ in a similar manner.

\section{Shimura curves}
In this section we collect some results on Shimura curve over totally real field following the exposition of \cite{Nekovar:level}.
Let $\frakl\nmid \nplus\Delta_{B}$ be a prime in $F$ inert in $K$. Let $\Bprime$ be an indefinite quaternion algebra with discriminant $\Delta_{B}\frakl$ and split at some $\sigma_{1}\in \Sigma$.  We fix an isomorphism $$\varphi_{B,\Bprime}: \Bhat^{\away{\frakl}} \longrightarrow \Bhat^{\prime \away{\frakl}}.$$ We also let $t^{\prime}:K\rightarrow B^{\prime}$ be an embedding such that it induces the composite map
\begin{equation}\label{embedd}
\Khat^{\away{\frakl}}\MapRR{t}\Bhat^{\away{\frakl}}\xrightarrow{\hspace*{0.2cm}{\varphi_{B,B^{\prime}}}\hspace*{0.2cm}}\Bhat^{\prime\away{\frakl}}.
\end{equation}

\subsection{Shimura curves and complex uniformization} Let $\cO_{\Bprime_{\frakl}}$ be a maximal order in $\Bprime_{\frakl}$.  Set $\openU^{\prime}=\varphi_{B,\Bprime}(\openU^{\away{\frakl}})\cO^{\cross}_{\Bprime_{\frakl}}$. Given $(\Bprime, \openU^{\prime})$, there is a Shimura curve $M_{\openU^{\prime}}$ defined over $F$  which is smooth and projective. The complex points of this curve are given by:
$$\ShiUprime(\C)=B^{\prime\cross}\backslash\ \cH^{\pm}\times \Bprimehatcross/ \openUprime$$
where $\cH^{\pm}={\mathbb{P}}^{1}(\C)-{\mathbb{P}}^{1}(\R).$

We write $[z,b^{\prime}]_{\openUprime}$ for a point in $\ShiUprime$ corresponding to $z\in \cH^{\pm}$ and $b^{\prime}\in \Bprime$.

\subsubsection{Quotient Shimura curves}\label{Shimura_curve}
Recall $Y= \Fhatcross$, there is an action of $Y$ on the Shimura curve $\ShiUprime$ and we denote by $\ShiUprimeY$ the quotient of the Shimura curve $\ShiUprime$ by the action of $Y$. The complex analytic uniformization of this quotient curve is just
$$\ShiUprimeY(\C)=\Bprimecross\backslash \cH^{\pm}\times \Bprimehatcross/\openUprime Y.$$

\begin{remark}
This is the curve denoted by $N^{*}_{H}$ in \cite{Nekovar_CV} and $M_{K^{\prime}}$ in \cite{Zhang:Height}.
\end{remark}

\subsection{$\frakl$-adic uniformization of Shimura curves}
Let $\cO^{ur}_{F,\frakl}$ be the ring of integers of the maximal unramified extension $F^{ur}_{\frakl}$ of $F_{\frakl}$. There is an integral model $\cM_{\openUprime}$ over $\cO_{F,\frakl}$ for $\ShiUprime\otimes F_{\frakl}$ \cite{Cerednik}. Let $\cH_{\frakl}$ be the Drinfeld's $\frakl$-adic upper-half plane. Its $\C_{\frakl}$-points are given by $\mathbb{P}^{1}(\C_{\frakl})-\mathbb{P}^{1}(F_{\frakl})$. The formal completion of $\cM_{\openUprime}$ along its special fiber is then identified with $$\Bcross\backslash\widehat{\cH_{\frakl}}\widehat{\otimes}_{\cO_{F_{\frakl}}}\widehat{\cO}^{ur}_{F,\frakl}\times \Bhat^{ \away{\frakl}\cross}/ \openU^{\away{\frakl}}$$ where $ \widehat{\cH_{\frakl}}$ is a formal model of $\cH_{\frakl}$. Note that the action of $\Bcross$ on $\widehat{\cH_{\frakl}}$ factors through $B^{\cross}\subset\Bcross_{\frakl}\isomor \GL_{2}(F_{\frakl})$, and via the isomorphism $\varphi_{B,\Bprime}$ we can write the above formal completion as
$$\widehat{\cM_{\openUprime}}=\Bcross\backslash\widehat{\cH_{\frakl}}\widehat{\otimes}_{\cO_{F_{\frakl}}}\widehat{\cO}^{ur}_{F,{\frakl}}\times \Bhat^{\away{\frakl}\cross}/\openU^{\away{\frakl}}.$$
Thus the associated rigid analytic space has the following uniformization:
$$\cM_{\openUprime}^{\rig}(\C_{\frakl})=\Bcross\backslash\widehat{\cH_{\frakl}}(\C_{\frakl})\times \Bhat^{\away{\frakl}\cross}/\openU^{\away{\frakl}}.$$

\subsubsection{$\frakl$-adic uniformization of quotient Shimura curves}
We write an element $g\in Y$ as $g=g_{\frakl}g^{\away{\frakl}}$, the action of $g$ on $\widehat{\cM_{\openUprime}}$ is via its action on $\widehat{\cH_{\frakl}}\widehat{\otimes}_{\cO_{F_{\frakl}}}\widehat{\cO}^{ur}_{F_{\frakl}}$ given by $1\times \Frob_{\frakl}^{\ord_{\frakl}(N_{B}(g_{\frakl}))}$. We also denote by $\cM_{\openUprime Y}$ the quotient of $\cM_{\openUprime}$ by the action of $Y$.

From the above discussion, we get immediately
$$\widehat{\cM_{\openUprime Y}}=\Bcross\backslash\widehat{\cH_{\frakl}}\widehat{\otimes}_{\cO_{F_{\frakl}}}\widehat{\cO}^{ur}_{F,{\frakl}}\times \Bhat^{\away{\frakl}\cross}/\openU^{\away{\frakl}}Y=\Bcross\backslash\widehat{\cH_{\frakl}}{\otimes_{\cO_{F_{\frakl}}}}{\cO}_{F,\frakl^{2}}\times \Bhat^{\away{\frakl}\cross}/\openU^{\away{\frakl}}Y^{\away{\frakl}}$$
where ${\cO}_{F,\frakl^{2}}$ is the ring of integers of the unramified quadratic extension of $F_{\frakl}$.

\subsection{Bad reduction of quotient Shimura curves}
We recall the structure of the Bruhat-Tits Tree $\cT_{\frakl}$ for $\Bcross_{\frakl}\isomor\GL_{2}(F_{\frakl})$. Its set of vertices is identified with $\cV(\cT_{\frakl})\isomor \Bcross_{\frakl}/ \openU_{\frakl}\Fcross_{\frakl}$ while its set of oriented edges is identified with $\oeBruhat\isomor \Bcross_{\frakl}/\openU(\frakl)_{\frakl}\Fcross_{\frakl}$.

The curve ${\cM_{\openUprime Y}}\otimes_{\cO_{F,\frakl}} \cO_{F, \frakl^{2}}$ is an admissible curve over $\cO_{F, \frakl^{2}}$ in the sense of \cite[Chapter 3]{JL:local}. The special fiber of this curve is organized by the dual reduction graph $\cG$ which admits the following description: the set of vertices which corresponds to the set of irreducible components of the special fiber is  
identified with 
\begin{equation}\label{vert}
\begin{aligned}
\cV(\cG)&=\Bcross\backslash\vBruhat\times \Z/2\Z\times \Bhat^{\away{\frakl} \cross}/\openU^{\away{\frakl}}Y^{\away{\frakl}}\\
&=\Bcross\backslash \Bhatcross/\openU Y\times \Z/2\Z \\
&=X_{B}(\openU Y)\times \Z/2\Z\\
\end{aligned}
\end{equation}
via the map given by $$B^{\cross}(b_{\frakl}\openU_{\frakl}Y_{\frakl}, j, b^{(\frakl)}U^{(\frakl)}Y^{(\frakl)})\rightarrow (B^{\cross}b_{\frakl}b^{(\frakl)}\openU Y, j, \ord_{\frakl}(N_{B}(b_{\frakl}))).$$
On the other hand the set of oriented edges which corresponds to the set of singular points on the special fiber is identified with  
\begin{equation}
\begin{aligned}
\vec{\cE}(\cG)&=\Bcross\backslash\oeBruhat\times \Z/2\Z\times \Bhat^{\away{\frakl}\cross}/\openU^{\away{\frakl}}Y^{\away{\frakl}}\\
&=\Bcross\backslash \Bhatcross/\openU(\frakl)Y\times \Z/2\Z \\
&=X_{B}(\openU(\frakl)Y)\times \Z/2\Z\\
\end{aligned}
\end{equation}
where we set 
$$\openU(\frakl)=\openU_{\frakl\nplus,\frakp^{n}}$$
by adding a level at $\frakl$ to our previous $\openU$. The last equality in the above equation is given by a similar formula as in \eqref{vert}.

\subsection{Bad reduction of the Jacobians}
Let $J(M_{\openU^{\prime} Y})$ be the Jacobian of the quotient Shimura curve which represents the functor $\text{Pic}^{0}_{M_{\openU^{\prime} Y}/F}\isomor\Res_{F^{\prime}/F}\Pic_{M_{\openU^{\prime}Y}/F^{\prime}}$ where $F^{\prime}$ is the field of moduli of $\ShiUprimeY$.

Let $L/F$ be a field extension and let $\Div^{0}\ShiUprimeY(L)$ be the group of divisors on $\ShiUprimeY(L)$ that have degree zero on each geometric components. For $D\in \Div^{0}\ShiUprimeY(L)$ we denote by $\cl(D)$ the point in $J(\ShiUprimeY)(L)$ representing $D$. 

There is a natural action of $\Hecke_{\Bprime}(\nplus Y, \frakp^{n})$ on the Jacobian under the Picard functoriality. On the other hand we have an isomorphism $\Hecke_{B}^{\away{\frakl}}(\frakl\nplus Y, \frakp^{n})\isomor\Hecke_{\Bprime}^{\away{\frakl}}(\nplus Y, \frakp^{n})$ given by $\varphi_{B, \Bprime}$ and thus we get an action of $\Hecke_{B}^{\away{\frakl}}(\frakl\nplus Y, \frakp^{n})$ on the Jacobian. We can extend this action to the complete Hecke algebra by defining a ring homomorphism 
\begin{equation}\label{phi}
\varphi_{*}: \Hecke_{B}(\frakl\nplus Y, \frakp^{n})\rightarrow \Hecke_{\Bprime}(\nplus Y, \frakp^{n})
\end{equation}
which sends $U_{\frakl}$ to $[\openUprime \pi^{\prime}_{\frakl}\openUprime ]$ where $\pi^{\prime}_{\frakl}$ is chosen such that $N_{B}(\pi^{\prime}_{\frakl})=\uf_{\frakl}$.

Let $\cJ/\cO_{F,\frakl^{2}}$ be the \Neron model of $J(M_{\openUprime Y})/F_{\frakl^{2}}$, $\cJ_{s}$ be the special fiber, $\cJ^{0}_{s}$ be the connected component of $\cJ^{0}_{s}$   and $\Phi_{\ShiUprimeY}$ be the group of connected components of $\cJ_{s}/\cJ^{0}_{s}$. The Hecke algebra $\Hecke_{\Bprime}(\nplus Y,\frakp^{n})$ acts naturally on $\cJ$ and hence also acts on $\Phi_{\ShiUprimeY}$.  Via the morphism $\varphi_{*}$, we can equip $\Phi_{\ShiUprimeY}$  with a structure of $\Hecke_{B}(\frakl\nplus Y,\frakp^{n})$-module.  

\subsubsection{First description of $\Phi_{\ShiUprimeY}$}\label{Phi} The group $\Phi_{\ShiUprimeY}$ can be described in terms of the graph $\cG$. Fix an orientation of $\cG$, by which we mean a section $\cE(\cG)\rightarrow \vec{\cE}(\cG)$ of the natural projection. We choose an orientation such that the source and target map $s,t: \cE(\cG)\rightarrow \vec{\cE}(\cG)\rightarrow \cV(\cG)$ is given by:
\begin{equation}\label{source_target}
\begin{aligned}
 & s: \cE(\cG) =\Bcross\backslash \Bhatcross/\openU(\frakl)Y \rightarrow \cV(\cG)=\Bcross\backslash\Bhatcross/\openU Y\times \Z/2\Z,\\
 & s(\Bcross b\openU(\frakl)Y)=(\Bcross b \openU Y, 0);\\
 \end{aligned}
 \end{equation}
 
 \begin{equation}
 \begin{aligned}
 & t:\cE(\cG)=\Bcross\backslash\Bhatcross/\openU(\frakl)Y\rightarrow \cV(\cG)=\Bcross\backslash\Bhatcross/\openU Y\times \Z/2\Z,\\
 & t(\Bcross b\openU(\frakl)Y)=(\Bcross b\pi^{-1}_{\frakl}\openU Y, 1).\\
 \end{aligned}   
\end{equation}

The chain and cochain complex of $\cG$ is defined by 
\begin{equation}
\begin{aligned}
&\Z[\cE(\cG)]\MapRR{d_{*}}{} \Z[\cV(\cG)]\\
&d_{*}=-s_{*}+t_{*};\\
\end{aligned}
\end{equation}

\begin{equation}
\begin{aligned}
&\Z[\cV(\cG)]\MapRR{d^{*}}{} \Z[\cE(\cG)]\\
&d^{*}=-s^{*}+t^{*}.
\end{aligned}
\end{equation}
Let $\Z[\cV(\cG)]_{0}={\rm im}(d_{*})$, then by \cite[Section 9.6 Theorem 1]{BLR:Neron}, there is a natural identification 
\begin{equation}\label{comp1}
\Phi_{\ShiUprimeY}\isomor \Z[\cV(\cG)]_{0}/ d_{*}d^{*} .
\end{equation}
Let $r_{\frakl}:J(\ShiUprimeY)(F_{\frakl^{2}})\rightarrow \Phi_{\ShiUprimeY}$ be the reduction map. We can in fact describe this map using the above identification. More precisely we have the following commutative diagram:
\begin{equation}\label{comd2}
\begin{CD}
\Div^{0}M_{\openUprime Y}(F_{\frakl^{2}})     @>cl>> J(M_{\openUprime Y})(F_{\frakl^{2}})  \\
@VVr_{\cV}V        @VVr_{\frakl}V\\
\Z[\cV(\cG)]_{0}/ d_{*}d^{*}     @>\isomor>>  \Phi_{\ShiUprime Y}
\end{CD}
\end{equation}
where $r_{\cV}$ is the specialization map : for each $D\in\Div^{0}M_{\openUprime Y}(F_{\frakl^{2}})$, we extend $D$ to a Cartier divisor $\tilde{D}$ on $\cM_{\openUprime Y}\otimes \cO_{F,\frakl^{2}}$, then $r_{\cV}(D)=\sum_{C\in \cV(\cG)}(C,\tilde{D})C$. Here $(C,\tilde{D})$ is the intersection pairing on $\cM_{\openUprime Y}\otimes \cO_{F,\frakl^{2}}$.

\subsubsection{Second description of $\Phi_{\ShiUprimeY}$} Now we give a modular description of $\Phi_{\ShiUprimeY}$. Recall we can identify $\Z[\cE(\cG)]\isomor\Z[X_{B}(\openU(\frakl) Y)]$ with $S^{B}_{2}(\openU(\frakl) Y,\Z)$ and identify $\Z[\cV(\cG)]\isomor\Z[X_{B}(\openU Y)]^{\oplus 2}$ with $S^{B}_{2}(\openU Y,\Z )^{\oplus 2}$. We denote the two natural maps $s, t$ as $\alpha,\beta$ under this identification, then $d_{*}, d^{*}$ are identified with
\begin{equation}
\begin{aligned}
&S^{B}_{2}(\openU(\frakl) Y, \Z)\MapRR{\delta_{*}}{} S_{2}^{B}(\openU Y, \Z)^{\oplus 2}\\
&\delta_{*}=(-\alpha_{*},\beta_{*});\\
\end{aligned}
\end{equation}
and 
\begin{equation}
\begin{aligned}
&S_{2}^{B}(\openU Y, \Z)^{\oplus 2} \MapRR{\delta^{*}}{} S_{2}^{B}(\openU(\frakl)Y, \Z)\\
&\delta^{*}=-\alpha^{*}+\beta^{*}.\\
\end{aligned}
\end{equation}

We define a ring homorphism  $\Hecke_{B}(\frakl \nplus Y, \frakp^{n} )\rightarrow \End(S^{B}_{2}(\openU Y, \Z)^{\oplus 2})$ by 
\begin{equation}
\begin{aligned}
& t\rightarrow \tilde{t}:(x,y)\rightarrow (t(x), t(y)) \text{ for } t\in \SHecke{\frakl}{\frakl \nplus Y,\frakp^{n}};\\
& U_{\frakl}\rightarrow \tilde{U_{\frakl}}:(x,y)\rightarrow (-\Nm\frakl y, x+T_{\frakl}y). 
\end{aligned}
\end{equation}
One can check immediately that $\delta_{*}$ is indeed a $\Hecke_{B}(\frakl\nplus Y, \frakp^{n})$-module map. We put $S^{B}_{2}(\openU Y, \Z)^{2}_{0}=\delta_{*}S_{2}^{B}(\openU(\frakl)Y, \Z)$.

\begin{prop}\label{modular_description}
We have the following $\CHecke{\LnplusY}{\frakp^{n}}$-module isomorphism
$$\Phi_{\ShiUprimeY}\isomor S_{2}^{B}(\openU Y, \Z)^{\oplus 2}_{0}/(\tilde{U}_{\frakl}^{2}-1).$$
\end{prop}

\begin{proof}
First we notice that a direct calculation shows that $$\delta_{*}\delta^{*}=\pMX{\Nm\frakl+1}{-T_{\frakl}}{-T_{\frakl}}{\Nm\frakl+1}.$$  Since $\tau=\pMX{-1}{-T_{\frakl}}{0}{-1}$ is an automorphism of $S^{B}_{2}(\openU Y, \Z)^{\oplus 2}$ and $\tilde{U}_{\frakl}^{2}-1=\delta_{*}\delta^{*}\tau$, we get the desired isomorphism from the first description of $\Phi_{\ShiUprimeY}$  in \eqref{comp1}.
\end{proof}

\subsection{$\frakl$-adic uniformization of Jacobian with purely toric reduction}
In this section, we review the theory of $l$-adic uniformization of the Jacobian of a Shimura curve at a prime of purely toric reduction following \cite[section 1.7]{Nekovar:level}. Let $B^{\prime}$ be the indefinite quaternion algebra with discriminant $\frakl\Delta$ and let $M_{n}^{[\frakl]}=\ShiUprimeY$ be the (quotient) Shimura curve defined in section 4.1.1. Let $J^{[\frakl]}_{n}$ be the Jacobian of $M_{n}^{[\frakl]}$ and let $\Phi^{[\frakl]}$ be the component group of the \Neron model of $J^{[\frakl]}_{n}$ over $\cO_{F,\frakl^{2}}$. Let $\cK$ be a finite unramified extension of $F_{\frakl^{2}}$. Since $J^{[\frakl]}_{n}$ has purely toric reduction over $\cO_{F_{\frakl^{2}}}$, we have the following uniformization exact sequence of $J^{[\frakl]}_{n}$:
\begin{equation}\label{l-unifor}
0\rightarrow X^{[\frakl]} \rightarrow \dual{X^{[\frakl]}}\otimes \bar{F}^{\cross}_{\frakl} \rightarrow J^{[\frakl]}_{n}(\bar{F}_{\frakl})\rightarrow 0.
\end{equation}
Here we denote by $X^{[\frakl]}$ the character group of the connected component of the special fiber of $J^{[\frakl]}_{n}$. Let $m$ be an integer, the above exact sequence yields immediately the following exact sequence of $J^{[\frakl]}_{n}$:
\begin{equation}\label{m-descent}
0\rightarrow \dual{X^{[\frakl]}}\otimes \mu_{m}\rightarrow J^{[\frakl]}_{n}[m]\rightarrow X^{[\frakl]}/m\rightarrow 0
\end{equation} 
where $\mu_{m}$ is the group of $m$-th root of unity. 

Combining Grothendieck's description of the component group with the canonical map $J^{[\frakl]}_{n}\rightarrow \Phi^{[\frakl]}$, we have the following commutative diagram :
\begin{equation}
\begin{CD}
0@>>> X^{[\frakl]}@>>>\dual{X^{[\frakl]}}\otimes \cK^{\cross}@>>> J^{[\frakl]}_{n}(\cK^{\cross})@>>>0\\
 @. @|         @VV\ord_{\frakl} V @VVV @.\\
0@>>> X^{[\frakl]} @>>> \dual{X^{[\frakl]}} @>>> \Phi^{[\frakl]}@>>>0.
\end{CD}
\end{equation}
Let $\text{Kum}: J^{[\frakl]}_{n}(\cK)/m \rightarrow H^{1}(\cK, J^{[\frakl]}_{n}[m])$ be the Kummer map, we can fit the Kummer map in the diagram below:
\begin{equation}\label{Kum}
\begin{CD}
 X^{[\frakl]}\otimes \Z/m\Z @>>>\dual{X^{[\frakl]}}\otimes \cK^{\cross}\otimes\Z/m\Z @>>> J^{[\frakl]}_{n}(\cK)\otimes \Z/m\Z @>>>0\\
  @|         @| @VV\text{Kum} V @.\\
 X^{[\frakl]}\otimes \Z/m\Z @>>> H^{1}(\cK, \dual{X^{[\frakl]}}\otimes \mu_{m}) @>>> H^{1}(\cK, J^{[\frakl]}_{n}[m]).
\end{CD}
\end{equation}
From this diagram it is clear that the image of the Kummer map lies in the image of the map $ H^{1}(\cK, \dual{X^{[\frakl]}}\otimes \mu_{m})\rightarrow H^{1}(\cK, J^{[\frakl]}_{n}[m])$.

\subsection{CM points} Recall we have an embedding  $t^{\prime}: K\rightarrow \Bprime$ \eqref{embedd} which induces in turn an embedding $t^{\prime}(K^{\cross})\subset t^{\prime}(K^{\cross}_{\sigma_{1}})\subset \Bprimecross_{\sigma_{1}}\isomor \GL_{2}(\R)$. There are two points on $\cH^{\pm}=\mathbb{P}^{1}(\C)-\mathbb{P}^{1}(\R)$ fixed by this action and we choose one of them and call it $z^{\prime}$. Then the set of $\frakl$-unramified CM points is defined as
\begin{equation}\label{CM}
CM^{\frakl-\ur}(\ShiUprimeY, K)=\{[z^{\prime}, b^{\prime}]: b^{\prime}\in \Bprimehatcross , b^{\prime}_{\frakl}=1\}.
\end{equation} 

Let $\rec_{K}: \widehat{K}^{\cross}\rightarrow G^{ab}_{K}$ be the geometrically normalized reciprocity map. Then Shimura's reciprocity law yields
\begin{equation}\label{shimura}
\rec_{K}(a)[z^{\prime}, b^{\prime}]= [z^{\prime}, t^{\prime}(a)b^{\prime}].
\end{equation}
Since $Y$ contains $F^{\cross}_{\frakl}$, we have 
\begin{equation}\label{l-unramified}
\iota_{\frakl}: CM^{\frakl-\ur}(\ShiUprimeY, K)\hookrightarrow \ShiUprimeY(K_{\frakl}).
\end{equation} 
Let $D\in \Div^{0}\ShiUprimeY(K_{\frakl})$ be a divisor supported on $CM^{\frakl-\ur}(\ShiUprimeY, K)$. The specialization map can be conveniently calculated by
\begin{equation}\label{specialization}
r_{\cV}(\sum_{i}n_{i}[z^{\prime},b^{\prime}_{i}]_{\openUprime Y})=\sum_{i}n_{i}[\varphi_{B,\Bprime}^{-1}(b^{\prime}_{i})]_{\openU Y}.
\end{equation}

\section{Level raising and Euler system of Heegner points}

We recall the set up in the beginning of the introduction. Let $f\in S_{k}(\frakn,1)$ be a Hilbert modular form of level $\frakn$ and weight $k$. We assume that $f$ is a normalized newform. Then $f$ gives rise to a morphism $\eta_{f}: \Hecke_{\frakn}\rightarrow \C$ where $\Hecke_{\frakn}$ is the Hecke algebra acting on $S_{k}(\frakn,1)$. Let $\{{\Heckepara}_{v}(f)\}$ be the system of Hecke eigenvalues attached to $f$. Then we denote by  $E_{f}$  the $p$-adic Hecke field with ring of integers $\cO_{f}$ and uniformizer $\uf$. We set
\begin{equation}\label{satake}
\begin{aligned}
&\Satakepara_{v}(f)= \text{the unit root of $X^{2}-{\Heckepara}_{v}(f)X+\Nm {v}^{k-1}$}\text{ for $v\mid p$},\\
&\Satakepara_{v}(f)= {\Heckepara}_{v}(f)\Nm{v}^{\frac{2-k}{2}}\text{ for $v\nmid p$}.\\
\end{aligned}
\end{equation}
We define a ring homorphism $\lambda_{f}: \CHecke{\nplus }{ \frakp^{n}}\otimes \cO_{f}=\CHecke{\nplus}{ \frakp^{n}}_{\cO_{f}}\rightarrow \cO_{f}$ by
\begin{equation}
\begin{aligned}
& \lambda_{f}(T_{v})=\Satakepara_{v}(f),\\
& \lambda_{f}(S_{v})=1 \text{ for $v\nmid p\frakn$},\\
&\lambda_{f}(U_{v})=\Satakepara_{v}(f)\text{ for $v\mid p\frakn$},\\
&\lambda_{f}(\<a\>)=a^{\frac{2-k}{2}} \text{ for $a\in \cO^{\cross}_{F,\frakp}$}.
\end{aligned}
\end{equation} 

\subsection{Level raising} let $n$ be a positive integer and recall that $\cO_{f,n}=\cO_{f}/\uf^{n}$ and let $\frakl$ be an $n$-admissible prime.

\begin{defn}\label{admissible_form}
An admissible form $\cD=(\Delta, g)$ is a pair consisting of a square free product of primes in $F$ inert in $K$ and an eigenform $g\in S^{B}_{2}(\openU Y, \cO_{f,n})$ on the definite quaternion algebra $B=B_{\Delta}$ with discriminant $\Delta$ over $F$, of level $\openU=\openU_{\nplus, \frakp^{n}}$\cf \eqref{open_cpt} and which is fixed by the action of $Y$. Then $\cD$ is admissible if the following conditions hold:
\begin{mylist}\label{admisible form}
\item{$\nminus \mid \Delta$ and $\Delta/\nminus$ is a product of $n$-admissible primes.}
\item{$g\pmod{\uf}\neq 0.$}
\item{The eigenform $g\equiv f \pmod{\uf^{n}}$. This means that $\lambda_{f}\equiv \lambda_{g}\pmod{\uf^{n}}$  where $\lambda_{g}:\CHecke{\nplus Y}{\frakp^{n}}_{\cO_{f}}\rightarrow \cO_{f,n}$ is the character on the Hecke algebra $\CHecke{\nplus Y}{\frakp^{n}}$ associated to $g$ mod $\uf^{n}$.}
\end{mylist}
\end{defn}

Let $\cD=(\Delta, g)$ be such an admissible form and let $\frakl\nmid \Delta$ be an $n$-admissible prime for $f$ with $\epsilon_{\frakl}\Satakepara_{\frakl}(f)\equiv \Nm{\frakl}+1\pmod{\uf^{n}}$. Define $\lambda_{g}^{[\frakl]}: \CHecke{\frakl\nplus Y} {\frakp^{n}}_{\cO_{f}} \rightarrow \cO_{f,n}$ by  $\lambda^{[\frakl]}_{g}(U_{\frakl})=\epsilon_{\frakl}$. We define $\cI_{g}$ to be the kernel of $\lambda_{g}$ and $\cI^{[\frakl]}_{g}$ to be the kernel of $\lambda^{[\frakl]}_{g}$. 
For $g$ as above, it induces a surjective map $\psi_{g}: S^{B}_{2}(\openU Y,\cO)/\cI_{g}\rightarrow \cO_{f,n}$ by $\psi_{g}(h)=\<g,h\>_{\openU}$. The following proposition follows from the multiplicity one result \thmref{TW}. However the multiplicity one result is proved for Hilbert modular forms with prime to $p$ level, in order to handle the problem in hand, we use Hida's theory based on the idea of \cite[Proposition 4.2]{Chida_Hsieh:main}. 

\begin{prop}\label{psig}
Assume $(\CR)$ and $\nimpr$ , we have an isomorphism $$\psi_{g}: S^{B}_{2}(\openU Y,\cO)/\cI_{g}\rightarrow \cO_{f,n}.$$
\end{prop}
\begin{proof}
Let $P_{k}$ be the ideal contained in $\cI_{g}$ generated by the set 
$$\{\<a\>-a^{\frac{k-2}{2}}: a\in \cO^{\cross}_{F,\frakp}\}.$$
Let $e=\lim_{n\rightarrow \infty} U^{n!}_{\frakp}$ be Hida's ordinary projector. Then by the same computation as in the proof of \cite[Theorem 8.1]{Hida:p-adic-Hecke-algebra} , we have an isomorphism
$$eS^{B}_{2}(\openU Y, \cO_{f,n})[P_{k}]=eS^{B}_{k}(\widehat{R}^{\cross}_{\frakp\nplus} Y, \cO_{f,n} ).$$
Taking the Pontryagin dual yields another isomorphism
$$eS^{B}_{2}(\openU Y, \cO_{f})/(P_{k}, \uf^{n})=eS^{B}_{k}(\widehat{R}^{\cross}_{\frakp\nplus} Y, \cO_{f})/(\uf^{n}).$$
Thus there is an action of $e\CHecke{\nplus Y}{\frakp^{n}}$ on $eS^{B}_{k}(\widehat{R}^{\cross}_{\frakp\nplus} Y, \cO_{f})/(\uf^{n})$. 
Let $e^{\circ}$ be the ordinary projector $\lim_{n\rightarrow \infty} T^{n!}_{\frakp}$, then the $\frakp$-stablization map induces an isomorphism 
$$e^{\circ}S^{B}_{k}(\widehat{R}^{\cross}_{\nplus} Y,\cO_{f})_{\frakm}\isomor eS^{B}_{k}(\widehat{R}^{\cross}_{\frakp\nplus} Y, \cO_{f})_{\frakm}$$
by the assumption $\PO$.  We define a surjective map $e\CHecke{\nplus Y}{\frakp^{n}}_{\cO_{f}}\rightarrow e_{0}\Hecke_{B}(\nplus)_{\cO_{f}}$ by sending $U_{\frakp}$ to $u_{\frakp}$, a unit root of $x^{2}-T_{\frakp}x-\Nm{\frakp}^{k-1}$ in $e_{0}\Hecke_{B}(\nplus)_{\cO_{f}}$.
Hence there is also an action of $e\CHecke{\nplus Y}{\frakp^{n}}_{\cO_{f}}$ on $e^{\circ}S^{B}_{k}(\widehat{R}^{\cross}_{\nplus} Y,\cO_{f})$.  
Since $U_{\frakp}-\alpha_{\frakp}(f)\in \frakm$ for the maximal ideal $\frakm$ containing $\cI_{g}$, we have the following sequence of isomorphisms
$$S^{B}_{2}(\openU Y,\cO_{f})_{\frakm}/(P_{k}, \uf^{n})\isomor S^{B}_{k}(\widehat{R}^{\cross}_{\frakp\nplus} Y, \cO_{f})_{\frakm}/(\uf^{n})\isomor S^{B}_{k}(\widehat{R}^{\cross}_{\nplus} Y,\cO_{f})_{\frakm}/(\uf^{n}).$$
We will show later that $S^{B}_{k}(\widehat{R}^{\cross}_{\nplus} Y, \cO_{f})_{\frakm}$ is a cyclic $\Hecke(\nplus Y)_{\frakm}$-module assuming $(\CR)$ and $\nimpr$ \cf \thmref{TW}. It follows then $S^{B}_{2}(\openU Y, \cO_{f})/\cI_{g}$ is generated by some $h$ as a $\CHecke{\nplus Y}{\frakp^{n}}$ module. Since $\psi_{g}$ is surjective, $\psi_{g}(h)\in \cO^{\cross}_{f,n}$. As the action of $\CHecke{\nplus Y}{\frakp^{n}}$ is self-adjoint with respect to the pairing $\<,\>_{\openU}$, it follows that the kernel of the map $\psi_{g}$ is precisely $\cI_{g} S^{B}_{2}(\openU Y, \cO_{f})$. 
\end{proof}

Let $B^{\prime}$ be the indefinite quaternion algebra with discriminant $\frakl\Delta$ and let $M_{n}^{[\frakl]}=\ShiUprimeY$ be the (quotient) Shimura curve defined as in section 4.1.1. Let $J^{[\frakl]}_{n}$ be the Jacobian of $M_{n}^{[\frakl]}$ and let $\Phi^{[\frakl]}$ be the component group of the \Neron model of $J^{[\frakl]}_{n}$ over $\cO_{F,\frakl^{2}}$.

\begin{thm}\label{one_dim}
We have an isomorphism $$\Phi^{[\frakl]}_{\cO_{f}}/\cI^{[\frakl]}_{g}\isomor S^{B}_{2}(\openU Y,\cO_{f})/\cI_{g}\MapR{\psi_{g}}{} \cO_{f,n}.$$
\end{thm}

\begin{proof}
Let $\frakm^{[\frakl]}$ be the maximal ideal of $\CHecke{\frakl\nplus}{\frakp^{n}}=\Hecke^{[\frakl]}$ containing $\cI^{[\frakl]}_{g}$ and write $S^{2}=S^{B}_{2}(\openU Y, \cO_{f})^{\oplus 2}_{\frakm^{[\frakl]}}$. Then by \propref{modular_description}, we have $$(\Phi^{[\frakl]}_{\cO_{f}})_{\frakm^{[\frakl]}}\isomor S^{2}/(\tilde{U}^{2}_{\frakl}-1)\isomor S^{2}/(\tilde{U}-\epsilon_{\frakl}).$$ Note it follows from \cite[Proposition 1.5.9(1)] {Nekovar:level} that the quotient $S^{B}_{2}(\openU Y, \cO_{f})/S^{B}_{2}(\openU Y,\cO_{f})_{0}$ is Eisenstein, while $\frakm^{[\frakl]}$ is not Eisenstein. Thus $$\Phi^{[\frakl]}_{\cO_{f}}/\cI^{[\frakl]}_{g}\isomor S^{2}/(\tilde{U}-\epsilon_{\frakl})/ \cI_{g}^{[\frakl]}\isomor (S^{B}_{2}(\openU Y, \cO_{f})^{\oplus 2}/(\tilde{U}_{\frakl}-\epsilon_{\frakl}))\otimes \Hecke^{[\frakl]}_{\frakm^{[\frakl]}}/\cI^{[\frakl]}_{g}.$$
But it is easy to show that 
\begin{equation}
\begin{aligned}
& (S^{B}_{2}(\openU Y, \cO_{f})^{\oplus 2}/(\tilde{U}_{\frakl}-\epsilon_{\frakl}))\otimes \Hecke^{[\frakl]}_{\frakm^{[\frakl]}}/ \cI_{g}^{[\frakl]}\\
&\isomor (S^{B}_{2}(\openU Y, \cO_{f})/(\epsilon_{\frakl}T_{\frakl}-\Nm\frakl-1))\otimes \Hecke_{\frakm}/\cI_{g}\\
&\isomor S^{B}_{2}(\openU Y,\cO_{f})/ \cI_{g}.\\
\end{aligned}
\end{equation}
This completes the proof in view of \propref{psig}.
\end{proof}

Let $T_{p}(J^{[\frakl]}_{n})=\prolim_{m}\Jnl[p^{m}](\bar{F})$ be the $p$-adic Tate module of $\Jnl$

\begin{thm}\label{level_raising}
We have an isomorphism of $G_{F}$-modules
$$T_{p}(\Jnl)_{\cO_{f}}/\cI^{[\frakl]}_{g}\isomor T_{f,n}.$$
\end{thm}

\begin{proof}
We put $T^{[\frakl]}=T_{p}(\Jnl)_{\cO_{f}}$. The proof divides into two parts. First, We prove that $T^{[\frakl]}/\ml\isomor T_{f,1}$. Recall that the component group is computed by the following exact sequence \cite[1.6.5.1]{Nekovar:level}
$$\exact{\chagroup}{\dual{\chagroup}}{\Phi^{[\frakl]}}$$
where $\chagroup$ is the character group of the maximal torus of the special fiber of the Jacobian $\Jnl$. On the other hand by the theory of $\frakl$-adic uniformization of a  Jacbian with totally split toric reduction\cite[1.7.1.1]{Nekovar:level}, we have
$$\exact{\dual{\chagroup}\otimes \mu_{p}}{\Jnl[p]}{\chagroup/p}.$$
Taking the above exact sequence modulo $\ml$, we obtain $$\exact{(\dual{\chagroup_{\cO_{f}}}/\ml/\cY)\otimes \mu_{p}}{\Jnl[p]_{\cO_{f}}/\ml}{\chagroup_{\cO_{f}}/\ml}$$ for some submodule $\cY$. It induces a long exact sequence 
$$\chagroup/\ml\mapR H^{1}(\Fll,(\dual{\chagroup_{\cO_{f}}}/\ml/\cY)\otimes \mu_{p}) \mapR H^{1}(\Fll, \Jnl[p]_{\cO_{f}}/\ml) \mapR H^{1}(\Fll, \chagroup_{\cO_{f}}/\ml).$$
We analyze some terms in this exact sequence as follows
\begin{mylist}
\item{$H^{1}(\Fll, (\dual{\chagroup_{\cO_{f}}}/\ml/\cY)\otimes \mu_{p})=(\dual{ \chagroup_{\cO_{f}}}/\ml/\cY)\otimes \Fll^{\cross}/(\Fll^{\cross})^{p}=\dual{\chagroup_{\cO_{f}}}/\ml/\cY$ by Kummer theory and the assumption that $\frakl$ is $n$-admissible;}
\item{ $H^{1}(\Fll, \chagroup_{\cO_{f}}/\ml)=H^{1}_{\fin}(\Fll,\chagroup_{\cO_{f}}/\ml)$ by the inflation-restriction exact sequence and local class field theory.}
\end{mylist}
Combining these two facts we obtain an exact sequence
\begin{equation}\label{key_exact}
\lexact{\overline{\Phi^{[\frakl]}_{\cO_{f}}/\ml}}{H^{1}(\Fll, T_{p}(\Jnl)_{\cO_{f}}/\ml)}{H^{1}_{\fin}(\Fll, \chagroup_{\cO_{f}}/\ml)}
\end{equation}
where $\overline{\Phi_{\cO_{f}}^{[\frakl]}/\ml}$ is a quotient of $\Phi_{\cO_{f}}^{[\frakl]}/\ml$.

Now by the main result in \cite{BLR:group}, we see that $$H^{1}(\Fll, T_{p}(\Jnl)_{\cO_{f}}/\ml)=H^{1}(\Fll, T_{f,1})^{s}$$ for some $s\geq 1$. Each $H^{1}(\Fll, T_{f,1})$ splits into unramified classes and ordinary classes by \lmref{admi_coho} and both are of rank one. Since $\overline{\Phi_{\cO_{f}}^{[\frakl]}/\ml}$ is at most one dimensional by \thmref{one_dim},  it follows that $s=1$ and $\overline{\Phi_{\cO_{f}}^{[\frakl]}/}\ml=\Phi_{\cO_{f}}^{[\frakl]}/\ml$. We get the result of the first part.

In the second part we prove that $T^{[\frakl]}/\cI^{[\frakl]}_{g}\isomor T_{f,n}$. By the first part and Nakayama's lemma, we can find generators $e_{1}, e_{2}\in T^{[\frakl]}/\cI^{[\frakl]}_{g}$ such that  their reductions generate $T^{[\frakl]}/\ml\isomor T_{f,1}$.  We will prove that the exponent of $T^{[\frakl]}/\cI^{[\frakl]}_{g}$ is $\uf^{n}$ in the next lemma. By the irreducibility of the $T_{f,1}$, we prove that $e_{1}, e_{2}$ have the same exponent and is $\uf^{n}$. The result follows.
\end{proof}

\begin{lm}
The exponent of $T_{p}(\Jnl)_{\cO_{f}}/\cI^{[\frakl]}_{g}$ is $\uf^{n}$.
\end{lm}
\begin{proof}
First notice that $T_{p}(\Jnl)_{\cO_{f}}/\cI^{[\frakl]}_{g}$ is naturally a $\CHecke{\frakl\nplus Y} {\frakp^{n}}_{\cO_{f}} /\cI^{[\frakl]}_{g}$-module, hence its exponent is at most $\uf^{n}$. Then by the proof of \cite[Lemma 5.16]{Bertolini_Darmon:IMC_anti}, we can find $n^{\prime}$ large enough such that $\Jnl[{p}^{n^{\prime}}]_{\cO_{f}}/\cI^{[\frakl]}_{g}\isomor T_{p}(\Jnl)_{\cO_{f}}/\cI^{[\frakl]}_{g}$ surjects onto $\Phi^{[\frakl]}_{\cO_{f}}/\cI^{[\frakl]}_{g}$. By \thmref{one_dim}, $\Phi^{[\frakl]}_{\cO_{f}}/\cI^{[\frakl]}_{g}$ has exponent $\uf^{n}$. It follows that the exponent of $T_{p}(\Jnl)_{\cO_{f}}/\cI^{[\frakl]}_{g}$ is exactly $\uf^{n}$.
\end{proof}

\subsection{Reduction and residue map}

Now let $\tilde{r}_{\frakl}:\Jnl(K_{m})\mapR{}{} \Phi^{[\frakl]}_{\cO_{f}}\otimes \cO_{f,n}[\Gamma_{m}]$ be the reduction map given by
\begin{equation}\label{reduction}
\tilde{r}_{\frakl}(D)=\sum_{\sigma\in \Gamma_{m}}r_{\frakl}(\iota_{\frakl}(D^{\sigma}))[\sigma]
\end{equation}
where $\iota_{\frakl}$ is defined as in \eqref{l-unramified}.  

\begin{prop}\label{H_sing}
We have an isomorphism $$\psi_{g}: \Phi^{[\frakl]}_{\cO_{f}}/\cI^{[\frakl]}_{g}\isomor H^{1}_{\sing}(K_{\frakl}, T_{f,n}).$$
\end{prop}
\begin{proof}
By the proof of \thmref{level_raising}, and in particular  replacing $\ml$ by $\ILg$ in the exact sequence \eqref{key_exact}, we obtain an exact sequence 
\begin{equation}\label{key_exact_1}
\lexact{\Phi_{\cO_{f}}^{[\frakl]}/\ILg}{H^{1}(K_{\frakl}, T_{f,n})}{H^{1}_{\fin}(K_{\frakl}, \chagroup_{\cO_{f}}/\ILg)}.
\end{equation}
Since $\Phi_{\cO_{f}}^{[\frakl]}/\ILg$ is of rank one and  $H^{1}(K_{\frakl}, T_{f,n})=H^{1}_{\fin}(K_{\frakl}, T_{f,n})\oplus H^{1}_{\sing}(K_{\frakl}, T_{f,n})$ as $\frakl$ is $n$-admissible, it follows that $\Phi_{\cO_{f}}^{[\frakl]}/\ILg\isomor H^{1}_{\sing}(K_{\frakl}, T_{f,n})$. 
\end{proof}

Let $\partial_{\frakl}: H^{1}(K_{m}, T_{f,n})\mapR H^{1}_{\sing}(K_{m,\frakl}, T_{f,n})$ be the residue map defined in \eqref{residue}. The following theorem gives a relation between the reduction map and the residue map. 

\begin{thm}\label{comd1}
The following diagram commutes

\begin{equation}
\begin{CD}
    \Jnl(K_{m})_{\cO_{f}}/\ILg @>\Kum>>  H^{1}(K_{m}, T_{f,n})  \\
@VV\tilde{r}_{\frakl}V        @VV\partial_{\frakl}V\\
  \compo_{\cO_{f}}/\ILg\otimes \cO_{f,n}[\Gamma_{m}]  @>\psi_{g}>\isomor> H^{1}_{\sing}(K_{m,\frakl}, T_{f,n}). \\
\end{CD}
\end{equation}
\end{thm}
\begin{proof}
This is a formal consequence of the previous proposition. It is important to note that the proof for the exact sequence \eqref{key_exact} shows that the natural Kummer map in fact factors through the specialization map $\tilde{r}_{\frakl}$. % Add more details The
\end{proof}

\begin{remark}
We can take a limit in the above theorem to get the following diagram
\begin{equation}
\begin{CD}
    \Jnl(K_{\infty})_{\cO_{f}}/\ILg @>\Kum>>  H^{1}(K_{\infty}, T_{f,n})  \\
@VV\tilde{r}_{\frakl}V        @VV\partial_{\frakl}V\\
  \compo_{\cO_{f}}/\ILg\otimes \cO_{f,n}[[\Gamma_{\infty}]]  @>\psi_{g}>\isomor> H^{1}_{\sing}(K_{\infty,\frakl}, T_{f,n}).\\
\end{CD}
\end{equation}
\end{remark}

\subsection{Gross points}
Let $\nplus\cO_{K}= \Nplus\overline{\Nplus}$, we define an element $\zeta_{v}\in \GL_{2}(F_{v})$ for each $v\mid \nplus $  by
\begin{equation}
\zeta_{v}=\delta^{-1}\pMX{\bftheta}{\bar{\bftheta}}{1}{1} \in \GL_{2}(K_{w})=\GL_{2}(F_{v}) \text{ if $v=w\bar{w}$ in $K$.}
\end{equation}

Let $m$ be a positive integer, we define $\zeta_{\frakp}^{(m)}\in \GL_{2}(F_{\frakp})$ by
\begin{equation}
\begin{aligned}
&\zeta_{\frakp}^{(m)}=\pMX{\bftheta}{-1}{1}{0}\pMX{\uf_{\frakp}^{m}}{0}{0}{1} \in \GL_{2}(K_{\frakP})=\GL_{2}(F_{\frakp}) \text{ if $\frakp=\frakP\bar{\frakP}$},\\
&\zeta_{\frakp}^{(m)}=\pMX{0}{1}{-1}{0}\pMX{\uf^{m}_{\frakp}}{0}{0}{1} \text{if $\frakp$ is inert}.
\end{aligned}
\end{equation}

We set $\zeta^{(m)}=\zeta_{\frakp}^{(m)}\prod_{v\mid \nplus}\zeta_{v} \in \Bhatcross_{(\frakp\nplus)}\hookrightarrow \Bhatcross$. Let $\cR_{m}=\cO_{F}+\uf_{\frakp}^{m}\cO_{K}$ be the order of $K$ of level $\uf^{m}_{\frakp}$. It is almost immediately verified that $(\zeta^{\away{m}})^{-1}\widehat{\cR}^{\cross}_{ m}\zeta^{\away{m}}\subset \openU_{\nplus,\frakp^{n}}$  if $m\geq n$.  Recall $\openU=\openU_{\nplus,\frakp^{n}}$,  we define a map 
\begin{equation}
\begin{aligned}
& x_{ m}: \text{Pic} (\cR_{m}Y)= K^{\cross}\backslash \Khatcross / \widehat{\cR}^{\cross}_{m}Y  \mapR  X_{B}(\openU Y),\\ 
& K^{\cross}a\widehat{\cR}^{\cross}_{ m} Y \mapR  x_{m}(a)= [a\zeta^{\away{m}}]_{\openU Y}.\\
\end{aligned}
\end{equation}
The set of points of the form $x_{m}(a)$ for $a\in \text{Pic} (\cR_{m}Y)$ are referred as Gross points of conductor $\frakp^{m}$ on the definite Shimura set $X_{B}(\openU Y)$.

\subsection{Construction of the Euler system $\kappa_{\cD}(\frakl)$}
Given an admissible form $\cD=(\Delta, g)$ and an admissible prime $\frakl\nmid \Delta$, in this section we will define a cohomology class $\kappa_{\cD}(\frakl)=\{\kappa_{\cD}(\frakl)_{m}\}_{m} \in \widehat{\Sel}_{\Delta\frakl}(K_{\infty}, T_{f,n})$.  See \cite[section 7.4.2]{Longo_IMC} for similar construction of the Euler system.

\subsubsection{Heegner points on Shimura curves and its Jacobian}
Let $\openUprime=\varphi_{B,B^{\prime}}(\openU^{\away{\frakl}})\cO^{\cross}_{B^{\prime}_{\frakl}}$ and let $M^{[\frakl]}_{n}=\ShiUprimeY$ be the quotient Shimura curve. For each $a\in\Khatcross$ and $m\geq n$, we define a Heegner point $P_{m}(a)$ by

$$P_{m}(a)=[z^{\prime}, \varphi_{B,\Bprime}(a^{\away{\frakl}}\zeta^{(m)}\tau_{n})]_{\openUprime Y}\in M^{[\frakl]}_{n}(\C)$$
where $\tau_{n}$ is the Atkin-Lehner involution \eqref{atkin-inv}. This point is defined on a subfield of the ring class field $H_{m}$ and the Galois action of $G_{m}=\Gal(H_{m}/K)$ is given by Shimura's reciprocity law \eqref{shimura}. 
We choose an auciliary prime $\frakq\nmid p\frakl\nplus\Delta$ such that $1+\Nm{\frakq}-\Satakepara_{\frakq}(f)\in \cO^{\cross}_{f}$.  We define a map
\begin{equation}
\begin{aligned}
&\xi_{\frakq}: \Div \Shiln(H_{m}) \mapR \Jnl(H_{m})_{\cO_{f}};\\
&\xi_{\frakq}(P)= \frac{1}{(1+\Nm{\frakq}-\Satakepara_{\frakq}(f))}\cl((1+\Nm{\frakq}-T_{\frakq})P).\\
\end{aligned}
\end{equation}
Put $P_{m}=P_{m}(1)$. We define the Heegner divisor over the anticyclotomic tower by
$$D_{m}=\sum_{\sigma\in \Gal(H_{m}/K_{m})}\xi_{\frakq}(P^{\sigma}_{m})\in \Jnl(K_{m})_{\cO_{f}}.$$

\subsubsection{Definition of $\kappa_{D}(\frakl)$}
Let $\text{Kum}: \Jnl(K_{m})_{\cO_{f}}\mapR H^{1}(K_{m}, T_{p}(\Jnl)_{\cO_{f}})$ be the Kummer map. We define the cohomology class $\kappa_{\cD}(\frakl)_{m}$ as:

\begin{equation}
 \kappa_{\cD}(\frakl)_{m}:= \frac{1}{\Satakepara^{m}_{\frakp}(f)} \text{Kum}(D_{m}) \pmod{\ILg}\in H^{1}(K_{ m}, T_{p}(\Jnl)_{\cO_{f}}/\ILg).
\end{equation}\label{class}
Note we have $H^{1}(K_{ m}, T_{p}(\Jnl)_{\cO_{f}}/\ILg) {\isomor} H^{1}(K_{ m}, T_{f,n})$ by \thmref{level_raising}.  Thus we treat $\kappa_{\cD}(\frakl)_{m}$ as an element in $H^{1}(K_{ m}, T_{f,n})$.

\begin{lm}[Norm compatibility]\label{Normcom}
$$\text{Cor}_{K_{m+1}/K_{ m}}(D_{m+1})=U_{\frakp}D_{m}.$$
\end{lm}
\begin{proof}
We need to calculate $\text{Cor}_{K_{ m+1}/K_{ m}}(P_{m+1})$:
\begin{equation}
\begin{aligned}
\text{Cor}_{K_{ m+1}/K_{ m}} P_{m+1} &= \sum_{\sigma\in \Gal(K_{m+1}/K_{m})} \sigma P_{m+1}\\
&=\sum_{y\in \widehat{\cR}^{\cross}_{ m}/\widehat{\cR}^{\cross}_{ m+1}} \rec{y} (P_{m+1})\\
&=\sum_{y\in \widehat{\cR}^{\cross}_{ m}/\widehat{\cR}^{\cross}_{ m+1}} [z^{\prime}, t^{\prime}(y)\varphi_{B,\Bprime}(\zeta^{(m+1)}\tau_{n})]_{\openUprime Y}.\\
\end{aligned}
\end{equation}
Note that $\Gal(K_ {m+1}/K_{m})=\{[1+\uf_{\frakp}^{m}u\bftheta]_{m+1}: u \in \cO_{F,\frakp}/\frakp\}$ and for each $u$, $t^{\prime}([1+\uf_{\frakp}^{m}u\bftheta]_{m+1})= \pMX{1+\uf^{m}_{\frakp}u T(\bftheta)}{-\uf_{\frakp}^{m}u N(\bftheta)}{\uf_{\frakp}^{m}u}{1}$.

In the case when $\frakp$ is split in $K$, we have $\zeta^{(m+1)}=\pMX{\bftheta}{-1}{1}{0}\pMX{\uf^{m+1}_{\frakp}}{0}{0}{1}$ we compute 
\begin{equation}
\begin{aligned}
&(\zeta^{(m+1)})^{-1}\pMX{1+\uf^{m}_{\frakp}u T(\bftheta)}{-\uf_{\frakp}^{m}u N(\bftheta)}{\uf_{\frakp}^{m}u}{1}\zeta^{(m+1)}\\
&=\pMX{1+\uf^{m}_{\frakp}u\bftheta }{-\frac{u}{\uf_{\frakp}}}{0}{1+\uf^{m}_{\frakp}u\bar{\bftheta}}\\
 &= \pMX{1}{-\frac{u}{\uf_{\frakp}(1+\uf^{m}_{\frakp}u\bftheta)}}{0}{1} \pMX{1+\uf^{m}_{\frakp}u\bftheta}{0}{0}{1+\uf^{m}_{\frakp}u\bar{\bftheta}}.\\
\end{aligned}
\end{equation}
Since  $\pMX{1+\uf^{m}_{\frakp}u\bftheta}{0}{0}{1+\uf^{m}_{\frakp}u\bar{\bftheta}}\in \openU Y$, we can ignore it.

Note $\zeta^{(m+1)}=\zeta^{(m)}\pMX{\uf_{\frakp}}{0}{0}{1}$, and we get
\begin{equation}
\begin{aligned}
&\zeta^{(m+1)}(\zeta^{(m+1)})^{-1} \pMX{1+\uf^{m}_{\frakp}u T(\bftheta)}{-\uf_{\frakp}^{m}u N(\bftheta)}{\uf_{\frakp}^{m}u}{1}\zeta^{(m+1)}\\
&=\zeta^{(m)}\pMX{\uf_{\frakp}}{0}{0}{1}\pMX{1}{-\frac{u}{\uf_{\frakp}(1+\uf^{m}_{\frakp}u\bftheta)}}{0}{1}\\
&=\zeta^{(m)}\pMX{\uf_{\frakp}}{-\frac{u}{(1+\uf^{m}_{\frakp}u\bftheta)}}{0}{1}.\\
\end{aligned}
\end{equation}

In the case when $\frakp$ is inert in $K$, we have $\zeta^{(m+1)}=\pMX{0}{1}{-1}{0}\pMX{\uf^{m+1}_{\frakp}}{0}{0}{1}$, and we do  exactly the same computation
\begin{equation}
\begin{aligned}
&(\zeta^{(m+1)})^{-1}\pMX{1+\uf^{m}_{\frakp}u T(\bftheta)}{-\uf_{\frakp}^{m}u N(\bftheta)}{\uf_{\frakp}^{m}u}{1}\zeta^{(m+1)}\\
&=\pMX{1}{\frac{u}{-\uf_{\frakp}}}{\uf^{2m+1}_{\frakp}uN(\bftheta)}{1+\uf^{m}_{\frakp}uT(\bftheta)}\\
 &= \pMX{1}{-\frac{u}{\uf_{\frakp}}}{0}{1} \pMX{1+\uf^{2m}_{\frakp}u^{2}N(\bftheta)}{\uf^{m-1}_{\frakp}u^{2}T(\bftheta)}{\uf^{2m+1}_{\frakp}uN(\bftheta)}{1+\uf^{m}_{\frakp}uT(\bftheta)}.\\
\end{aligned}
\end{equation}
Again the last  matrix is in $\openU Y$ and hence can be ignored. The rest of the computation is the same as in the previous case.
\end{proof}

\begin{remark}
We remark that the Hecke action on $\Jnl$ and $\Z[X_{B}(\openU Y)]$ is with respect to the Picard functoriality. The additional $\tau_{n}$ is a modification due to the fact that the action of the Hecke correspondence induces dual actions on the Jacobian under Picard functoriality and on the space of modular forms.
\end{remark}

By the above lemma, we can define the cohomology class $\kappa_{\cD}(\frakl)=\{\kappa_{\cD}(\frakl)_{m}\}_{m}\in \widehat{H}^{1}(K_{\infty}, T_{f,n})$. These elements are in fact in the Selmer group defined in \defref{Delta_selmer} as shown by the following propositon.

\begin{prop}
The cohomology class $\kappa_{\cD}(\frakl)$ is in the Selmer group $\widehat{\Sel}_{\Delta\frakl}(K_{\infty}, T_{f,n}).$
\end{prop}

\begin{proof}
For each $m\geq n$, we need to show that 
\begin{enumerate}
\item{$\partial_{v}(\kappa_{\cD}(\frakl))=0$ for $v\nmid \Delta\frakl p$;}
\item{${\rm res}_{v}(\kappa_{\cD}(\frakl))\in {H^{1}_{\ord}}(K_{\infty}, T_{f,n})$ for $v\mid \Delta\frakl p$.}
\end{enumerate}

For the first point, if $v\nmid\frakn^{+} $, the result follows from the fact that $\Jnl$ has good reduction at such $v$. On the other hand if $v\mid\frakn^{+}$, then $\widehat{H}^{1}(K_{\infty,v}, T_{f,n})=\widehat{H}_{\fin}^{1}(K_{\infty,v}, T_{f,n})$ by \lmref{split_inert} and this shows the claim.

For the second point if $v\mid \frakl\Delta$, then the result follows from our previous discussion \cf{Section 4.5}. Note the Kummer image belongs to the ordinary part of the cohomology group as the Jacobian has purely toric reduction at these primes. See \eqref{Kum} and the discussion after the diagram.

Finally it remains to prove the case for $v\mid\frakp$. Let $T=T_{p}(\Jnl)_{\cO_{f}}$, $\cI=\cI^{[\frakl]}_{g}$ and $\Hecke^{[\frakl]}=\Hecke_{B}(\frakl\nplus Y,\frakp^{n})$, then $T_{\frakm^{[\frakl]}}$ is a direct summand of $T$ and $T_{\frakm^{[\frakl]}}/\cI\isomor T_{f,n}$. Notice that $V_{\frakm^{[\frakl]}}=T_{\frakm^{[\frakl]}}\otimes E_{f}$ is a direct sum of representations $\rho_{h}\otimes \epsilon$ for $h$ a $p$-ordinary Hilbert new form of weight $2$ and central character $\epsilon^{-2}$ such that $\rho_{h}\otimes \epsilon \isomor  \rho^{*}_{f}\mod \uf$. As before we have an exact sequence 
$$\exact{F^{+}_{v}V_{\frakm^{[\frakl]}}}{V_{\frakm^{[\frakl]}}}{F^{-}_{v}V_{\frakm^{[\frakl]}}}.$$ The inertia group at $v$ acts on $F^{+}_{v}V_{\frakm^{[\frakl]}}$ by $\epsilon_{\Hecke}\epsilon$ and on $F^{-}_{v}V_{\frakm^{[\frakl]}}$ via $\epsilon_{\Hecke}^{-1}$. Here $\epsilon_{\Hecke}:G_{F_{v}}\rightarrow \Hecke^{[\frakl]}$ is given by $\epsilon_{\Hecke}(\sigma)=\<\epsilon(\sigma)\>$. We set $F^{+}_{v}T_{\frakm^{[\frakl]}}=T_{\frakm^{[\frakl]}}\cap F^{+}_{v} V_{\frakm^{[\frakl]}}$, then $F^{+}_{v}V_{\frakm^{[\frakl]}}/\cI\isomor F^{+}_{v}T_{f,n}$ as Galois modules. We have the following commutative diagram \cf\cite[Proposition 12.5.8]{Nekovar:SelmerComplexes}
\begin{equation}
\begin{CD}
   H^{1}(K_{m,v}, T_{\frakm^{[\frakl]}})  @>\alpha>>  H^{1}(K_{m,v}, F^{-}_{v}T_{\frakm^{[\frakl]}})  \\
@VV\beta V        @VV\beta^{\prime}V\\
  H^{1}(K_{m,v}, V_{\frakm^{[\frakl]}}) @>\alpha^{\prime}>> H^{1}(K_{m,v},F^{-}_{v} V_{\frakm^{[\frakl]}}) .\\
\end{CD}
\end{equation}

Now let $\Kum(D_{m})_{\frakm^{[\frakl]}}$ be the image of $\Kum(D_{m})$ in $H^{1}(K_{m,\frakp}, T_{\frakm^{[\frakl]}})$, it suffices to show that $\alpha(\Kum(D_{m})_{\frakm^{[\frakl]}})=0$. This follows from the fact that $\alpha^{\prime}\circ\beta(\Kum(D_{m})_{\frakm^{[\frakl]}})=0$ by \cite[Example 3.11]{Bloch_Kato} and that $\beta^{\prime}$ is injective by \lmref{no-p-invariant}.
\end{proof}.

\section{Explicit reciprocity laws}

In this chapter we explain the relationship between Heegner point and the theta element. While the theta element  interpolates $L$-values, Heegner points give rise to elements in the Selmer group. The relationship allows us to bound the Selmer group in terms of $L$-values. The coexistence of Heegner point and theta element is a reflection of the parity of the Selmer group and is a special case of the so called bipartite Euler system \cite{Howard:euler}.

\subsection{The first explicit reciprocity law}
Let $\cD=(\Delta, g)$ be an admissible form. Define $$\Theta_{m}(\cD)=\frac{1}{\Satakepara_{\frakp}^{m}}\sum_{[a]\in G_{m}}g(x_{m}(a))[a]_{m}\in \cO_{f,n}[G_{m}]$$ where $[a]_{m}=\rec_{K}(a)\mid_{ G_{m}}\in G_{m}$ is the map induced by the reciprocity law normalized geometrically. These elements are norm compatible in the following sense.

\begin{lm}
If we denote the quotient map by $\pi_{m+1,m}:G_{m+1}\mapR G_{m}$, then $$\pi_{m+1,m}(\Theta_{m+1}(\cD))=\Theta_{m}(\cD).$$
\end{lm}
\begin{proof}
This follows from the exact same computation in \lmref{Normcom}  and the fact that $g$ is a $U_{\frakp}$-eigenform.
\end{proof}

Let $\pi_{m}: G_{m}\rightarrow \Gamma_{m}=\Gal(K_{m}/K)$ be the natural map and let 
$$\theta_{m}(\cD)=\pi_{m}(\Theta_{m}(\cD))\in \cO_{f,n}[\Gamma_{m}],$$ 
$$\theta_{\infty}(\cD)=(\theta_{m}(\cD))_{m}\in \cO_{f,n}[[\Gamma]].$$
The following theorem is referred to as the first reciprocity law which was a terminology used first in \cite{Bertolini_Darmon:IMC_anti}. For proofs of the first reciprocity in other setups, see \cite[Theorem 4.1]{Bertolini_Darmon:IMC_anti}, \cite[Theorem 7.22]{Longo_IMC}.

\begin{thm}[First reciprocity law] \label{1law}
Let $m\geq n\geq 0$ and $\frakl$ be an $n$-admissible prime for $f$. For an $n$-admissible form $\cD=(\Delta, g)$, we have 
$$\partial_{\frakl}(\kappa_{\cD}(\frakl)_{m})=\theta_{m}(\cD)\in \cO_{f,n}[\Gamma_{m}]$$
and in paticular we have
$$\partial_{\frakl}(\kappa_{\cD}(\frakl))=\theta_{\infty}(\cD)\in \cO_{f,n}[[\Gamma]].$$
\end{thm}

\begin{proof}
By the commutativity of the diagram in \thmref{comd1} and the definition of $\tilde{r}_{\frakl}$, we obtain
\begin{equation}
\begin{aligned}
\partial_{\frakl}(\kappa_{\cD}(\frakl)_{m})= \sum_{\sigma\in \Gamma_{m}}\psi_{g}(r_{\frakl}(D^{\sigma}_{m}))[\sigma] .
\end{aligned}
\end{equation}

By the commutativity of the diagram \eqref{comd2} and equation \eqref{specialization}, we have 
\begin{equation}
r_{\frakl}(D_{m}^{\sigma})= \sum_{[b]_{m}\in \Gal(H_{ m}/K_{ m})} x_{m}(ub)\tau_{n}
\end{equation}
where $u$ is given by $[u]_{m}=\sigma\in \Gal(K_{ m}/K)$ and $\tau_{n}$ is the Atkin-Lehner involution given in  \eqref{atkin-inv}.

Hence we conclude that 

\begin{equation}
\begin{aligned}
\partial_{\frakl}(\kappa_{\cD}(\frakl)_{m})= & \sum_{[u]_{m}\in \Gal(K_{m}/K)}\sum_{[b]_{m}\in\Gal(H_{m}/K_{m})}\<g, x_{m}(ub)\tau_{n}\>_{\openU Y}[u]_{m}\\
 &= \sum_{a\in \Gal(H_{m}/K)}g(x_{m}(a))\pi_{m}([a]_{m}).\\
\end{aligned}
\end{equation}
\end{proof}

\begin{remark}
In the above proof, especially the last line, we identify the Gross point $x_{m}(a)$ as an element in $S^{B}_{2}(\openU Y, \Z)$ in the following manner.  We have a sequence of isomorpshims
$$\Z[X_{B}(\openU Y)]\isomor \Hom_{\Z}(S^{B}_{2}(\openU Y, \Z), \Z)\isomor S^{B}_{2}(\openU Y, \Z) $$
where the first isomorphism is given by sending $x \in X_{B}(\openU Y)$ to the evaluation function of $f\rightarrow f(x)$ and the second isomorphism is given by $\<,\>_{\openU Y}$. These isomorphisms are all "Hecke compatible".\end{remark}

\subsection{Ihara's lemma for Shimura curves}
% We will need to discuss the Ihara's lemma and integral model of the Shimura curve. We will only give the statement of the second reciprocity law.
Let $\cU$ be an open compact subgroup of $\widehat{B}^{\prime \cross}$ and $M_{\cU}$ be the Shimura curve of level $\cU$. For any ring $A$, let $\cF_{k}(A)$ be the local system associated to $L_{k}(A)$ over $M_{\cU}$ and we set $\cL_{k}(\cU, A)= H_{et}^{1}(M_{\cU}, \cF_{k}(A))$. If  $v\nmid \Delta_{B}p\nplus$ is a prime in $F$, we denote by $\cU(v)$ the open compact subgroup of $\widehat{B}^{\prime \cross}$ obtained by adding a level at $v$ to $\cU$. Let $\F=\cO_{f,1}$ Then the following assumption is usually known as the "Ihara's lemma":

\begin{assumption}\label{Ihara}
If $\cU$ is maximal at $v$ and places above $p$ and $2\leq k<p+1$, then we have an injective map
$$1+\eta_{v}^{*}: \cL_{k}(\cU, {\F})^{2}_{\m}\rightarrow \cL_{k}(\cU(v),{\F})_{\m}.$$
Here $\eta_{v}$ is the natural degeneracy map from $M_{\cU(v)}$ to $M_{\cU}$ and $\m$ is a non-Eisenstein maximal ideal in the spherical Hecke algebra of level $\cU(v)$.
\end{assumption}

\begin{remark}
A proof of the Ihara's lemma is contained in \cite{Cheng:Ihara} following the method of Diamond-Taylor \cite{Diaomond_Taylor:Inventione}. It was never published and we treat this as an assumption. Note that recently, Manning and Shotton provide another proof of the Ihara's lemma via patching in \cite{ManSho:Ihara} and if we assume that the residual representation has large image we can apply their results. 
\end{remark}

\subsection{The second explicit reciprocity law}

Recall we have a homomorphism $$\lambda^{[\frakl]}_{g}: \CHecke{\frakl\nplus Y} {\frakp^{n}}_{\cO_{f}} \rightarrow \cO_{f,n}$$ whose kernel is denoted by $\cI^{[\frakl]}_{g}$. Let $\m$ be the maximal ideal containing $\cI^{[\frakl]}_{g}$ in $\CHecke{\frakl\nplus Y} {\frakp^{n}}_{\cO_{f}}$. Also denote $\widehat{R^{\prime}}_{\frakc}^{\cross}$ the Eichler order of level $\frakc$ for some integral ideal $\frakc$ in $F$. Note we can not apply the Ihara's lemma directly as our open compact group $\openU^{\prime}=\openU_{\frakn^{+},\frakp^{n}}$ is not maximal at $\frakp$.  To remedy, we use the following proposition which again uses Hida theory.

\begin{prop}\label{Ihara_remedy}
Assume $\PO$ and the "Ihara's lemma", we have $$1+\eta_{v}^{*}: \cL_{2}(\openU^{\prime},\F)_{\m}^{2}[P_{k}]\rightarrow \cL_{2}(\openU^{\prime}(v),\F)_{\m}[P_{k}]$$ is injective.
\end{prop}
\begin{proof}
Using \cite[ Theorem 8.1, Corollary 8.2]{Hida:p-adic-Hecke-algebra}, we have the following isomorphisms
$$\cL_{2}(\openU^{\prime},\F)_{\m}[P_{k}]\isomor \cL_{k}(\widehat{R^{\prime}}_{\frakn^{+}\frakp}^{\cross},\F)_{\m},$$ 
$$\cL_{2}(\openU^{\prime}(v), \F)_{\m}[P_{k}]\isomor\cL_{k}(\widehat{R^{\prime}}_{\frakn^{+}v\frakp}^{\cross}, \F)_{\m}.$$ 
Now let $\cU$ be either $\widehat{R^{\prime}}^{\cross}_{\frakn^{+}}$ or $\widehat{R^{\prime}}^{\cross}_{\frakn^{+} v}$ and consider the injective map
$$\cL_{k}(\cU,\C_{p})_{\m}\rightarrow \cL_{k}(\cU(\frakp), \C_{p})_{\m}.$$ 
 By the Eichler-Shimura relation, the cokernel of this map is is two copies of the $\frakp$-new part of the ordinary Hilbert modular forms on $B^{\prime}$ and hence is killed by $U_{\frakp}^{2}-\Nm{\frakp}^{k-2}$. Thus the cokernel of the map $$\cL_{k}(\cU,\cO_{f})_{\m}\rightarrow \cL_{k}(\cU(\frakp), \cO_{f})_{\m}$$ is $\cO_{f}$-torsion under the assumption $\PO$. By following the same argument as in \cite[Lemma 4]{Diaomond_Taylor:Inventione}, $\cL_{k}(\cU(\frakp), \cO_{f})_{\m}$ is torsion free. Therefore $\cL_{k}(\cU,\cO_{f})_{\m}\rightarrow \cL_{k}(\cU(\frakp), \cO_{f})_{\m}$ is an isomorphism and so is its reduction modulo $\uf$. Thus we have proved the following isomorphisms:
$$\cL_{k}(\widehat{R^{\prime}}^{\cross}_{\frakn^{+}\frakp}, \F)_{\m}\isomor \cL_{k}(\widehat{R^{\prime}}^{\cross}_{\frakn^{+}}, \F)_{\m},$$
$$\cL_{k}(\widehat{R^{\prime}}^{\cross}_{\frakn^{+}v\frakp }, \F)_{\m}\isomor \cL_{k}(\widehat{R^{\prime}}^{\cross}_{\frakn^{+} v}, \F)_{\m}$$
under the {\PO} assumption. Therefore the conclusion follows from "Ihara's lemma".
\end{proof}

%add more details to the above proof, add more precise things for Hida theory
Let $\frakl_{1}, \frakl_{2}$ be two $n$-admissible primes such that $\frakl_{1}\frakl_{2}\nmid \Delta_{B}$. We denote the underlying rational prime in $\Q$ of $\frakl_{2}$ by $l_{2}$. Let ${\cM}^{[\frakl_{1}]}_{n}$ be the integral model of $M^{[\frakl_{1}]}_{n}$ over $\cO_{F,\frakl_{2}}$ studied by Carayol \cite{Carayol:Bad} and also in Zhang \cite{Zhang:Height}. We review the definition of this model in preparation for the next theorem. Let $F^{\prime}$ be an auxiliary CM field over $F$ such that $\frakl_{2}$ is split in $F^{\prime}$. We set $D=B^{\prime}\otimes F^{\prime}$ and denote $\upsilon(\openU^{\prime})= {N}_{D}(U^{\prime}\widehat{F^{\prime}}^{\cross})\cap \cO^{\cross}_{F}$.  We fix an embedding $\iota_{\frakl_{2}}: F^{\prime}\hookrightarrow F_{\frakl_{2}}$.  For an $\cO_{D,l_{2}}$ module $\Lambda$, there is an natural decomposition  $\Lambda=(\oplus_{i=1}^{r}\Lambda_{1}^{i})\oplus (\oplus_{i=1}^{r}\Lambda_{2}^{i})$ compatible with respect to the natural decomposition $\cO_{D,l_{2}}=(\oplus_{i=1}^{r}\cO_{D,1}^{i})\oplus (\oplus_{i=1}^{r}\cO_{D,2}^{i})$ where we make a choice such that $\cO_{D,\frakl_{2}}$ corresponds to $\cO_{D,2}^{1}$. Note that as $D$ splits at $\frakl_{2}$, we can set $\Lambda^{1,1}_{2}=\pMX{1}{0}{0}{0}\Lambda^{1}_{2}$.
Let $\widehat{V}$ be the profinite completion of an $\cO_{D}$-module $V$ and $\widehat{V}^{(l_{2})}$ be the subset of $\widehat{V}$ with trivial $l_{2}$ component. We further denote by $V^{(\frakl_{2})}_{2}=\oplus^{r}_{i=2}V_{l_2,2}^{i}$.  These notations in particular applies to the case when $V= \cO_{D}$ viewed as an $\cO_{D}$-module via the natural left multiplication, $V =H_{1}(A, {\Z})$ or $\Lie(A)$ for an abelian variety $A$ with $\cO_{D}$ action. To define the integral model ${\cM}^{[\frakl_{1}]}_{n}$, one first constructs an auxiliary integral model ${\tilde{\cM}}^{[\frakl_{1}]}_{n}$ for another Shimura curve ${\tilde{M}^{[\frakl_{1}]}}_{n}$ which admits a finite map $f: M^{[\frakl_{1}]}_{n} \rightarrow \tilde{M}^{[\frakl_{1}]}_{n}$. Then one takes ${\cM}^{[\frakl_{1}]}_{n}$ to be the normalization of $\tilde{\cM}^{[\frakl]}_{n}$.

Let $x\in {\tilde{\cM}}^{[\frakl_{1}]}_{n}(S)$ for a scheme $S$ over $\cO_{F,\frakl_{2}}$.  Then $x=(A/S, \iota, \bar{\theta}, \bar{\kappa}^{[\frakl_{2}]})$ where 
\begin{itemize}
\item {$A/S$ is an abelian scheme of dimension $4d$  with an action $\iota: \cO_{D}\rightarrow \End{A}$ satisfying certain conditions on $\Lie(A)$: $\Lie(A)^{1}_{2}$ is a locally free $\cO_{S}$-module of rank $2$ and  $\Lie(A)^{\frakl_{2}}_{l_{2}}=0$ \cf \cite[Proposition 1,22]{Zhang:Height};}
\item{$\bar{\theta}: A\rightarrow \dual{A}$ is an $\upsilon(\openU^{\prime})$-isomorphism class of polarizations of $A$ with prime to $l_{2}$ degree;}
\item{$\bar{\kappa}^{[\frakl_{2}]}: {T}^{(\frakl_{2})}_{l_{2}}(A)\oplus \widehat{T}^{(l_{2})}(A)\rightarrow W^{(\frakl_{2})}_{l_{2}}\oplus \widehat{W}^{(l_{2})}$ is an isomorphism up to the action of $\openU^{\prime}Y$.} 
\end{itemize}

Let $A$ be the universal abelian scheme over ${\tilde{\cM}}^{[\frakl_{1}]}_{n}$ and $A_{l_{2}^{\infty}}$ be the $l_{2}^{\infty}$-divisible group of $A$, then $(A_{l_{2}^{\infty}})^{1,1}_{2}$ is a divisible $\cO_{F,{\frakl_{2}}}$-module of height $2$ and dimension $1$. We denote by $\mathbf{E}_{\infty}$ the pullback of $(A_{l_{2}^{\infty}})^{1,1}_{2}$ to $\cM^{[\frakl_{1}]}_{n}$. Let $x$ be a geometric point of ${\cM}^{[\frakl_{2}]}_{n}$, then the fiber $\mathbf{E}_{\infty,x}$ is either of the form $F_{\frakl_{2}}/\cO_{F,\frakl_{2}}\times \Sigma_{1}$ or $\Sigma_{2}$ where $\Sigma_{1}$ is the formal $\cO_{F,\frakl_{2}}$ module of height $1$ and $\Sigma_{2}$ is the  formal $\cO_{F,\frakl_{2}}$ module of height $2$. If the first case happens, then we say $x$ is ordinary and if the second case happens then we say $x$ is supersingular. See  \cite[Section 0.8]{Carayol:Bad}.

We fix an isomorphism $\varphi_{B^{\prime\prime},B^{\prime}}: {\widehat{B}^{\prime\prime(\frakl_{2})\cross}}\rightarrow {\widehat{B}^{\prime(\frakl_{2})\cross}}$ and put $\openU^{\prime\prime}= \varphi^{-1}_{B^{\prime\prime}, B^{\prime}}(\openU^{\prime(\frakl_{2})})\cO^{\cross}_{B^{\prime\prime},\frakl_{2}}$. By \cite[Section 11.2]{Carayol:Bad}, we have an identification
$$\gamma: X_{B^{\prime\prime}}(\openU^{\prime\prime}Y)\mapR {\cM}^{[\frakl_{1}]}_{n}(\F_{\frakl^{2}_{2}})^{ss}$$
where the righthand side is the supersingular locus of the special fiber of ${\cM}^{[\frakl_{1}]}_{n}$ and $\F_{\frakl^{2}_{2}}$ is the residue field of $\cO_{K,\frakl_{2}}$.

\begin{lm}
We have $\text{red}_{\frakl_{2}}(P_{m}(a))\in  {\cM}^{[\frakl_{1}]}_{n}(\F_{\frakl^{2}_{2}})^{ss}$ and moreover $\text{red}_{\frakl_{2}}(P_{m}(a))=\gamma(x_{m}(a)\tau_{n})$.
\end{lm}
\begin{proof}
The first part is proved by observing that for a geometric point $x$ underlying $\text{red}_{\frakl_{2}}(P_{m}(a))$ the  formal $\cO_{F,\frakl_{2}}$ module $\mathbf{E}_{\infty,x}$ admitts an $F_{\frakl_{2}}$ linear action of $K_{\frakl_{2}}$. Since $\frakl_{2}$ is inert, it follows that the connected part $\mathbf{E}^{0}_{\infty,x}$ of $\mathbf{E}_{\infty,x}$ has to be $\Sigma_{2}$. 

By \cite[Section 11.2]{Carayol:Bad}, the action of $\widehat{B}^{\prime\prime\cross}$ is transitive on ${\cM}^{[\frakl_{1}]}_{n}(\F_{\frakl^{2}_{2}})^{ss}$ and so $\gamma$ can be chosen such that $\text{red}_{\frakl_{2}}(P)=\gamma([1])$ for some CM point $P$ and $\text{red}_{\frakl_{2}}(P_{m}(a))=x_{m}(a)\tau_{n}(\text{red}_{\frakl_{2}}(P))$.
\end{proof}

\begin{thm}[Second reciprocity law]
Under the assumption $(\CR)$, $\nimpr$, $\PO$ and "Ihara's lemma", there exists an $n$-admissible form $\cD^{\prime\prime}=(\Delta_{B}\frakl_{1}\frakl_{2}, g^{\prime\prime})$ such that the following equalities hold

$$v_{\frakl_{1}}(\kappa_{\cD}(\frakl_{2}))=v_{\frakl_{2}}(\kappa_{\cD}(\frakl_{1}))=\theta_{\infty}(\cD^{\prime\prime})\in \cO_{f,n}[[\Gamma]].$$
\end{thm}

\begin{proof}
Let $\cJ^{[\frakl_{1}]}_{n}=\Pic({\cM}^{[\frakl_{1}]}_{n})$. Then $\gamma$ induces a map

$$\gamma_{*}: \Z[X_{B^{\prime\prime}}(\openU^{\prime\prime}Y)]\mapR \cJ^{[\frakl_{1}]}_{n}(k_{\frakl^{2}_{2}})_{\cO_{f}}/\cI^{[\frakl_{1}]}_{g}\mapR H^{1}_{\fin}(K_{\frakl_{2}}, T_{f,n})\isomor \cO_{f,n}.$$
By \cite[Proposition 11.9]{Van_order:main} or \cite[Lemma 7.20]{Longo_IMC} based on Ihara's general theorem \cite[Section 3 (G)]{Ihara:shimura_curve} and \propref{Ihara_remedy}, $\gamma_{*}$ is in fact  surjective and hence gives rise to a modular form $g^{\prime\prime}\in S^{B^{\prime\prime}}_{2}(\openU^{\prime\prime} Y, \cO_{f,n})$. Moreover it is an eigenform  for the Hecke algebra $\Hecke_{B^{\prime\prime}}(\nplus Y, \frakp^{n})$ such that its Hecke eigenvalues are given by $\lambda^{\prime\prime}: \Hecke_{B^{\prime\prime}}(\nplus Y, \frakp^{n})\mapR \cO_{f,n}$ which is determined by $\lambda^{\prime\prime}(T_{v})=\lambda_{f}(T_{v}) \pmod{\uf^{n}}$ for $v\nmid \frakp\frakn\frakl_{1}\frakl_{2}$, $\lambda^{\prime\prime}(U_{v})=\lambda_{f}(U_{v}) \pmod{\uf^{n}}$ for $v\mid \frakp\frakn$ and $\lambda^{\prime\prime}(U_{\frakl_{i}})=\epsilon_{\frakl_{i}}$ for $i\in\{1,2\}$. Such an eigenform is unique up to a unit in $\cO_{f,n}$ again by the multiplicity one result \thmref{TW}.

It follows from the above discussion that
\begin{equation}
\begin{aligned}
v_{\frakl_{2}}(\kappa_{\cD}(\frakl_{1})_{m}) &=\sum_{[a]_{m}\in G_{m}}\<x_{m}(a)\tau_{n}, g^{\prime\prime}\>\pi([a]_{m})\\
&=\sum_{[a]_{m}\in G_{m}} g^{\prime\prime}(x_{m}(a))\pi([a]_{m})\\
&=\theta_{m}(\cD^{\prime\prime}).\\
\end{aligned}
\end{equation}
By switching the role of $\frakl_{1}$ and $\frakl_{2}$, we get the conclusion of the theorem.
\end{proof}

\section{Euler system argument}\label{ES_argument}

In this chapter we carry out the Euler system argument to bound the Selmer group. The method employed here is originally due to \cite{Bertolini_Darmon:IMC_anti}. The main result \thmref{main_weak} is proved under the assumptions:  $(\CR)$, $\nmin$ and $\PO$. The assumption $\nmin$ is stronger than the desired assumption $\nimpr$  but we will later relax it to $\nimpr$ in chapter 10. 
\subsection{Special elements in the Galois representation}
Throughout this section we will assume the assumptions  $\PO$, $(\CR)$ and $\nmin$. Let $H=\bar{\rho}^{*}_{f}(G_{F})$ be the image of the residual Galois representation.  We examine the existence of certain special elements in the image. We have the following lemma.

\begin{lm}\label{-1}
The group $H$ contains 
\begin{enumerate}
\item{the scalar matrix $-1$;} 
\item{ an element $h\in H$ such that  $\Tr(h)=\det(h)+1$ with $\det(h)\ne\pm1\in \F_{\frakp}$.} 
\end{enumerate}
\end{lm}
\begin{proof}
For (1), let $Z$ be the center of $H$. From $(\CR)(1)$, we note that $Z$ is contained in the set of scalar matrices. If $p\mid \#H$, then it is well-known that $H$ contains a conjugate of $\SL_{2}(\F_{p})$ and the result follows. So we assume that $H$ has order prime to $p$ but then $H/Z$ has the following possibilities:

$H/Z$ is cyclic. This is impossible in view of $(\CR)(1)$.

$H/Z$ is the three exceptional groups $A_{4}, S_{4}, A_{5}$. However if we let $I_{\frakp}$ be the image of the inertia at $\frakp$ in $H/Z$, then we know it is a cyclic group of order $>5$ by the assumption that $\#( \F^{\cross}_{\frakp})^{k-1}>5$. Then these possibilities can not happen.

$H/Z$ is the dihedral group. Since $\#( \F^{\cross}_{\frakp})^{k-1}$ is even, $I_{\frakp}$ is contained in a cyclic group of even order and hence there is an element of order $2$ in $H/Z$. We can then conclude that there is an element of order $2$ in $Z$.

For (2), if $p\nmid \Nm{\frakl}^{2}-1$ for some $\frakl\mid \nminus$, then we can simply choose $\Frob_{\frakl}$ as our element $h$. If this is not the case, then by our assumption $\CR(3)$, $\bar{\rho}^{*}_{f}(G_{F})$ is ramified at $\frakl$ and hence $H$ contains the image of the inertia which is a subgroup of order divisible by $p$. Therefore $H/Z$ contains the group $\PSL_{2}(F)$ for some finite extension $F/\F_{p}$ and is contained in $\PGL_{2}(F)$. Therefore $[\SL_{2}(F), \SL_{2}(F)]=\SL_{2}(F)$ is contained in $H$. We take $C=\GL_{2}(F)\cap H$. Then $C$ sits in between $\SL_{2}(F)$ and $\GL_{2}(F)$. Let $\gamma\in H$ be the image of a generator of the inertia group at $\frakp$. Then we have $\det(\gamma)=u$ generates $\F^{\cross}_{\frakp}$ and moreover $\Tr(\gamma)=u^{k/2}+u^{1-k/2}\in \F_{\frakp}$ is non-zero. By the previous discussion, we have $H\subset \bar{\F}_{\frakp}\GL_{2}(F)$. We conclude from $\Tr(\gamma)=u^{k/2}+u^{1-k/2}\in \F_{\frakp}$ that $\gamma\in C$. Thus $C\supset\F_{\frakp}SL_{2}(F)$. The result follows as then $H$ contains $\GL_{2}(\F_{\frakp})$.
\end{proof}

As a consequence of the above two lemmas, we prove the following key theorem for our Euler system argument.

\begin{thm}\label{admissible}
Let $t\leq n$ be a postive integer. Let $\kappa\in H^{1}(K, A_{f,t})$ be a non-zero element, then there exist infinitely many $n$-admissible primes $\frakl$ such that $\partial_{\frakl}(\kappa)=0$ and $v_{\frakl}:\<\kappa\>\mapR H^{1}_{\fin}(K_{\frakl}, A_{f,t})$ is injective. Here $\<\kappa\>$ is the $\cO_{f,t}$-submodule of $H^{1}(K, A_{f,t})$ generated by $\<\kappa\>$.
\end{thm}
\begin{proof}
First of all, since $\frakl$ is an $n$-admissible prime, $H^{1}_{\fin}(K_{\frakl}, A_{f,1})\hookrightarrow H^{1}_{\fin}(K_{\frakl}, A_{f,t})$ is injective. Thus without loss of generality we can assume that $t=1$ and $\kappa\in H^{1}(K, A_{f,1})$. 

Let $L_{n}=\bar{F}^{\ker(\rho_{n}) }$ be the field cut out by $\rho_{n}:=\rho^{*}_{f}\pmod{\uf^{n}}$. We remark that since $D_{K/F}$ is prime to $\frakn p$ and $\rho_{n}$ is unramified away form $\frakn p$, the field $L_{n}$ is linearly disjoint from $K$. 

We put $M=L_{n}K$, then $\Gal(M/F)=\Gal(K/F)\times \Gal(L_{n}/F)$ which is of  the form $\<\tau\>\times \Aut(A_{f,n})$. By \lmref{-1}(1), $\Gal(L_{n}/F)$ contains a subgroup of the form $\<\pm1\>$.  Then by the Hochschild-Serre spectral sequence $$H^{r}(\Gal(L_{n}/F)/\<\pm1\>, H^{s}(\<\pm1\>, A_{f,n}))\imply H^{r+s}(L_{n}/F, A_{f,n}),$$ we see $H^{i}(L_{n}/F, A_{f,n})=H^{i}(M/K, A_{f,n})=0$ for all $i$. Thus  $$H^{1}(K, A_{f,1})\mapR H^{1}(M, A_{f,1})$$is injective . We identify ${\kappa}$ with its non-zero image $\kappa\in H^{1}(M, A_{f,1} )=\Hom(G_{M}, A_{f,1})$. Let $M_{\kappa}$ be the field cut out by ${\kappa}$ over $M$ and $A_{\kappa}:= {\kappa}(\Gal(M_{\kappa}/M))\subset A_{f,1}$ be a $\F_{p}[\Gal(M/F)]$ submodule of $A_{f,1}$. Notice $A_{\kappa}$ is of dimension 2 by \eqref{Irr}.

 Without loss of generality,  we can assume $\tau \kappa=\delta \kappa$ for some $\delta\in \{\pm1\}$. Then $M_{\kappa}/F$ is Galois and we can identify $\Gal(M_{\kappa}/F)$ with $A_{\kappa}\rtimes \Gal(M/F)$ where $(\tau^{j}, \sigma)\in \Gal(M/F)$ acts on $v\in A_{\kappa}$ by $(\tau^{j}, \sigma)v=\delta^{j}\rho_{1}(\sigma)v$.
 
 By \lmref{-1}, we can find an element $({\bf v} ,\tau, \sigma)\in \Gal(M_{\kappa}/ F)\isomor A_{\kappa}\rtimes (\Gal(K/F)\times \Gal(L_{n}/F))$ such that 
 
\begin{enumerate}
\item{$\rho_{n}(\sigma)$ has eigenvalues $\delta$ and $\lambda$ in $(\cO_{f,n})^{\cross}$ where $\lambda\ne\pm1$ and has order prime to $p$;}
\item{ ${\bf v}$ is in the $\delta$-eigenspace of $\sigma$.}
\end{enumerate}

By the Chebotarev density theorem, there exist infinitely many unramified $\frakl$ with $\frakl\nmid \frakn$ such that $\Frob_{\frakl}(M_{\kappa}/ F)=({\bf v},\tau,\sigma)$. Then it follows that $\frakl$ is $n$-admissible for $f$ as we have $\Frob_{\frakl}(M/F)=(\tau,\sigma)$. Let $\mathfrak{L}$ be  any place above $\frakl$ in $M$. Then we have $\Frob_{\mathfrak{L}}(M_{\kappa}/M)= ({\bf v},\tau, \sigma)^{d}= {\bf v}+\delta \sigma {\bf v}\cdots +\delta\sigma^{d-1}{\bf v}=d{\bf v}$ where $d$ is the residue degree of  $\mathfrak{L}$ over $\frakl$ which is necessarily even. Hence ${\kappa}(\Frob_{\frakL}(M_{\kappa}/M))=d {\kappa}({\bf v})\ne 0$ and $v_{\frakl}(\kappa)=({\kappa}(\Frob_{\frakL}(M_{\kappa}/M)))_{\frakL\mid \frakl}$ is not zero. 
\end{proof}

Let $\Delta$ be a square free product  of primes in $F$ such that $\Delta/\nminus$ is a  product of n-admissible primes.

\begin{defn}[$n$-admissible set]
A finite set $S$ of primes in $F$ is said to be $n$-admissible for $f$ if:
\begin{enumerate}
\item{$S$ consists of $n$-admissible primes of $f$.}
\item{The map $\Sel_{\Delta}(K, T_{f,n})\rightarrow \oplus _{\frakl\in S}H^{1}_{\fin}(K_{\frakl}, T_{f,n})$ is injective.}
\end{enumerate}
\end{defn}

\begin{remark}\label{enlarge}
It follows from \thmref{admissible} that any collection of $n$-admissible primes can be enlarged to an $n$-admissible set.
\end{remark}

\subsection{Freeness of Selmer groups}

Let $S$ be an $n$-admissible set for $f$ and $L/K$ be a finite extension contained in $K_{\infty}/K$. In this section all the results are proved under the assumptions $(\CR)$, $\nmin$, and ${\PO}$.
\begin{lm}
The map $\Sel_{\Delta}(L, T_{f,n})\mapR \oplus_{\frakl\in S}H^{1}_{\fin}(L_{\frakl}, T_{f,n})$ is injective.
\end{lm}

\begin{proof}
Let $C$ be the kernel of this map. Since $\Gal(L/K)$ is a $p$-group, there exists a non-trivial fixed point of its action. It follows from \propref{control} that this fixed point is in the kernel of $\Sel_{\Delta}(K, T_{f,n})\mapR \oplus_{\frakl\in S}H_{\fin}^{1}(K_{\frakl}, T_{f,n})$ which is a contradiction to the $n$-admissibility of $S$.
\end{proof}
We have the following Poitou-Tate exact sequence.

\begin{lm}\label{Poitou-Tate}
$$0\mapR \Sel_{\Delta}(L, T_{f,n})\mapR \Sel^{S}_{\Delta}(L, T_{f,n})\mapR \oplus_{\frakl\in S} H^{1}_{\sing}(L_{\frakl}, T_{f,n})\mapR \dual{\Sel_{\Delta}(L, A_{f,n})}\mapR 0.$$
\end{lm}

\begin{proof}
This is standard, see \cite[Theorem 1.7.3]{Rubin:book}. The right exactness follows from the $n$-admissibility of $S$.
\end{proof}

\begin{prop}\label{free}
The Selmer group $\Sel^{S}_{\Delta}(L, T_{f,n})$ is a free $\cO_{f,n}[\Gal(L/K)]$-module of rank $|S|$.
\end{prop}

\begin{proof}
Let $L=K_{m}$ for some $m$ large enough, and let $\frakm_{n,m}$ be the maximal ideal of $\cO_{f,n}[\Gal(K_{m}/K)]$.  Notice $\cO_{f,n}[\Gal(K_{m}/K)]$ is of the form $$\cO_{f,n}[x_{1},\cdots, x_{d}]/(x_{1}^{p^{m}}-1,\cdots, x_{d}^{p^{m}}-1)$$ which implies it is local complete intersection and thus Gorentein.  

By \lmref{control}, we have $\Sel^{S}_{\Delta}(K, T_{f,1})=\Sel^{S}_{\Delta}(K_{m}, T_{f,n})[\frakm_{n,m}]$. Hence $$\dual{\Sel^{S}_{\Delta}(K_{m}, T_{f,n})}/\frakm_{n,m}\isomor \dual{\Sel^{S}_{\Delta}(K, T_{f,1})}$$ where the latter module is free of rank $|S|$ over $\cO_{f,1}$ by \lmref{Poitou-Tate}. Then by Nakayama's lemma, $\dual{\Sel^{S}_{\Delta}(K_{m}, T_{f,n})}$ has a basis of $|S|$ elements over $\cO_{f,n}$. But by \lmref{Poitou-Tate} and a cardinality consideration, $\dual{\Sel^{S}_{\Delta}(K_{m}, T_{f,n})}$ is free of rank $|S|$. By the Gorensteiness of $\cO_{f,n}[\Gal(K_{m}/K)]$, it follows that $\Sel^{S}_{\Delta}(K_{m}, T_{f,n})$ is also free of rank $|S|$.
\end{proof}

We summarize some easy consequence of the above lemma in the following corollary. 
\begin{cor}\label{control2}
Let $S$ be an $n$-admissible set, then 
\begin{enumerate}
\item{$\widehat{\Sel}^{S}_{\Delta}(K_{\infty}, T_{f,n})$ is free of rank $\#S$ over $\Lambda/\uf^{n}$.}
\item{$\widehat{\Sel}^{S}_{\Delta}(K_{\infty}, T_{f,n})/\frakm_{\Lambda}= {\Sel}^{S}_{\Delta}(K, T_{f,1})$.}
\item{$\widehat{\Sel}^{S}_{\Delta}(K_{\infty}, T_{f,n})/\uf{\Lambda}= \widehat{\Sel}^{S}_{\Delta}(K_{\infty}, T_{f,1})$.}
\end{enumerate}
\end{cor}

\subsection{Divisiblity of the Euler system and Selmer group}
Let  $\xi:\Lambda\mapR \cO_{\xi}$ be an $\cO_{f}$-algbera homomorphism, where $\cO_{\xi}$ is a discrete valuation ring of characteristic $0$. Let $\uf_{\xi}$ be a uniformizer of $\Oxi$. Let $M$ be an $\Oxi$-module. For each $x\in M$, we denote by $\ord_{\uf_{\xi}}(x)=\sup\{m\in \Z\cup\{\infty\}\mid x\in \uxi^{m}M\}$ the divisibility index of $x$. Note $x=0$ if and only if $\ord_{\uxi}(x)=\infty$. 

For a finitely generated $\Lambda$-module, we will write $\text{char}_{\Lambda}(M)$ the characteristic ideal of $M$. The following result original due to Bertonili-Darmon \cite{Bertolini_Darmon:IMC_anti} and in our case due to Longo \cite{Longo_IMC} gives a criterion of identifying elements in $\text{char}_{\Lambda}(M)$.

\begin{lm}\label{criterion}
Let $M$ be a finitely generated $\Lambda$-module and $L\in\Lambda$. Suppose that for any homomorphism $\xi: \Lambda\mapR \cO_{\xi}$, we have $\length_{\cO_{\xi}}(M\otimes_{\xi}\cO_{\xi})\leq \ord_{\uf_{\xi}}(\xi(L))$. Then $L\in \text{char}_{\Lambda}(M)$.
\end{lm}

\begin{proof}
It follows from the statement of this Lemma that $\xi(L)\in \Fitt_{\cO_{\xi}}(M_{\xi})$ for all $\xi$, then the lemma follows from \cite[lemma 7.4]{Longo_IMC}.
\end{proof}

Let $n$ be a positive integer and $\Delta$ be a square-free product of primes in $F$ as in \defref{admissible_form}. For any $\xi:\Lambda\rightarrow \cO_{\xi}$, we define two integers for each $n$-admissible form $\cD=(\Delta, f_{n})$:
\begin{equation}
\begin{aligned}
& s_{\cD}=\length_{\cO_{\xi}} \dual{\Sel_{\Delta}(K_{\infty}, A_{f,n})}\otimes_{\xi}\cO_{\xi};\\
& t_{\cD}=\ord_{\uxi} \xi(\theta_{\infty}(\cD))\text{ where $\xi(\theta_{\infty}(\cD))\in \cO_{f,n}[[\Gamma]]\otimes_{\xi}\cO_{\xi}=\cO_{\xi}/\uxi^{n}$}.\\
\end{aligned}
\end{equation}

We will prove the following analogue of \cite[Proposition 4.3]{Pollack_Weston:AMU} using an inductive argument. The one divisibility result will immediately follow from it.

\begin{thm}\label{induction}
Assume $(\CR)$, $\nmin$, $\PO$ and "Ihara's lemma" hold. Let $t_{0}\leq n$ be a non-negative ineteger and let $\cD_{n+t_{0}}=(\Delta, f_{n+t_{0}})$ be an $(n+t_{0})$-admissible form. We set $\cD_{n}=\cD_{n+t_{0}}\pmod {\uf^{n}}=(\Delta, f_{n+t_{0}}\pmod{\uf^{n}})$. Suppose that $t_{\cD_{n}}\leq t_{0}$, then $s_{\cD_{n}}\leq 2t_{\cD_{n}}$.
\end{thm}

\subsubsection{Proof of the \thmref{induction}}
We will prove this theorem by induction on $t_{\Dnton}$.  If $t_{\Dnton}=\infty$ or $s_{\Dnton}=0$, then the claim is trivially true. So we consider the case when 
\begin{equation}\label{trivial}
\begin{aligned}
& t_{\Dnton}<\infty  \imply \xi(\theta_{\infty}(\Dnton))\ne 0;\\
& s_{\Dnton}>0 \imply  \Sel_{\Delta}(K_{\infty}, A_{f,n})\otimes_{\xi}\cO_{\xi}\ne \{0\}.\\
\end{aligned}
\end{equation}

From now on we will denote
$$t=t_{\cD_{n}},$$ 
$$s=s_{\cD_{n}}.$$

Let $\Sxi_{n}=\cSel^{S}_{\Delta}(K_{\infty}, T_{f,n})\otimes_{\xi}\cO_{\xi} $ and  for some $n$-admissible prime $\frakl$, let $\kappa_{\cD_{n+t},\xi}(\frakl)\in\Sxi_{n+t}$ be the image of $\kappa_{\cD_{n+t}}(\frakl)$. We define an integer $e_{\cD_{n+t}}(\frakl)$ by $$e_{\cD_{n+t}}(\frakl)= \ord_{\uf_{\xi}}(\kappa_{\cD_{n+t},\xi}(\frakl)).$$ By the first reciprocity law in\thmref{1law}, we have 
\begin{equation}\label{e<t}
e_{\cD_{n+t}}(\frakl)\leq \ord_{\uf_{\xi}}(\partial_{\frakl}(\kappa_{\cD_{n+t},\xi}(\frakl)))= \ord_{\uf_{\xi}}(\xi(\theta_{\infty}(\cD_{n+t})))=\ord_{\uxi}(\xi(\theta_{\infty}(\Dnton)))=t.
\end{equation}

Let $\tka_{\cD_{n+t}, \xi}(\frakl)$ be such that $\uxi^{e_{{\cD_{n+t}}}(\frakl)}\tka_{\cD_{n+t},\xi}{(\frakl)}=\kappa_{\cD_{n+t},\xi}(\frakl)$ and let $\kappa^{\prime}_{\cD_{n+t},\xi}(\frakl)$ be the  image of $\tilde{\kap}_{\cD_{n+t}}(\frakl)$ under the map $\Sxi_{n+t}\mapR\Sxi_{n}$.

\begin{lm} \label{trivial}
We have
\begin{enumerate}
\item{$\ord_{\uxi}(\pkaDxil)=0$,}
\item{$\ord_{\uxi}(\partial_{\frakl}(\pkaDxil))=t-e_{\cD_{n+t}}(\frakl)$,}
\item{$\partial_{v}(\pkaDxil)=0\text{ for all $v\nmid \Delta \frakl$}$,}
\item{$\res_{v}(\pkaDxil)\in \Hhat^{1}_{\ord}(K_{\infty,v}, T_{f,n})$ for all $v\mid \Delta\frakl$.}
\end{enumerate}
\end{lm}

\begin{proof}
(1) follows from the control theorem of Selmer groups \corref{control2}: $\Sxi_{n+t}/\uf_{\xi}\rightarrow\Sxi_{n}/\uf_{\xi}$ is an isomorphism and hence $\ord_{\uxi}(\pkaDxil)=\ord_{\uxi}(\tka_{\cD_{n+t},\xi}(\frakl))=0$.

(2) follows from the first reciprocity law in \thmref{1law}. (3) and (4) follow from the definition of the relevant Selmer group.
\end{proof}

\begin{lm}\label{vanish}
Let $\eta_{\frakl}: \Hhat^{1}_{\sing}(K_{\infty,\frakl}, T_{f,n})\otimes_{\xi}\cO_{\xi}\mapR\dual{\Sel_{\Delta}(K_{\infty}, A_{f,n})}\otimes_{\xi}\cO_{\xi}$ be the map defined by $\eta_{\frakl}(c)(x)=\<c, v_{\frakl}(x)\>$ for $x\in {\Sel_{\Delta}(K_{\infty}, A_{f,n})[\ker{\xi}]}$ and $c\in \Hhat^{1}_{\sing}(K_{\infty,\frakl}, T_{f,n})$. Then $\eta_{\frakl}(\partial_{\frakl}(\pkaDxil))=0.$
\end{lm}

\begin{proof}
Note that $\dual{\Sel_{\Delta}(K_{\infty}, A_{f,n})}\otimes_{\xi}\cO_{\xi}=\dual{\Sel_{\Delta}(K_{\infty}, A_{f,n})[\ker\xi]}$. Consider the sum $\sum_{\frakq}\<\res_{\frakq}(\pkaDxil), \res_{\frakq}(x)\>=0$ by the global class field theory.

Notice that for $\frakq\nmid \Delta\frakl$, $\partial_{\frakq}(\pkaDxil)=0$, and for $\frakq \mid \Delta$, both $\res_{\frakq}(x)$ and $\res_{\frakq}(\pkaDxil)$ are ordinary. It follows that the only possible non-trivial term in the above sum is  $\eta_{\frakl}(\res_{\frakl}(\pkaDxil))=\<\res_{\frakl}(\pkaDxil), \res_{\frakl}(x)\>$. This shows that this term itself is zero.
\end{proof}

\begin{lm}[Base case]\label{Base-case}
If $t=0$, then $s=0$.
\end{lm}

\begin{proof}
If $t=\ord_{\uxi} \xi(\theta_{\infty}(\Dnton))=0$(in this case $e_{\cD_{n+t}}(\frakl)=t=0$ by \eqref{e<t} and $\pkaDxil=\kaDntonxil$), then $\xi(\theta_{\infty}(\Dnton))$ is a unit and hence by the first reciprocity law, for any admissible prime $\frakl$, $\partial_{\frakl}(\kaDntonxil)$ spans the rank one space $\Hhat^{1}_{\sing}(K_{\infty,\frakl}, T_{f,n})\otimes_{\xi}\cO_{\xi}$. It follows from \lmref{vanish} that the map $\eta_{\frakl}$ is zero. Now if $s\ne0$,  $\dual{\Sel_{\Delta}(K_{\infty}, A_{f,n})}\otimes_{\xi}\cO_{\xi}$ is non-trivial and by Nakayama's lemma $$(\dual{\Sel_{\Delta}(K_{\infty}, A_{f,n})}\otimes_{\xi}\cO_{\xi})/\frakm_{\xi}=\dual{(\Sel_{\Delta}(K_{\infty}, A_{f,n})[\frakm_{\Lambda}])}\otimes_{\xi}\cO_{\xi}$$ is nontrivial. By \thmref{admissible} we can choose an admissible prime $\frakl$ and a nontrivial element $c\in \Sel_{\Delta}(K_{\infty}, A_{f,n})[\frakm_{\Lambda}]=\Sel_{\Delta}(K, A_{f,1})$  such that $v_{\frakl}(c)$ is non-zero.  It then follows that $\eta_{\frakl}(\partial_{\frakl}(\pkaDxil))(c)$ is non-trivial which is a contradiction to the previous lemma.
\end{proof}

Now we assume that $t>0$. Let $\Xi$ be the set of primes $\frakl$ such that  
\begin{enumerate}
\item{$\frakl$ is an $n+t_{0}$-admissible and $\frakl\nmid \Delta$;}
\item{the integer $e_{\cD_{n+t}}(\frakl)$ is minimal for the set of $\frakl$ in (1).}
\end{enumerate}
 
Let $e=e_{\cD_{n+t}}(\frakl)$ for any $\frakl$ satisfies the above conditions.

\begin{lm}
We have $e< t$.
\end{lm}

\begin{proof}
As we already know $e\leq t$, we assume $e=t$ and derive a contradiction. Since we assumed $s>0$, the Selmer group $\dual{\Sel_{\Delta}(K_{\infty}, A_{f,n})}\otimes_{\xi}\cO_{\xi}$ is non-trivial. We may use the same argument as in \lmref{Base-case} to conclude that there is a non-trivial class $x\in \Sel_{\Delta}(K_{\infty}, A_{f,n})[\frakm_{\Lambda}]\subset H^{1}(K, A_{f,1})$. Then by \thmref{admissible}, there are $n+t_{0}$ admissible primes $\frakl$ such that  $v_{\frakl}(x)\ne 0$. Now by the assumption $e=t$ and \lmref{trivial}, we have $\partial_{\frakl}(\pkaDxil)$ is indivisible and hence its image is non-zero in $H^{1}(K, A_{f,1})\otimes_{\xi}\cO_{\xi}$. But this is a contradiction to \lmref{vanish}.
\end{proof}

We choose $\frakl_{1}\in \Xi$ and enlarge it to an $(n+t_{0})$-adimissible set. Let $\kappa_{1}$ be the image of $\pkaDxi(\frakl_{1})$ under the map
$$\cSel^{S}_{\Delta}(K_{\infty}, T_{f,n})\otimes_{\xi}\cO_{\xi}/\uf_{\xi}\hookrightarrow H^{1}(K, T_{f,1})\otimes_{\xi}\cO_{\xi}.$$ 

By \lmref{trivial}(1), $\kappa_{1}$ is nontrivial viewing it as an element in $H^{1}(K, T_{f,1})\otimes_{\xi}\cO_{\xi}$. So by \thmref{admissible}, we can pick an $(n+t_{0})$-admissible prime $\frakl_{2}$ such that $v_{\frakl_{2}}(\kap_{1})\ne 0\in H^{1}(K_{\frakl_{2}}, T_{f,1})\otimes_{\xi}\cO_{\xi}$.

It follows from above choice that $$\oxi(v_{\frakl_{2}}(\kaDxi(\frakl_{1})))=\oxi(\kaDxi(\frakl_{1}))\leq\oxi(\kaDxi(\frakl_{2}))\leq\oxi(v_{\frakl_{1}}(\kaDxi(\frakl_{2})))$$ where the middle inequality is due to the fact that $\frakl_{1}$ is an element in $\Xi$.

From the second explicit reciprocity law, we can find an $n+t_{0}$-admissible form $\cD^{\prime\prime}_{n+t_{0}}=(\Delta\frakl_{1}\frakl_{2}, g_{n+t_{0}})$ such that

$$v_{\frakl_{2}}(\kap_{\cD_{n+t}}(\frakl_{1}))=v_{\frakl_{1}}(\kap_{\cD_{n+t}}(\frakl_{2}))=\theta_{\infty}(\cD^{\prime\prime}_{n+t})$$ where $\Dpp_{n+t}=\Dpp_{n+t_{0}} \pmod{\uf^{n+t}}$. This implies that the above sequence of inequalities are all in fact equalities and in particular that $e=e_{\cD_{n+t}}(\frakl_{1})=\oxi (\kaDxi(\frakl_{1}))= e_{\cD_{n+t}}(\frakl_{2})=\oxi (\kaDxi(\frakl_{2})).$ It then follows that $\frakl_{2}\in \Xi$. Let $\Dpp_{n}=(\Delta\frakl_{1}\frakl_{2}, g_{n+t_{0}}\pmod{\uf^{n}})$, then we have 
$$t_{\Dpp_{n}}=\ord_{\uxi}(\xi(\theta_{\infty}(\Dpp_{n})))=\oxi(v_{\frakl_{2}}(\kaDxi(\frakl_{1})))=\oxi(\kaDxi(\frakl_{1}))=e<t\leq t_{0}.$$
Then by the induction hypothesis applied to $\Dpp_{n}$, we have $s_{\Dpp_{n}}\leq 2 t_{\Dpp_{n}}$. To finish the proof, we prove the following inequality

$$s_{\cD_{n}}\leq s_{\Dpp_{n}}+2(t-t_{\Dpp_{n}}).$$

We denote by $S_{[\frakl_{1},\frakl_{2}]}$ the submodule of $\Sel_{\Delta}(K_{\infty}, A_{f,n})$ consisting of classes that are locally trivial at primes over $\frakl_{1}$ and $\frakl_{2}$. By definition we have the following exact sequence
\begin{equation}\label{e1}
\Hhat^{1}_{\sing}(K_{\infty, \frakl_{1}}, T_{f,n})\oplus\Hhat^{1}_{\sing}(K_{\infty,\frakl_{2}}, T_{f,n})\xrightarrow{P}\dual{\Sel_{\Delta}(K_{\infty}, A_{f,n})} \mapR \dual{S_{[\frakl_{1},\frakl_{2}]}}\mapR 0
\end{equation}
where $P$ is induced by $\<,\>_{\frakl_{1}}\oplus\<,\>_{\frakl_{2}}$. Similarly, we have
\begin{equation}\label{e2}
\Hhat^{1}_{\fin}(K_{\infty, \frakl_{1}}, T_{f,n})\oplus\Hhat^{1}_{\fin}(K_{\infty,\frakl_{2}}, T_{f,n})\xrightarrow{P} \dual{\Sel_{\Delta\frakl_{1}\frakl_{2}}(K_{\infty}, A_{f,n})} \mapR \dual{S_{[\frakl_{1},\frakl_{2}]}}\mapR 0.
\end{equation}

We tensor the first exact sequence by $\cO_{\xi}$ and we fix isomorphisms $$(\Hhat^{1}_{\sing}(K_{\infty, \frakl_{1}}, T_{f,n})\oplus\Hhat^{1}_{\sing}(K_{\infty,\frakl_{2}}, T_{f,n}))\otimes_{\xi}\cO_{\xi}\isomor \cO^{2}_{\xi}/\xi(\uf)^{n}$$ and $$(\Hhat^{1}_{\fin}(K_{\infty, \frakl_{1}}, T_{f,n})\oplus\Hhat^{1}_{\fin}(K_{\infty,\frakl_{2}}, T_{f,n}))\otimes_{\xi}\cO_{\xi}\isomor \cO^{2}_{\xi}/\xi(\uf)^{n} .$$ 

By \lmref{vanish}, we get from the first exact sequence above, 
\begin{equation}\label{e3}
\frac{\cO_{\xi}}{(\partial_{\frakl_{1}}(\pkaDxi(\frakl_{1})))}\oplus \frac{\cO_{\xi}}{(\partial_{\frakl_{2}}(\pkaDxi(\frakl_{2})))}\xrightarrow{P}\dual{\Sel_{\Delta}(K_{\infty}, A_{f,n})}\otimes_{\xi}\cO_{\xi} \rightarrow \dual{S_{[\frakl_{1},\frakl_{2}]}}\otimes_{\xi}\cO_{\xi}\rightarrow 0.
\end{equation}

\begin{lm} We have an isomorphism
$$\dual{\Sel_{\Delta\frakl_{1}\frakl_{2}}(K_{\infty}, A_{f,n})}\otimes_{\xi}\cO_{\xi} \isomor \dual{S_{[\frakl_{1},\frakl_{2}]}}\otimes_{\xi}\cO_{\xi}.$$
\end{lm}
\begin{proof}
Let $s\in \Sel_{\Delta\frakl_{1}\frakl_{2}}(K_{\infty}, A_{f,n})[\ker(\xi)]$. We consider the sum $$\sum_{v}\<\res_{v}(\pkaDxi(\frakl_{1})), \res_{v}(s)\>_{v}=0.$$  For $v\nmid\frakl_{2}$, the summand $\<\res_{v}(\pkaDxi(\frakl_{1})), \res_{v}(s)\>_{v}=0$ and it follows that $$\<v_{\frakl_{2}}(\pkaDxi(\frakl_{1})), \res_{\frakl_{2}}(s)\>_{\frakl_{2}}=0.$$ The same argument can be used to show that $\<v_{\frakl_{1}}(\pkaDxi(\frakl_{2})), \res_{\frakl_{1}}(s)\>_{\frakl_{1}}=0$.  Hence \eqref{e2} factors through
$$\frac{\cO_{\xi}}{(v_{\frakl_{1}}(\pkaDxi(\frakl_{2})))}\oplus \frac{\cO_{\xi}}{(v_{\frakl_{2}}(\pkaDxi(\frakl_{1})))}\xrightarrow{P} \dual{\Sel_{\Delta\frakl_{1}\frakl_{2}}(K_{\infty}, A_{f,n})}\otimes_{\xi}\cO_{\xi} \rightarrow \dual{S_{[\frakl_{1},\frakl_{2}]}}\otimes_{\xi}\cO_{\xi}\rightarrow 0.$$
But $\oxi(v_{\frakl_{1}}(\pkaDxi(\frakl_{2})))=t_{\Dpp_{n+t}}-e=0 $ and the claim follows.
\end{proof}

Now \thmref{induction} follows by pluging in the isomorphism to the exact sequence \eqref{e3} and noticing that $\oxi(\partial_{\frakl_{1}}(\pkaDxi(\frakl_{1})))=\oxi(\partial_{\frakl_{2}}(\pkaDxi(\frakl_{2})))=t-e= t-t_{\Dpp_{n}}$.

\subsection{One divisibility of the Iwasawa main conjecture}
Now we deduce the one divisibility of the main conjecture from the above theorem. Let $\pi$ be the unitary automorphic representation associated to $f$. Following \cite{Chida_Hsieh}, Hung \cite{Hung:thesis} has defined $\theta_{\infty}(f)=\theta_{\infty}(f,1)\in \Lambda$. We briefly sketch its construction below and deduce our main result from it.

By the Jaquet-Langlands correspondence we can find an automorphic representation $\pi^{\prime}$ for the group $G=\Res_{F/\Q}(B^{\cross})$ and an normalized eigenform $f_{B}\in S_{k}(\widehat{R}^{\cross}_{\nplus}Y, \C)$ with the property that  $T_{v}f_{B}=\bfa_{v}(f)f_{B}$ for $v\nmid \frakn$ and $U_{v}f_{B}=\alpha_{v}(f)f_{B}$ for $v\mid \frakn$. Here we let $B$ be a quaternion algebra over $F$ with discriminant $\frakn^{-}$ and $\widehat{R}_{\nplus}$ be an Eichler order of level $\frakn^{+}$. We normalize  $f_{B}$ such that its $p$-adic avatar $\widehat{f_{B}}\neq 0(\mod \uf)$. 

For ${\bf{v}}\in L_{k}(\C)$ and $h_{B}\in S_{k}(\widehat{R}^{\cross}_{\frakn^{+}}Y, \C)$, we can define $\Psi({\bf{v}}\otimes h_{B}) \in \cA_{k}(\nplus, \C)$ in the space of scalar valued automorphic forms on $\widehat{B}^{\cross}$ of level $\frakn^{+}$. Here $\Psi({\bf{v}}\otimes h_{B})(g)= \<\rho_{k,\infty}(g_{\infty}){\bf{v}}, h_{B} \>_{k}$. 

Back to our setting, we set $\varphi_{B}=\Psi({\bf v}_{0}\otimes f_{B})\in \cA_{k}(\nplus, \C)[\pi^{\prime}]$ where ${\bf v}_{0}=X^{\frac{k-2}{2}}Y^{\frac{k-2}{2}}$.  And define the $\frakp$-stablization $\varphi_{B}^{\dagger}$ of  $\varphi_{B}$ with respect to $\alpha_{\frakp}$ as $$\varphi^{\dagger}_{B}=\varphi_{B}-\frac{1}{\alpha_{\frakp}}\pi^{\prime}(\pMX{1}{0}{0}{\uf_{\frakp}})\varphi_{B}.$$ The theta element is defined as
$$\Theta_{m}(f)= {\alpha_{\frakp}^{-m}}\sum_{a\in \Gal(H_{m}/K)} \varphi^{\dagger}_{B}(x_{m}(a))[a]\in \C[\Gal(H_{m}/K)].$$

By \cite[Lemma 5.4 (1)]{Hung:thesis}, one can in fact identify $\theta_{m}(f)$ as an element in $\cO_{f}[\Gal(H_{m}/K)]$. By the same computation as in \lmref{Normcom}, we verify $$\pi_{m+1,m}(\Theta_{m+1}(f))=\Theta_{m}(f)$$  for the transition map $\pi_{m+1, m}: \Gal(H_{m+1}/K)\rightarrow \Gal(H_{m}/K)$. We define $\theta_{m}(f)$ as the image of $\Theta_{m}(f)$ under the natural projection $\Gal(H_{m}/K)\rightarrow \Gal(K_{m}/K).$ In particular we obtain an element $\theta_{\infty}(f)\in \cO_{f}[[\Gal(H_{\infty}/K)]]$. We can project this element to $\cO_{f}[[\Gal(K_{\infty}/K)]]$. We define the $p$-adic $L$-function $L_{p}(K_{\infty}, f)$ by $L_{p}(K_{\infty}, f)=\theta_{\infty}(f)^{2}$  and we have the following interpolation formula proved by Hung \cite{Hung:thesis}.
\begin{thm}\label{interpolate}
For each finite character $\chi: \Gamma \mapR \C_{p}$ of conductor $\frakp^{m}$,
$$\chi({\theta^{2}_{\infty}(f)})=\Gamma(\frac{k}{2})^{2}u^{2}_{K}\frac{\sqrt{D_{K}}D^{k-1}_{K}}{D_{F}^{\frac{3}{2}}}\Nm{\frakp}^{m(k-1)}\chi(\frakN^{+})\epsilon_{\frakp}(f)\ap^{-m}e_{\frakp}(f,\chi)^{2}\frac{L(f/K,\chi, k/2)}{\Omega_{f,\frakn^{-}}}.$$
\end{thm}
Here we denote by
\begin{enumerate}
\item $D_{F}$ \resp $D_{K}$ is the absolute discriminant of  $F$ \resp $K$. 
\item We fix a decomposition $\frakn^{+}=\frakN^{+}{\overline{\frakN^{+}}}$.
\item $\epsilon_{\frakp}(f)$ is the local root number at $\frakp$. 
\item $u_{K}=[\cO^{\times}_{K}:\cO^{\times}_{F}]$.
\item $e_{\frakp}(f,\chi)$ is the $p$-adic multiplier defined as follows
 \[e_{\mathfrak{p}}(f,\chi)=\begin{cases}
1 & \hbox{if $\chi$ is ramified;}\\
(1-\alpha_{\mathfrak{p}}^{-1}\chi(\mathfrak{P}))(1-\alpha_{\mathfrak{p}}^{-1}\chi(\overline{\mathfrak{P}})) & \hbox{if $s=0$ and $\mathfrak{p}=\mathfrak{P\overline{P}}$ is split;}\\
1-\alpha_{\mathfrak{p}}^{-2} & \hbox{if $s=0$ and $\mathfrak{p}$ is inert;}\\
1-\alpha_{\mathfrak{p}}^{-1}\chi(\mathfrak{P}) & \hbox{if $s=0$\hbox{ and }$\mathfrak{p}=\mathfrak{P}^{2}$ is ramified.}
\end{cases}\]
\item $\Omega_{f,\frakn^{-}}$ is a complex period associated to $f$ known as the Gross period \cite[(5.2)]{Hung:thesis}.
\end{enumerate}

\begin{prop}\label{cong}
If $\Delta=\nminus$, there is an $n$-admissible form $\cD^{f}_{n}=(\nminus,f^{\dagger,[k-2]}_{n})$, such that  $$\theta_{m}(\cD^{f}_{n})\equiv \theta_{m}(f)\pmod{\uf^{n}}$$
and in particular that
$$\theta_{\infty}(\Dfn)\equiv \theta_{\infty}(f)\pmod{\uf^{n}}.$$
\end{prop}

\begin{proof}
Let $B$ and $f_{B}$ be chosen as in the previous discussion. Let $$f^{\dagger,[k-2]}_{n}(b)= \sqrt{\beta}^{\frac{2-k}{2}}\<X^{k-2}, {\widehat{f^{\dagger}}}_{B}(b)\>_{k}\mod \uf^{n}.$$ Then  we have $f^{\dagger,[k-2]}_{n}\in S^{B}_{2}(\openU Y, \cO_{f,n})$. Moreover it is shown in \cite[Lemma 6.11] {Hung:thesis} that $f^{\dagger,[k-2]}_{n}$ is not Eisenstein. Hence we conclude by \cite[Corolarry 6.8]{Hung:thesis} that $f^{\dagger,[k-2]}_{n} (\mod \uf)$ is non-zero.  Finally we conclude with \cite[Lemma 5.4] {Hung:thesis} for the desired congruence relation. Thus we know that $f^{\dagger,[k-2]}_{n}$ is $n$-admissible.
\end{proof}

\begin{thm}\label{main_weak}
Assume the hyothesis $(\CR)$, $\nmin$, $\PO$ and "Ihara's lemma" hold. We have $\char_{\Lambda}\dual{\Sel(K_{\infty}, A_{f})}\supset (L_{p}(K_{\infty}, f))$.
\end{thm}

\begin{proof}
let $\xi: \Lambda\mapR \cO_{\xi}$ be an $\cO_{f}$-algebra homomorphism.  If $\xi(L_{p}(K_{\infty},f))$ is zero, then then it is in $\Fitt_{\cO_\xi}(\dual{\Sel(K_{\infty}, A_{f}))}\otimes_{\xi}\cO_{\xi}$. If $\xi(L_{p}(K_{\infty},f))$ is non-zero, then we can choose $t_{0}>\oxi(\xi(L_{p}(K_{\infty},f)))$. For each positive $n$, we have the $n+t_{0}$- admissible form $\cD^{f}_{n+t_{0}}=(\nminus, \f^{\dagger,[k-2]}_{n+t_{0}})$ as in \propref{cong} and the $n$-admissible form $\cD^{f}_{n}=\cD^{f}_{n+t_{0}}\pmod{\uf^{n}}$. By applying \thmref{induction}, we have $\xi(L_{p}(K_{\infty},f))=\xi(\theta_{\infty}(\cD^{f}_{n})^{2})\in \Fitt_{\cO_{\xi}}(\dual{\Sel_{\nminus}(K_{\infty}, A_{f,n})}\otimes_{\xi}\cO_{\xi})$ for all $\xi$ and $n$. Then the criterion \lmref{criterion} shows that $L_{p}(K_{\infty}, f)\in \Fitt_{\Lambda}\dual{\Sel_{\nminus}(K_{\infty}, A_{f})}$. By \cite[Theorem B]{Hung:thesis}, $L_{p}(K_{\infty}, f)$ is non-zero and hence  $\Sel_{\nminus}(K_{\infty}, A_{f})$  is $\Lambda$-cotorsion. The theorem then follows from \propref{pollack-weston} which identifies $\Sel_{\nminus}(K_{\infty}, A_{f})$ with $\Sel(K_{\infty}, A_{f})$.
\end{proof}

\section{Some arithmetic consequences}
Let $L$ be a number field and $V$ be a representation of the Galois group $G_{L}$ over some $p$-adic field $E$. Recall that the Bloch-Kato Selmer group $\Sel_{f}(L, V)$ is defined by
$$\Sel_{f}(L, V)=\ker\{H^{1}(L,V)\rightarrow \prod_{v}H^{1}(L_{v}, V)/H^{1}_{f}(L_{v}, V)\}$$
where
\begin{equation}\label{BK}
\begin{aligned} 
& H^{1}_{f}(L_{v},V)=H^{1}_{\fin}(L_{v}, V) \text{ when $v\nmid p$},\\  
& H^{1}_{f}(L_{v}, V)=\ker\{H^{1}(L_{v}, V)\rightarrow H^{1}(L_{v}, V\otimes_{\Q_{p}} B_{\rm cris})\} \text{ for $v\mid p$}.\\
\end{aligned}
\end{equation}

The Bloch-Kato conjecture relates the rank of the Selmer group and the order of vanishing of the $L$-function attached to the Galois representation. In this section we prove certain cases of the refined type Bloch-Kato conjecture in the rank zero case.

\subsection{Bloch-Kato type conjecture and parity conjecture}
In this chapter, we denote by $H_{\frakc}$ the ring class field of K with conductor $\frakc$. Let $\chi$ be a finite order character of $\Gal(H_{\frakc}/K)$, we consider the twisted Bloch-Kato Selmer group  $\Sel_{f}(K, V_{f}\otimes \chi)$. We have a canonical isomorphism 
$\Sel_{f}(K, V_{f}\otimes \chi)\isomor \Sel_{f}(H_{\frakc}, V_{f})^{\chi^{-1}}$
where the righthand side is the $\chi^{-1}$-isotypic component of $\Sel_{f}(K, V_{f})$ for the action of $\Gal(H_{\frakc}/K)$. The complex conjugation of $\Gal(H_{\frakc}/K)$ induces an isomorphism between $$\Sel_{f}(H_{\frakc}, V_{f})^{\chi^{-1}} \text{ and } \Sel_{f}(H_{\frakc}, V_{f})^{\chi}.$$ The Bloch-Kato conjecture relates the order of the $L$-function $L(f/K,\chi, k/2)$ to the rank of these Selmer groups, more precisely we expect:
$$\ord_{s=k/2} L(f/K,\chi, s)= \rk \Sel_{f}(K, V_{f,\chi}).$$ 
In this section we provide some results in this direction. In particular we prove the so called parity conjecture in this setup.

First we provide the following proposition which relates the divisibility index of the theta element to the size of the Selmer group studied in this paper. We remark that the proof of this result only depends on the first reciprocity law in \thmref{1law} and in particular we don't have to assume "Ihara's lemma". Let $\chi$ be a finite order character of $\Gal(K_{m}/K)$ and let $\uf_{\chi}$ be the uniformizer for the ring generated over $\cO_{f}$ by the values of $\chi$.

\begin{prop}\label{finite_sel}
Let $c=\ord_{\uf_{\chi}}(\chi(\theta_{m}(f)))$. Then $\uf_{\chi}^{c}\Sel_{\nminus}(K_{m}, A_{f,n})^{\chi}=0$ for every large $n$.
\end{prop}

\begin{proof}
For every integer $n>c$, we can find an admissible form $\cD^{f}_{n}=(\nminus, f^{\dagger,[k-2]}_{n})$, such that 
$$\theta_{m}(\Dfn)\equiv \theta_{m}(f)\pmod{\uf^{n}}.$$
By the first reciprocity law and \thmref{admissible}, we therefore can find an admissible prime $\frakl$ such that $\partial_{\frakl}(\kappa_{\cD^{f}_{n}}(\frakl))\equiv \theta_{m}(f)\pmod{\uf_{\frakp}^{n}}$. Let $s\in \Sel_{\frakn^{-}}(K_{m}, A_{f,n})^{\chi}$, we can also choose the admissible prime $\frakl$ such that $\res_{\frakl}(s)$ is non-zero using \thmref{admissible}. Then by considering the sum $$\sum_{v}\<\res_{v}(\kappa_{\cD^{f}_{n}}(\frakl)), \res_{v}(s)\>=0,$$ we conclude that $$\<\partial_{\frakl}(\kappa_{\cD^{f}_{n}}(\frakl)), \res_{\frakl}(s)\>=0$$ using the same type of argument as in \lmref{vanish}. Hence $\uf_{\chi}^{c}s=0$ by the non-degeneracy of the local Tate pairing. Therefore we see that $\uf_{\chi}^{c}\Sel_{\nminus}(K_{m}, A_{f,n})^{\chi}=0$. 
\end{proof}

Now an immediate corollary  from the above theorem is the proof of the following rank zero case of the refined Bloch-Kato conjecture.

\begin{thm}
Suppose that $L(f,\chi,k/2)\neq 0$. Then $\Sel_{f}(K_{m}, V_{f})^{\chi}=0$. 
\end{thm}
\begin{proof}
By the interpolation formula \thmref{interpolate} and \propref{finite_sel}, we see that $\Sel_{\nminus}(K_{ m}, A_{f,n})^{\chi}$ is bounded independent of $n$. Therefore it follows that $\Sel_{\nminus}(K_{ m}, A_{f})^{\chi}$ is finite. Since $\Sel_{f}(K_{m}, V_{f})^{\chi}\subset \Sel_{\nminus}(K_{m}, V_{f})^{\chi}$, we obtain $\Sel_{f}(K_{m}, V_{f})^{\chi}=0$.
\end{proof}

\subsection{The parity conjecture}

From the Bloch-Kato type theorem proved in the last section, one can deduce a version of the parity conjecture for twists of $V_{f}$. The result stated here is already proved by Nekovar \cite{Nekovar_Grow} and we give another proof based on the general treatment of the parity conjecture in \cite{Nekovar_parity3}. A similar treatment of the parity conjecture in a slightly different situation is proved in \cite[Theorem 6.4]{Castella_Hsieh}, we follow their method closely. We let $\chi$ be a finite order character of $\Gamma$ of conductor $\frakp^{m}$ as above. Let $V_{f,\chi}=V_{f}\otimes \chi$. We have an isomorphism $\Sel_{f}(K, V_{f,\chi})\cong \Sel_{f}(K_{m}, V_{f})^{\chi^{-1}}$. The parity conjecture in our context asserts the following:

\begin{thm}
$$\text{ord}_{s=k/2} L(f/K,\chi, s)\equiv \text{rank } \Sel_{f}(K, V_{f,\chi})\equiv 0\pmod{2}.$$
\end{thm}

\begin{proof}
Recall that $\Lambda$ is the Iwasawa algebra for $\Gamma$ over $\cO_{f}$. Let $\chi^{uni}:G_{K}\rightarrow \Lambda^{\cross}$ be the universal deformation of the character $\chi$. Denote by $\cT$ the $G_{F}$-module $T\otimes \text{Ind}_{K}^{F}\chi^{univ}$ and let $\cV=\cT\otimes E_{f}$. We also put $\cT^{+}_{p}=F^{+}T\otimes \Lambda$. Then the pair $(\cT, \cT^{+}_{p})$ and its specializations satisfies the condition listed in \cite[5.12.(1)-(9)]{Nekovar_parity3}. Let $\phi$ be a finite order character of $\Gamma$ sufficiently wildly ramified such that $\chi\phi(\theta_{\infty}(f))^{2}\neq 0$. Then by \thmref{finite_sel}, we can conclude that $\text{rank } \Sel_{f}(K, V_{f,\chi\phi})=0$.  We put $\cV_{\phi}$ as the specialization of $\cV$ at $\phi$. By \cite[5.3.1]{Nekovar_parity3}, we have $\text{rank } \Sel_{f}(F, \cV_{\phi})\equiv \text{rank }\Sel_{f}(F,\cV_{triv})\pmod{2}$ where $\cV_{triv}$ is the specialization of $\cV$ at the trivial character. We also know the sign of the associated Weil-Deligne representation $\epsilon(\cV_{\phi^{\prime}})$ is independent of the specialization $\phi^{\prime}$ and is equal to $1$. Notice that we have $\Sel_{f}(F, \cV_{\phi^{\prime}})\cong \Sel_{f}(K, V_{f,\chi\phi^{\prime}})$ for any specialization $\phi^{\prime}$ and we conclude by letting $\phi^{\prime}$ be $\phi$ and be the trivial character respectively. This shows
$$
\begin{aligned}
&\text{rank }\Sel_{f}(K, V_{f,\chi})\equiv \rk \Sel_{f}(K, V_{f,\chi\phi})\equiv 0\pmod{2}\\
\end{aligned}
$$
which finishes the proof.
\end{proof}

\section{Multiplicity one for quaternion alegbra}
The purpose of this section is to prove a multiplicity one result for the space of automorphic forms on a quaternion algebra localized at certain maximal ideal of the Hecke algebra. This result plays a crucial role in  establishing the reciprocity laws.

\subsection{Taylor-Wiles construction for Hilbert modular forms}
Let $\pi$ be the automorphic representation attached to $f$. Let $\pi^{\prime}$ be the preimage of $\pi$ under the Jaquet-Langlands correspondence.  So $\pi^{\prime}$ gives rise to a morphism $\lambda_{\pi^{\prime}}: \Hecke_{B}(\nplus Y)\rightarrow \cO_{f}$ with $Y=\widehat{F}^{\cross}$. Let $\frakn_{0}=\frakn_{\bar{\rho}_{f}}$ be the Artin conductor of the residual Galois representation $\bar{\rho}_{f}$ of $\rho_{f}$.
Let $\frakn^{-}_{1}$ be the product of primes of $\nminus$ not dividing $\frakn_{0}$ and let $\frakn_{\emptyset}=\frakn^{-}_{1}\frakn_{0}$.
By the level lowering results \cite{Rajaei:Thesis}, \cite{Javis:Level}, there exists a morphism $\lambda_{\emptyset}: \Hecke_{B}(\frakn_{\emptyset}Y)\rightarrow \cO_{f}$ such that $\lambda_{\emptyset}(T_{v})=\lambda_{\pi^{\prime}}(T_{v}) \pmod{\uf}$ for all $v\nmid \frakn$. We write $\prod_{v} v^{m_{v}}=\frakn/\frakn_{\emptyset}$. We recall our assumptions $(\CR)$ and $\nimpr$ namely:

\begin{hyp}[$\text{CR}^{+}$]\label{CR*}
\begin{enumerate}
\item{$p>k+1$ and $\#(\F^{\times}_{\frakp})^{k-1}>5$.} % not sure about this condition
\item{The restriction of $\bar{\rho}_{f}$ to $G_{F(\sqrt{p^{*}})}$ is irreducible where $p^{*}=(-1)^{\frac{p-1}{2}}p$.}
\item{$\bar{\rho}_{f}$ is ramified at $\frakl$ if $\frakl\mid\frakn^{-}$ and ${\Nm{\frakl}^{2}} \equiv 1 \pmod{p}$.}
\item{$\frakn_{0}$ is prime to $\frakn/\frakn_{0}$.}
\end{enumerate}
\end{hyp}
$$\nimpr:\text{If $\frakl\mid\mid \nplus$ and $\Nm{\frakl}\equiv1\pmod{p}$, then $\bar{\rho}_{f}$ is ramified.}$$

By \cite[Theorem 1.5]{Javis:Level} and $(\CR)$, we know that $m_{v}\leq 2$ and $m_{v}=0$ unless $v\mid \nplus$. Let $\Sigma$ be a subset of the set of prime factors of $\frakn/\frakn_{\emptyset}$ and we put $\frakn_{\Sigma}=\frakn_{\emptyset}\prod_{v\in\Sigma}v^{m_{v}}$. Let $\Hecke_{\Sigma}=\Hecke_{B}(\frakn_{\Sigma}Y)_{\m_{\Sigma}}$ be the localization of the Hecke algebra at the maximal ideal $\m_{\Sigma}$ generated by:
$$\uf, T_{v}-\lambda_{\emptyset}(T_{v})\text{ for  $v\nmid \frakn_{\Sigma}$}, U_{v}-\lambda_{\pi^{\prime}}(U_{v})\text{ for  $v\mid\frakn_{\Sigma}$}.$$ 
Let $S_{\Sigma}=S_{k}(\widehat{R}_{\frakn_{\Sigma}}^{\cross}Y, \cO_{f})_{\m_{\Sigma}}$ where ${R}_{\frakn_{\Sigma}}$ is the Eicher order of level $\frakn_{\Sigma}$.
This chapter is devoted to proving the following multiplicity one result. 

\begin{thm}
Assume $(\CR)$ and $\nimpr$ hold, the module $S_{\Sigma}$ is free over $\Hecke_{\Sigma}$ of rank 1.
\end{thm}
The proof of this is based on the refinement of Diamond on the Taylor-Wiles method \cite{Diamond:TW}. Notice this result is already shown in \cite[Theroem 3.2]{Taylor:Ihara_avoidance} in the case when $m_{v}=2$ for every $v\in \Sigma$. We review his result and extend it to our situation, this section follows closely to \cite[Section 6]{Chida_Hsieh}.

There is a Galois representation $\rho_{\Sigma}: G_{F}\rightarrow GL_{2}(\Hecke_{\Sigma})$ such that 

\begin{itemize}
\item{$\rho_{\Sigma}$ is unramified outside $p\frakn_{\Sigma}$.}
\item{$\Tr\rho_{\Sigma}(\Frob_{v})=T_{v}$ for all $v$ away from $p\frakn_{\Sigma}$.}
\item{There is a character $\chi_{v}: G_{F_{v}}\rightarrow \Hecke^{\cross}_{\Sigma}$ such that $\rho_{\Sigma}\mid_{G_{F_{v}}}=\pMX{\chi^{-1}_{v}\epsilon}{*}{0}{\chi_{v}}$ and $\chi_{v}\mid_{ I_{F_{v}}}=\epsilon^{(2-k)/2}$ for $v\mid p$.}
\item{If $v\| \frakn_{\Sigma}/\frakn^{-}_{1}$, then $\rho_{\Sigma}\mid_{G_{F_{v}}}= \pMX{\chi^{-1}_{v}\epsilon}{*}{0}{\chi_{v}}$ for some character $\chi_{v}: G_{F_{v}}\rightarrow \Hecke^{\cross}_{\Sigma}$ with $\chi_{v}(\Frob_{v})=U_{v}$.}
\item{If $v\mid \frakn^{-}_{1}, \rho_{\Sigma}\mid_{G_{F_{v}}}= \pMX{\pm\epsilon}{*}{0}{\pm1}$.}
\end{itemize}

Let $v\mid \frakn/\frakn^{-}_{1}$ be a prime such that $v\not\in\Sigma$. We define a level raising map $L_{v}:S_{\Sigma}\rightarrow S_{\Sigma\cup \{v\}}$ and an element $u_{\Sigma, v}\in \Hecke_{\Sigma}$ as follows:
\begin{itemize}
\item{if $m_{v}=2$, we put $u_{\Sigma,{v}}=0$, $$L_{v}(f)=\Nm{v}f-\pMX{1}{0}{0}{\uf_{v}}T_{v}f+\pMX{1}{0}{0}{\uf^{2}_{v}}f.$$}
\item{If $m_{v}=1$, we have a congruence $x^{2}-T_{v}x+\Nm{v}\equiv (x-\epsilon_{v})(x-\Nm{v}\epsilon_{v})\pmod{\uf}$. By Hensel's lemma $\epsilon_{v}$ lifts to a unique root $u_{\Sigma,v}$ of $x^{2}-T_{v}x+\Nm{v}$  since $\Nm{v}\not\equiv 1 \pmod{p}$. And we put $$L_{v}(f)=u_{\Sigma,{v}}f-\pMX{1}{0}{0}{\uf_{v}}f.$$}
\item{One can verify immediately that $L_{v}\circ u_{\Sigma, v}=U_{v}\circ L_{v}$}. 
\end{itemize}

\begin{lm}\label{Ihara_def}
The map $L_{v}$ is injective.
\end{lm}
\begin{proof}
This is \cite[Lemma 3.1]{Taylor:Ihara_avoidance}. Both of the above cases rely on the the strong approximation theorem and the fact that $\m_{\Sigma}$ is non-Eisenstein.
\end{proof}

It follows then that the map $\Hecke_{\Sigma\cup\{v\}}\rightarrow \Hecke_{\Sigma}$ induced by sending $U_{v}$ to $u_{\Sigma, v}$ is surjective. In particular we have a surjective map $\Hecke_{\Sigma}\twoheadrightarrow\Hecke_{\emptyset}$. Let $\lambda_{\Sigma}: \Hecke_{\Sigma}\rightarrow \Hecke_{\emptyset}\rightarrow \cO_{f}$ be the composite morphism and let $I_{\lambda_{\Sigma}}$ be its kernel. Set
$$S_{\Sigma}[\lambda_{\Sigma}]=\{c\in S_{\Sigma}: I_{\lambda_{\Sigma}}c=0\}.$$

\begin{lm}\label{reduced}
The localized Hecke algebra $\Hecke_{\Sigma}$ is reduced.
\end{lm}
\begin{proof}
We need to show that $U_{v}$ is semisimple in $\Hecke_{\Sigma}$. When $v\in \Sigma$ and $m_{v}=2$, $U_{v}=0$ see \cite[Corollary 1.8]{Taylor:Ihara_avoidance}.  And the case $m_{v}=1$ follows from our assumption $\nimpr$ since $U_{v}=\pMX{T_{v}}{-1}{\Nm{v}}{0}$ under a suitable basis.
\end{proof}
Hence by strong multiplicity one and the above lemma, we conclude that $S_{\Sigma}[\lambda_{\Sigma}]$ is rank one over $\cO_{f}$. Let $S_{\Sigma}[\lambda_{\Sigma}]$ be the dual module defined by
$$S_{\Sigma}[\lambda_{\Sigma}]^{\bot}=\{c\in S_{\Sigma}[\lambda_{\Sigma}]\otimes E_{f}: \<c, x\>_{\frakn_{\Sigma}}\in \cO_{f},  \forall x\in S_{\Sigma}[\lambda_{\Sigma}]\}.$$
Then the congruence module for $\lambda_{\Sigma}$ is defined  by $C(\frakn_{\Sigma})=S_{\Sigma}[\lambda_{\Sigma}]^{\bot}/S_{\Sigma}[\lambda_{\Sigma}]$ and $\mu_{\Sigma}=\lambda_{\Sigma}(\text{Ann}_{\Hecke_{\Sigma}}(I_{\lambda_{\Sigma}}))$ is called the congruence ideal. Then we have the following criterion for the freeness of $S_{\Sigma}$ over $\Hecke_{\Sigma}$.

\begin{lm}\label{numerical}
$\#C(\frakn_{\Sigma})\leq \#(\cO_{f}/\mu_{\Sigma})$ and equality holds if $S_{\Sigma}$ is free over $\Hecke_{\Sigma}$.
\end{lm}
\begin{proof}
This is the numerical criterion in \cite[Theorem 2.4]{Diamond:TW}.
\end{proof}

\begin{lm}\label{cong_module}
If $v\| \frakn/\frakn_{\emptyset}$ and $v\not \in \Sigma$, then $\#C(\frakn_{\Sigma\cup\{v\}})=\#C(\frakn_{\Sigma})\#(\cO_{f}/((\lambda_{\emptyset}(u_{\emptyset,v})^{2}-1)\cO_{f})$.
\end{lm}

\begin{proof}
Let $L^{*}_{v}: S_{\Sigma}\rightarrow S_{\Sigma\cup\{v\}}$ be the adjoint map of $L_{v}$ with respect to $\<.\>_{\frakn_{\Sigma}}$ and $\<.\>_{\frakn_{\Sigma\cup\{v\}}}$. Then by \lmref{Ihara_def} and the above discussion, we have $L_{v}(S_{\Sigma}[\lambda_{\Sigma}])=S_{\Sigma\cup\{v\}}[\lambda_{\Sigma\cup\{v\}}]$ and dually $L^{*}_{v}(S_{\Sigma\cup\{v\}}[\lambda_{\Sigma\cup\{v\}}]^{\bot})=S_{\Sigma}[\lambda_{\Sigma}]^{\bot}$. 

A direct computation shows that $L^{*}_{v}=[\openU\pMX{\uf_{v}}{0}{0}{1}\openU^{\prime}]-[\openU\openU^{\prime}]$  for $\openU=\widehat{R}^{\cross}_{\frakn_{\Sigma}}$ and $\openU^{\prime}=\widehat{R}^{\cross}_{\frakn_{\Sigma\cup\{v\}}}$. Thus we have 

\begin{equation}
\begin{aligned}
&L^{*}_{v}\circ L_{v}= u_{\Sigma,v}[\openU\pMX{\uf_{v}}{0}{0}{1}\openU]u_{\Sigma, v}-(1+\Nm{v})u_{\Sigma, v}[\openU\pMX{\uf_{v}}{0}{0}{1}\openU]\\
&-(1+\Nm{v})u_{\Sigma\cup\{v\}}[\openU\openU]+[\openU\pMX{1}{0}{0}{\uf_{v}}\openU]\\
&=(1+u_{\Sigma,v}^{2})T_{v}-2u_{\Sigma,v}(1+\Nm{v}).\\
\end{aligned}
\end{equation}

Since $u^{2}_{\Sigma,v}-T_{v}u_{\Sigma, v}+\Nm{v}=0$, we obtain
$$L^{*}_{v}\circ L_{v}=u_{\Sigma,v}^{-1}(u^{2}_{\Sigma, v}-1)(u^{2}_{\Sigma,v}-\Nm{v}).$$

Since by $\nimpr$, we have $u_{\Sigma, v}^{2}-\Nm{v} \equiv 1-\Nm{v}\not\equiv 0 \pmod{m_{\Sigma}}$. Hence 
\begin{equation}
\begin{aligned}
\#C(\frakn_{\Sigma\cup\{v\}})&=\#(S_{\Sigma\cup\{v\}}[\lambda_{\Sigma\cup\{v\}}]^{\bot}/S_{\Sigma\cup\{v\}}[\lambda_{\Sigma\cup\{v\}}])\\
&=\#(S_{\Sigma}[\lambda_{\Sigma}]^{\bot}/L^{*}_{v}\circ L_{v}(S_{\Sigma}[\lambda_{\Sigma}]))\\
&=\#S_{\Sigma}[\lambda_{\Sigma}]^{\bot}/(S_{\Sigma}[\lambda_{\Sigma}]).\#S_{\Sigma}[\lambda_{\Sigma}]/L^{*}_{v}\circ L_{v} (S_{\Sigma}[\lambda_{\Sigma}])\\
&=\#C(\frakn_{\Sigma}).\#\cO_{f}/(\lambda_{\emptyset}(u_{\emptyset, v})^{2}-1)\cO_{f}.\\
\end{aligned}
\end{equation}
\end{proof}

\subsection{Deformation ring and its tangent space}

Let $\cD_{\Sigma}: \text{CNL}_{\cO_{f}}\rightarrow \text{SET}$ be the functor that sends a complete Noetherian local $\cO_{f}$-algebra $A$ to the set $\cD_{\Sigma}(A)$ of representations $\rho_{A}: G_{F}\rightarrow \GL_{2}(A)$ such that
\begin{itemize}
\item{$\rho_{A}\equiv \bar{\rho_{f}} \pmod{\uf_{A}}$.}
\item {$\rho_{A}$ is minimally ramified outside $\frakn^{-}_{1}\Sigma$ in the sense of \cite{Diamond:Boston}.}
\item {For $v\mid p$, there is a character $\chi_{v}: G_{F}\rightarrow A^{\cross}$ such that $$\rho_{A}\mid_{ G_{F_{v}}}=\pMX{\chi_{v}^{-1}\epsilon}{*}{0}{\chi_{v}}$$ with $\chi_{v}\mid _{I_{F_{v}}}= \epsilon^{(2-k)/2}$.}
\item{For $v\mid \frakn_{\Sigma}/\frakn_{\emptyset}$, there exists a unramified character $\chi_{v}: G_{F_{v}}\rightarrow A^{\cross}$ such that $$\rho_{A}\mid_{G_{F_{v}}}=\pMX{\chi_{v}^{-1}\epsilon}{*}{0}{\chi_{v}}.$$}
\item{For $v\mid \frakn^{-}$, then $$\rho_{A}\mid_{ G_{F_{v}}}=\pMX{\pm \epsilon}{*}{0}{\pm 1}$$ with $\uf_{A}\mid *$.}
\end{itemize}

It is well-known that this deformation functor is representable by  $R_{\Sigma}$ and we have a universal deformation $\rho_{\Sigma}: G_{F}\rightarrow \GL_{2}(R_{\Sigma})$ for the residual Galois representation. In particular the universal property of $R_{\Sigma}$ induces two maps $R_{\Sigma}\rightarrow R_{\emptyset}$ and  $R_{\Sigma}\rightarrow \Hecke_{\Sigma}$.

\begin{lm}
The map $R_{\Sigma}\rightarrow \Hecke_{\Sigma}$ is surjective.
\end{lm}

\begin{proof}
This is clear because:
\begin{itemize}
\item{If $v\nmid \frakn_{\Sigma}$, then $\Tr\rho_{R_{\Sigma}}(\Frob_{v})\rightarrow T_{v}\in \Hecke_{\Sigma}$.}
\item{If $v^{2}\mid \frakn$, then $U_{v}=0\in \Hecke_{\Sigma}$.}
\item{If $v\mid \frakn^{-}_{1}$, then $U_{v}=\pm 1 \in \Hecke_{\Sigma}$.}
\item{Finally if $v\| \frakn_{\Sigma}/\frakn^{-}_{1}$, $$\rho_{\Sigma}\mid_{G_{F_{v}}} =\pMX{\chi^{-1}_{v}\epsilon}{*}{0}{\chi_{v}}$$ such that $\chi_{v}(\Frob_{v})\rightarrow U_{v}\in \Hecke_{\Sigma}$.}
\end{itemize}
\end{proof}

Consider the composite morphism $R_{\Sigma}\rightarrow R_{\emptyset}\rightarrow \Hecke_{\emptyset} \MapR{\lambda_{\emptyset}}{} \cO_{f}$. Denote by $\wp_{\Sigma}$ the kernel of the composite morphism. Set $W_{\rho}= {\rm ad}^{0}\rho_{\lambda_{\emptyset}}\otimes E_{f}/\cO_{f}$. Define the subspace $W^{+}_{\rho}=\{\pMX{a}{b}{0}{-a}\}\subset W_{\rho}=\{\pMX{a}{b}{c}{-a}\}$. We will define the Selmer group $\Sel_{\Sigma}(W_{\rho})$ by choosing local conditions $\{\cL_{v}\}$ for various $v$.

\begin{itemize}
\item{For $v\|\frakn_{\Sigma}/\frakn_{\emptyset}$, we let $\cL_{v}=H^{1}(F_{v}, W^{+}_{\rho})$.}
\item{For $v\mid \frakn^{-}_{1}$, we choose a lift of Frobenius denoted by $\Frob_{v}$, and let $\cL_{v}=\ker\{H^{1}(G_{F_{v}}, W_{\rho})\rightarrow H^{1}(\<\Frob_{v}\>, W_{\rho} )\}$.}
\item{For $v$ at other places, $\cL_{v}=H^{1}_{f}(G_{F_{v}}, W_{\rho})$ for the usual local Bloch-Kato group {\cf \eqref{BK}}.}
\end{itemize}

Then the adjoint Selmer group is defined by 

$$\Sel_{\Sigma}(W_{\rho})=\ker\{H^{1}(F, W_{\rho})\rightarrow \prod_{v}H^{1}(F_{v}, W_{\rho})/\cL_{v}\}.$$
The Selmer group controls the tangent space of the deformation ring $R_{\Sigma}$ in the following manner:
\begin{equation}\label{tangent}
\Hom(\wp_{\Sigma}/\wp_{\Sigma}^{2}, E_{f}/\cO_{f})\isomor \Sel_{\Sigma}(W_{\rho}).
\end{equation}
We denote by $\Sigma^{(2)}$ the subset of prime factors $v$ of $\frakn/\frakn_{\emptyset}$ with $m_{v}=2$. Assume $\Sigma^{(2)}\subset \Sigma$ we have the following divisibility:

\begin{lm}\label{compare}
$\#\Sel_{\Sigma}(W_{\rho})/\#\Sel_{\Sigma^{(2)}}(W_{\rho})\mid \prod_{v\| \frakn_{\Sigma}/\frakn_{\emptyset}}\#H^{1}(I_{F_{v}}, W^{+}_{\rho})^{\<\Frob_{v}\>}.$
\end{lm}

\begin{proof}
Put $W^{-}_{\rho}=W_{\rho}/W^{+}_{\rho}$ and let $v\mid \frakn_{\Sigma}/\frakn_{\emptyset}$. Observe that we have the following commutative diagram:

$\begin{CD}
0@>>> H^{1}(G_{F_{v}}/I_{F_{v}}, W_{\rho})  @>>> H^{1}(G_{F_{v}}, W_{\rho})@>\pi>>H^{1}(I_{F_{v}}, W_{\rho})^{\<\Frob_{v}\>}@>>>0\\
    @. @VVV     @V\pi_{1}VV       @V\pi_{2}VV @.\\
0@>>> H^{1}(G_{F_{v}}/I_{F_{v}}, W^{-}_{\rho} ) @>>> H^{1}(G_{F_{v}}, W^{-}_{\rho})@>>> H^{1}(I_{F_{v}}, W^{-}_{\rho})^{\<\Frob_{v}\>}@>>>0.\\
\end{CD}$

By $\nimpr$, we know $\Nm{v}\not\equiv 1\pmod{p}$ and hence $H^{0}(G_{F_{v}}/I_{F_{v}}, W^{-}_{\rho})$, $H^{0}(G_{F_{v}}, W^{-}_{\rho})$ and $H^{1}(G_{F_{v}}/I_{F_{v}}, W^{-}_{\rho})$ are all zero.  It follows then from the snake lemma that:

\begin{itemize}
\item{$\ker(\pi_{1})=\cL_{v}=H^{1}(G_{F_{v}}, W^{+}_{\rho}).$}
\item{$H^{1}(G_{F_{v}}/I_{F_{v}}, W^{+}_{\rho})=H^{1}(G_{F_{v}}/I_{F_{v}}, W_{\rho}).$}
\item{$H^{1}(I_{F_{v}}, W^{+}_{\rho})^{\<\Frob_{v}\>}=H^{1}(G_{F_{v}}, W^{+}_{\rho})/ H^{1}(G_{F_{v}}/I_{F_{v}}, W^{+}_{\rho})=\cL_{v}/ H^{1}(G_{F_{v}}/I_{F_{v}}, W_{\rho}^{+}).$}
\end{itemize}

Then our claim follows from the tautological exact sequence:

$0\rightarrow \Sel_{\Sigma^{(2)}}(W_{\rho})\rightarrow \Sel_{\Sigma}(W_{\rho})\rightarrow \prod_{v\| \frakn_{\Sigma}/\frakn_{\emptyset}}\cL_{v}/ H^{1}(G_{F_{v}}/I_{F_{v}}, W_{\rho}^{+})$.

\end{proof}

As a corollary we obtain
 
\begin{cor}\label{tang_compare}
$\#(\wp_{\Sigma}/\wp^{2}_{\Sigma})\mid \#(\wp_{\Sigma^{(2)}}/\wp^{2}_{\Sigma^{(2)}})\prod_{v\mid \frakn/\frakn_{\emptyset}}\#(\cO_{f}/(\lambda^{2}_{\emptyset}(u_{\emptyset,v})-1)\cO_{f})$.
\end{cor}

\begin{proof}
By a standard calculation using the Euler characteristic formula:

$
\begin{aligned}
\#H^{1}(I_{F_{v}}, W^{+}_{\rho})^{\<\Frob_{v}\>}&=\#H^{1}(G_{F_{v}}, W^{+}_{\rho})/\#H^{1}(G_{F_{v}}/I_{F_{v}}, (W^{+}_{\rho})^{I_{F_{v}}})\\
&=\#H^{1}(G_{F_{v}}, W^{+}_{\rho})/\#H^{0}(G_{F_{v}}, W^{+}_{\rho})\\
&=\#H^{0}(G_{F_{v}}, W^{+}_{\rho}(1)).\\
\end{aligned}
$

A direct calculation shows that $$\#H^{0}(G_{F_{v}}, W^{+}_{\rho}(1))=\#\cO_{f}/ (\Nm{v}-1)(\lambda_{\emptyset}^{2}(u_{\emptyset,v})-1)\cO_{f}.$$  This calculation combined with \lmref{compare} and \eqref{tangent} conclude this corollary.
\end{proof}

Now we can prove the main result of this chapter. 
\begin{thm}\label{TW}
Assume $(\CR)$ and $\nimpr$ hold, the module $S_{\Sigma}$ is free over $\Hecke_{\Sigma}$ of rank 1.
\end{thm}

\begin{proof}
By \cite[Theorem 3.2]{Taylor:Ihara_avoidance}, we have $\#\wp_{\Sigma^{(2)}}/\wp^{2}_{\Sigma^{(2)}}\mid \#C(\frakn_{\Sigma^{(2)}}).$ Since we know $\#C(\frakn_{\Sigma})=\#C(\frakn_{\Sigma^{(2)}})\prod_{v\|\frakn_{\Sigma}/\frakn_{\Sigma^{(2)}}}\#\cO_{f}/(\lambda^{2}_{\empty}(u_{\emptyset,v})-1)\cO_{f}$ by \lmref{cong_module} and $\#(\wp_{\Sigma}/\wp^{2}_{\Sigma})\mid \#(\wp_{\Sigma^{(2)}}/\wp^{2}_{\Sigma^{(2)}})\prod_{v\mid \frakn/\frakn_{\emptyset}}\#(\cO_{f}/(\lambda^{2}_{\emptyset}(u_{\emptyset,v})-1)\cO_{f})$ by \corref{tang_compare}, we have  $\#\wp_{\Sigma}/\wp^{2}_{\Sigma}\mid \#C(\frakn_{\Sigma})$. Then the desired multiplity one result follows from \cite[Theorem 2.4]{Diamond:TW}.
\end{proof}

\section{$\mu$-invariant and $\Phi$-functor}

Recall we have proved the one divisibility result towards the anticyclotomic Iwasawa main conjecture in \thmref{main_weak} under a slightly stronger condition $\nmin$ than $\nimpr$. In this section, we relax $\nmin$ and $\nimpr$. The reason for assuming a stronger condition in \thmref{main_weak} is that the Euler system method we employ relies on the freeness result \propref{free}. In order to relax the condition $\nmin$, we need to recover the freeness of the Selmer group. A closer look at the Euler system argument in Section $7.3$  shows that we in fact only need the freeness of a larger Selmer group, namely $\widehat{\Sel}_{\Delta}^{S\nplus}(K_{\infty}, T_{f,n}).$
\subsection{Greenberg-Vastal method}
To relax the assumption $\nmin$ to $\nimpr$, we study the Iwasawa $\mu$-invariant. And the following method is based on the unpublished note by Pollack-Weston communicated to us by Chan-Ho Kim and their preprint \cite{PW_free}. The key observation is the following theorem.

\begin{thm}
Under the condition $(\CR)$ and $\nimpr$, $\Sel_{\nminus}(K_{\infty}, A_{f})$ is $\Lambda$ cotorsion and the algebraic $\mu$-invariant $$\mu(\dual{\Sel_{\nminus}(K_{\infty}, A_{f})})$$ vanishes.
\end{thm}

\begin{proof}
Let $g$ be a Hilbert modular form congruent to $f$ modulo $\uf$ which satisfies $(\CR)$ and $\nimpr$. Notice the existence of such $g$ is guaranteed by level-lowering \cite{Javis:Level}.
Then the theorem for $g$ follows immediately from \thmref{main_weak} and the main result of \cite[Theorem B]{Hung:thesis} concerning the analytic $\mu$-invariant. Since all primes dividing $\nplus$ splits in $K$, we can apply the same method as \cite[Cor 2.3]{Greenberg&Vatsal} and conclude that $\Sel_{\nminus}^{\nplus}(K_{\infty}, A_{g})$ is $\Lambda$-cotorsion with vanishing $\mu$-invariant. Since $g$ is obtained from $f$ by level lowering, an adaption of the method of \cite[Prop 2.8]{Greenberg&Vatsal} shows that $\Sel_{\nminus}^{\nplus}(K_{\infty}, A_{f})[\uf]=\Sel_{\nminus}^{\nplus}(K_{\infty}, A_{g})[\uf]$. The point is that when we relaxed the local condition at $\nplus$ for our Selmer groups$\Sel_{\nminus}^{\nplus}(K_{\infty}, A_{f})[\uf]$ (\resp $\Sel_{\nminus}^{\nplus}(K_{\infty}, A_{g})[\uf]$) only depends on the residual representation $A_{f,1}=A_{g,1}$. Since $\Sel_{\nminus}^{\nplus}(K_{\infty}, A_{g})$ is known to be cotorsion with vanishing $\mu$-invariant, $\Sel_{\nminus}^{\nplus}(K_{\infty}, A_{g})[\uf]$ is finite. This implies in turn that $\Sel_{\nminus}^{\nplus}(K_{\infty}, A_{f})[\uf]$ is finite. Then we can conclude that $\Sel_{\nminus}^{\nplus}(K_{\infty}, A_{f})$ is cotorsion with vanishing $\mu$-invariant. Now the same reasoning as in  \cite[Cor 2.3]{Greenberg&Vatsal} allows us to conclude that $\Sel_{\nminus}(K_{\infty}, A_{f})$ has the same properties.
\end{proof}

For any place $v$ of $F$, we put $\cH_{v}=\dirlim_{m} \prod_{\omega\mid v} H^{1}(K_{m,\omega}, A_{f})$. 

\begin{lm}\label{rightexact}
Let $S$ be finite set of places of $F$ that are away from $\frakn$, then we have an exact sequence:
$$0\rightarrow \Sel_{\nminus}^{\nplus}(K_{\infty}, A_{f}) \rightarrow \Sel_{\nminus}^{S\nplus}(K_{\infty}, A_{f})\rightarrow \prod_{v\in S} \cH_{v}\rightarrow 0.$$
\end{lm}

\begin{proof}
This follows from the definition except for the surjectivity of the last map. The same method as in \cite[Proposition A.2]{Pollack_Weston:AMU} applies showing the surjectivity.
\end{proof}

\begin{lm}\label{cotor}
Let $\frakl$ be an $n$-admissible prime for $f$, then we have $\cH_{\frakl}\cong \dual{(\Lambda/\uf^{t})}$ for some $t\geq n$.
\end{lm}
\begin{proof}
First notice that as $\frakl$ splits in $K_{\infty}$, for each $m$ we have $$H^{1}(K_{m, \frakl}, A_{f})=H^{1}(K_{\frakl}, A_{f})\otimes \cO_{f}[\Gal(K_{m}/K)].$$ 
We consider the inflation-restriction exact sequence $$0\rightarrow H^{1}(K^{ur}_{\frakl}/K_{\frakl}, A_{f})\rightarrow H^{1}(K_{\frakl}, A_{f})\rightarrow H^{1}(I_{\frakl}, A_{f})^{\Gal(K^{ur}_{\frakl}/K)}\rightarrow 0.$$

First of all, $H^{1}(K^{ur}_{\frakl}/K_{\frakl}, A_{f})=A_{f}/(\gamma-1)A_{f}$ for a topological generator of $\Gal(K^{ur}_{\frakl}/K_{\frakl})$. This shows that $H^{1}(K^{ur}_{\frakl}/K_{\frakl}, A_{f})=0$. On the other hand $H^{1}(I_{\frakl}, A_{f})\cong \Hom(I_{\frakl}, A_{f})=\Hom(\Z_{p}(1), A_{f})$ and thus $H^{1}(I_{\frakl}, A_{f})^{\Gal(K^{ur}_{\frakl}/K_{\frakl})}=\Hom(I_{\frakl}, A_{f})^{\Gal(K^{ur}_{\frakl}/K_{\frakl})}$. By the $n$-admissibility of $\frakl$, the two eigenvalues $\alpha^{2}$ and $\beta^{2}$ of $\Frob_{\frakl}$ on $A_{f}$ has the property that $\alpha\equiv \epsilon_{\frakl}\pmod{\uf^{n}}$ and $\beta\equiv\epsilon_{\frakl}\Nm{\frakl}\pmod{\uf^{n}}$. Note also that $\Frob_{\frakl}$ acts by $\Nm{\frakl}^{2}$ on $\Z_{p}(1)$, it follows therefore $\Hom(I_{\frakl}, A_{f})^{\Gal(K^{ur}_{\frakl}/K_{\frakl})}=\cO_{f}/\uf^{t}$ for the largest power $t$ such that $\beta\equiv \Nm{\frakl}\pmod{\uf^{t}}$. Then it is clear that by taking the direct limit with respect to $m$, we obtain the desired result.
\end{proof}

\begin{cor}\label{sel(S)}
$\dual{\Sel_{\nminus}^{S\nplus}(K_{\infty}, A_{f})}$ is pseudo-isomorphic to a $\Lambda$-module of the form:
$$(\bigoplus_{i=1}^{\#S}\Lambda/\uf^{t_{i}})\times Y$$
with each $t_{i}\geq n$ and $Y$ a torsion $\Lambda$-module with $\mu(Y)=0$.
\end{cor}

\begin{proof}
This is proved by combining the \lmref{rightexact} and \lmref{cotor}.
\end{proof}

\begin{remark}
Note that the above corollary implies in particular that $\mu(\dual{\Sel_{\nminus}^{S\nplus}(K_{\infty}, A_{f})})\geq n\#S.$
\end{remark}

\subsection{$\Phi$-functor}

For each $n\geq 1$, Pollack-Weston defined the following functor:
\begin{equation}
\Phi:  \{\text{cofinitely generated  $\Lambda$-module} \}  \rightarrow  \{\text{finitely generated $\Lambda/\uf^{n}$-module} \}
\end{equation}
where $\Phi_{n}(S)=\prolim_{m} S[\uf^{n}]^{\Gamma^{m}}$ for the group $\Gamma^{m}=\Gal(K_{\infty}/K_{m})$. And the transition map is given by the trace map.

The following lemma collects some properties of this functor.

\begin{lm}
\begin{enumerate}
\item{The functor $\Phi_{n}$ is covariant and left exact.}
\item{$\Phi_{n}(S)=\Phi_{n}(S[\uf^{n}])$.}
\item{If $S$ is finite, then $\Phi_{n}(S)=0$.}
\end{enumerate}
\end{lm}
\begin{proof}
Only the third point needs some explanation as the others follow from the definitions. For the third point, notice that $\Gamma^{m}$ fixes $S$ for $m$ sufficiently large. Then the transition map becomes simply multiplication by $p$. The third point follows.
\end{proof}

\begin{lm}
Let $0\rightarrow A\rightarrow B\rightarrow C\rightarrow 0$ be an exact sequence of cofinitely generated $\Lambda$-modules.
\begin{enumerate}
\item{If $C$ is finite, then $\Phi_{n}(A)=\Phi_{n}(B)$.}
\item{If $A$ is finite, then $\Phi_{n}(B)=\Phi_{n}(C)$.}
\end{enumerate}
\end{lm}

\begin{proof}
The first part follows immediately from the previous lemma and the fact that $\Phi$ is left exact.

For the second point, we first take the $\uf^{n}$ torsion of the exact sequence to obtain:
$$0\rightarrow A[\uf^{n}]\rightarrow B[\uf^{n}]\rightarrow C[\uf^{n}]\rightarrow A/\uf^{n}.$$
Then for some finite $H$, we have
$$0\rightarrow A[\uf^{n}]\rightarrow B[\uf^{n}]\rightarrow C[\uf^{n}]\rightarrow H\rightarrow 0.$$
We split this exact sequence into two short exact sequences:
$$0\rightarrow A[\uf^{n}]\rightarrow B[\uf^{n}]\rightarrow C^{\prime}\rightarrow 0 $$
and
$$0\rightarrow C^{\prime} \rightarrow C[\uf^{n}]\rightarrow H\rightarrow 0.$$
By applying the previous lemma, we have $\Phi_{n}(C^{\prime})\cong\Phi_{n}(C[\uf^{n}])\cong \Phi_{n}(C)$. For the first exact sequence, we take $\Gamma^{m}$ invariant to get $$0\rightarrow A[\uf^{n}]^{\Gamma^{m}}\rightarrow B[\uf^{n}]^{\Gamma^{m}}\rightarrow (C^{\prime})^{\Gamma^{m}}\rightarrow H^{\prime}\rightarrow 0$$ for some finite $H^{\prime}$. Again by breaking the this four term exact sequence into two short exact sequence, we have
$$0\rightarrow A[\uf^{n}]^{\Gamma^{m}}\rightarrow B[\uf^{n}]^{\Gamma^{m}}\rightarrow C^{\prime}_{m}\rightarrow 0$$
and
$$0\rightarrow C^{\prime}_{m}\rightarrow (C^{\prime})^{\Gamma^{m}}\rightarrow H^{\prime}\rightarrow 0$$
where $C^{\prime}_{m}$ is again  finite. Taking inverse limit with respect to $n$ yields:
$$0\rightarrow \Phi_{n}(A)\rightarrow \Phi_{n}(B)\rightarrow \prolim_{m} C^{\prime}_{m}\rightarrow 0$$ where the left exactness is guaranteed by the finiteness of $A$ and the Mittag-Leffler condition. Then we get $\Phi_{n}(B)=\prolim_{m}C^{\prime}_{m}$. The same reasoning applied to the second exact sequence yields $\prolim C_{m}^{\prime}\cong \Phi_{n}(C^{\prime})$. Therefore $\Phi_{n}(B)=\Phi_{n}(C^{\prime})=\Phi_{n}(C)$. 
\end{proof}

\begin{cor}
If $A$ and $B$ are two cofinitely generated $\Lambda$-modules which are pseudo-isomorphic, then $\Phi_{n}(A)\cong \Phi_{n}(B)$
\end{cor}
\begin{proof}
In light of the definition of pseudo-isomorphism, this follows immediately from the two previous lemmas.
\end{proof}

\begin{prop}\label{muzero}
For a fintely generated cotorsion $\Lambda$-module $Y$ with vanishing $\mu$-invariant, $\Phi_{n}(\dual{Y})=0$ for all $n\geq 1$.
\end{prop}
\begin{proof}
By the structure theorem of finitely generated cotorsion $\Lambda$-module $Y$, we can assume that $Y=\Lambda/f\Lambda$ for some $f$ coprime to $p$.  Then we have $\Phi_{n}(\dual{Y})\cong\Phi_{n}(\dual{Y}[\uf^{n}])\cong\Phi_{n}(\dual{(Y/\uf^{n})})$. As $Y/\uf^{n}=\Lambda/(f,\uf^{n})$ is finite, it follows that $\Phi_{n}(\dual{Y})=0$
\end{proof}

\begin{prop}\label{munonzero}
If $Y=\Lambda/\uf^{t}$ with $t\geq n$, then $\Phi_{n}(\dual{Y})\cong \Lambda/\uf^{n}$.
\end{prop}
\begin{proof}
By definition, we have 

$
\begin{aligned}
\Phi_{n}(\dual{(\Lambda/\uf^{t}\Lambda)})&=\prolim_{n}\dual{(\Lambda/\uf^{t})}[\uf^{n}]^{\Gamma^{m}}
\cong \prolim_{n}\dual{(\Lambda/\uf^{n})_{\Gamma^{m}}}\\
&\cong \prolim_{n}\dual{\cO_{f,n}[\Gamma_{m}]}.\\
\end{aligned}
$

By the Gorensteiness of $\dual{\cO_{f,n}[\Gamma_{m}]}$, $\dual{\cO_{f,n}[\Gamma_{m}]}\cong \cO_{f,n}[\Gamma_{m}]$ and the transition maps with respect to $n$ are compatible with this isomorphism. It follows the that 

$
\begin{aligned}
\Phi_{n}(\dual{(\Lambda/\uf^{t}\Lambda)})&\cong \prolim_{n}\dual{\cO_{f,n}[\Gamma_{m}]}
\cong \prolim_{n}\cO_{f,n}[\Gamma_{m}]\\
&\cong \Lambda/\uf^{n}.\\
\end{aligned}
$
\end{proof}

\subsection{Finish of the proof} 

\begin{prop}\label{switch}
We have $\Phi_{n}(\Sel_{\nminus}^{S\nplus}(K_{\infty}, A_{f}))=\widehat{\Sel}_{\nminus}^{S\nplus}(K_{\infty}, T_{f,n})$.
\end{prop}

\begin{proof}
We evaluate the functor $\Phi_{n}$ at $\Sel_{\nminus}^{S\nplus}(K_{\infty}, A_{f})$:

$$
\begin{aligned}
\Phi_{n}(\Sel_{\nminus}^{S\nplus}(K_{\infty}, A_{f}))&\cong \prolim_{n}\Sel_{\nminus}^{S\nplus}(K_{\infty}, A_{f})[\uf^{n}]^{\Gamma^{m}}\cong \widehat{\Sel}_{\nminus}^{S\nplus}(K_{\infty},T_{f,n})\\
\end{aligned}
$$
where the last isomorphism is given by the control theorem treated in \propref{control}.

\end{proof}
\begin{thm}\label{newfree}
The compact Selmer group $\widehat{\Sel}_{\nminus}^{S\nplus}(K_{\infty}, T_{f,n})$ is free of rank $\#S$ over $\Lambda/\uf^{n}$.
\end{thm}

\begin{remark}
This theorem improves \propref{free}.
\end{remark}

\begin{proof}
By \corref{sel(S)}, we know $\dual{\Sel_{\nminus}^{S\nplus}(K_{\infty}, A_{f})}$ is pseudo-isomorphic to 
$$(\bigoplus_{i=1}^{\#S}\Lambda/\uf^{t_{i}})\times Y$$ for some $Y$ with vanishing $\mu$-invariant. Combining \propref{muzero} and \propref{munonzero}, we have
$$\Phi_{n}(\Sel_{\nminus}^{S\nplus}(K_{\infty}, A_{f}))\cong (\Lambda/\uf^{n})^{\#S}.$$
Hence $\widehat{\Sel}_{\nminus}^{S\nplus}(K_{\infty}, T_{f,n})\cong (\Lambda/\uf^{n})^{\#S}$ by \propref{switch}.
\end{proof}

Now \thmref{main_weak} can be improved by assuming the desired assumption $(\CR)$ and $\nimpr$. Note it suffices to use a larger Selmer group $\widehat{\Sel}^{S\nplus}_{\Delta}(K_{\infty}, T_{f,n})$ in the proof of \thmref{induction}. The proof of the following freeness theorem for this Selmer group  will provide all the missing ingredients to prove our main result \thmref{main1} assuming $(\CR)$ and $\nimpr$ \cf \corref{control2} . 
\begin{lm}\label{new_free}
Let $\Delta$ be a square free product of  $n$-admissible primes in $F$ with $\nminus$ and $S$ be a square free product of primes in $F$ away from $\frakn p$. Assume only $(\CR)$, then
\begin{enumerate}
\item $\widehat{\Sel}_{\Delta}^{\nplus S}(K_{\infty}, T_{f,n})/\frakm_{\Lambda}=\Sel_{\Delta}^{\nplus S}(K, T_{f,1})$.
\item $\widehat{\Sel}_{\Delta}^{\nplus S}(K_{\infty}, T_{f,n})$ is free over $\Lambda/\uf^{n}$.
\end{enumerate} 
\end{lm}

\begin{proof}
The proof of (1) \ie the control theorem follows from \corref{control2}  since we are relaxing conditions on $\nplus$.

As for $(2)$, we have proved in \thmref{newfree} that $\widehat{\Sel}_{\nminus}^{\nplus S}(K_{\infty}, T_{f,n})$ is free. To prove  $\widehat{\Sel}_{\Delta}^{\nplus S}(K_{\infty}, T_{f,n})$ is free, it suffices to consider the case when $\Delta=\frakl \nminus$ for a single $n$-admissible prime $\frakl$. For this we consider the following exact sequence
$$0\rightarrow \widehat{\Sel}^{S}_{\Delta}(K_{\infty}, T_{f,n})\rightarrow \widehat{\Sel}^{S\frakl}_{\nminus}(K_{\infty}, T_{f,n})\rightarrow \widehat{H}^{1}_{\fin}(K_{\infty,\frakl}, T_{f,n})\rightarrow 0.$$
This exact sequence results from the decomposition $$\widehat{H}^{1}(K_{\infty,\frakl}, T_{f,n})=\widehat{H}^{1}_{\fin}(K_{\infty,\frakl}, T_{f,n})\oplus \widehat{H}^{1}_{\ord}(K_{\infty,\frakl}, T_{f,n})$$ in \lmref{admi_coho}. The freeness of  $\widehat{\Sel}^{S}_{\Delta}(K_{\infty}, T_{f,n})$ follows from the freeness of $\widehat{\Sel}^{S\frakl}_{\nminus}(K_{\infty}, T_{f,n})$ and that of $\widehat{H}^{1}_{\fin}(K_{\infty}, T_{f,n})$. 
\end{proof}

\bibliographystyle{alpha}

\bibliography{hainingbib.bib}

\end{document}